\documentclass[10pt]{article}
\usepackage[english]{babel}
\usepackage[letterpaper,top=2cm,bottom=2cm,left=2.5cm,right=2.5cm,marginparwidth=1.75cm]{geometry}
\usepackage{amsmath}
\usepackage{dsfont}
\usepackage{amsthm}
\usepackage{amssymb}
\usepackage{graphicx,xcolor,float}
\usepackage{subfigure}
\usepackage{epstopdf}
\usepackage[colorlinks=true, allcolors=blue]{hyperref}
\usepackage{cite}
\usepackage{tikz}
\usepackage{tikz-cd}
\usepackage{wrapfig,lipsum}  
\usepackage{xspace}
\usepackage{aligned-overset}
\usepackage{bbm}
\usepackage{booktabs}
\usepackage[utf8]{inputenc}

\usepackage{enumitem}

\usepackage{mathtools}
\mathtoolsset{showonlyrefs}

\DeclareMathOperator{\bp}{\mathbf{p}}
\DeclareMathOperator{\bP}{\mathbf{P}}

\DeclareMathOperator{\bX}{\mathbf{X}}
\DeclareMathOperator{\TV}{\mathrm{TV}}
\DeclareMathOperator{\Fluc}{\mathrm{Fluc}}
\DeclareMathOperator{\Prob}{\mathrm{Prob}}
\DeclareMathOperator{\tot}{\mathrm{tot}}
\DeclareMathOperator{\Var}{\mathrm{Var}}
\DeclareMathOperator{\rem}{\mathrm{rem}}
\newcommand{\term}[1]{\text{\tt{#1}}\xspace}
\newcommand{\Max}{\term{max}}

\newcommand{\lex}{\term{lex}}
\newcommand{\stable}{\term{stable}}
\newcommand{\Pre}{\term{Pre}_{\term{CL}}}
\newcommand{\CL}{\term{CL}}
\newcommand{\for}{\mathrm{forward}}
\newcommand{\back}{\mathrm{backward}}
\newcommand{\Blocks}{\mathsf{Blocks}}
\newcommand{\new}{\mathrm{new}}

\newcommand{\Unif}{\mathrm{Unif}}
\newcommand{\rmleft}{\mathrm{left}}
\newcommand{\rmright}{\mathrm{right}}
\newcommand{\ileft}{\mathrm{I}_\rmleft}
\newcommand{\iright}{\mathrm{I}_\rmright}
\newcommand{\al}{\mathcal{A}_{\rmleft}}
\newcommand{\ar}{\mathcal{A}_{\rmright}}
\newcommand{\tal}{\widetilde{\mathcal{A}}_{\rmleft}}
\newcommand{\tar}{\widetilde{\mathcal{A}}_{\rmright}}

\newcommand{\p}{p}

\newcommand{\bea}{\begin{eqnarray}}
\newcommand{\eea}{\end{eqnarray}}
\newcommand{\<}{\langle}
\renewcommand{\>}{\rangle}

\newcommand\eg{{\text{\eg~}}}

\def\Unif{{\sf Unif}}

\def\eps{{\varepsilon}}

\def\bP{{\boldsymbol{P}}}

\def\cF{{\mathcal F}}

\def\cT{{\mathcal T}}

\def\s{\sigma}
\def\Z{{\mathbb Z}}

\def\vbn{\vec \bn}

\def\Z{{\mathbb Z}}

\def\mix{{\mathsf{mix}}}

\def\bX{\boldsymbol{X}}

\def\<{\langle}
\def\>{\rangle}

\def\P{\mathbb{P}}

\def\cD{{\cal D}}

\def\b0{{\boldsymbol{0}}}

\def\Bin{{\sf Bin}}
\def\Hyp{{\sf Hyp}}
\def\Mult{{\sf Mult}}
\def\Var{{\rm Var}}

\DeclareMathOperator*{\plim}{p-lim}

\def\cD{{{\mathcal D}}}

\def\cI{{\mathcal I}}
\def\cS{{\mathcal S}}
\def\bn{{\boldsymbol n}}

\renewcommand{\b}{\mathbf{b}}

\def\lt{\left}
\def\rt{\right}

\def\eps{\varepsilon}

\def\bbE{{\mathbb{E}}}

\def\bbN{{\mathbb{N}}}

\def\bbR{{\mathbb{R}}}

\def\cF{{\mathcal{F}}}

\def\tot{{\text{tot}}}

\def\bp{{\mathbf{p}}}

\def\bP{\mathbf{P}}

\def\TV{{\mathrm{TV}}}

\DeclareMathOperator*{\E}{\bbE}

\newcommand{\oC}{\overline{C}}

	\newtheorem{thm}{Theorem}[section]
	\newtheorem{prop}[thm]{Proposition}
	
	\newtheorem{lem}[thm]{Lemma}
	\newtheorem{defi}[thm]{Definition}
	\newtheorem{rmk}[thm]{Remark}
	\numberwithin{equation}{section}

\title{
Universality of Cutoff for Riffle Shuffling
}
\author{Mark Sellke \and Jialu Shi\and Jiamin Wang}
\date{}

\begin{document}

	\maketitle

	\begin{abstract}
	\noindent
	A Gilbert--Shannon--Reeds (GSR) shuffle is performed on a deck of $N$ cards by cutting the top $n\sim\Bin(N,1/2)$ cards and interleaving the two resulting piles uniformly at random. 
	The celebrated ``Seven shuffles suffice'' theorem of Bayer and Diaconis in \cite{bayer1992trailing} established cutoff for this Markov chain: to leading order, total variation mixing occurs after precisely $\frac{3}{2}\log_2 N$ shuffles.
	Later work of Lalley \cite{lalley2000rate} and the first author \cite{mark2022cutoff} extended this result to asymmetric binomial cuts $n\sim\Bin(N,p)$ for all $p\in (0,1)$.
    These results relied heavily on the binomial condition and many natural chains were left open, including uniformly random cuts and exact bisections.

	We establish cutoff for riffle shuffles with general pile size distribution.
	Namely, suppose the cut sizes $(n^{(t)})_{t\geq 1}$ are IID and the convergence in distribution $n^{(t)}/N \stackrel{d}{\to} \mu$ holds for some probability measure $\mu$ on the interval $[0,1]$. 
	Then the mixing time $t_{\mix}$ satisfies $t_{\mix}/\log N\to \oC_{\mu}$ for an explicit constant $\oC_{\mu}$.
    The same result holds for any (deterministic or random) sequence of pile sizes with empirical distribution converging to $\mu$ on all macroscopic time intervals (of length $\Omega(\log N)$).
    It also extends to multi-partite shuffles where the deck is cut into more than $2$ piles in each step.
    Qualitatively, we find that the ``cold spot'' phenomenon identified by Lalley characterizes the mixing time of riffle shuffling in great generality.
    \end{abstract}

\section{Introduction}

    Riffle shuffling is among the most common methods to randomize a deck of cards. 
	The best-known probabilistic model is the Gilbert--Shannon--Reeds (GSR) shuffle \cite{gilbert1955theory,reeds1981}, which proceeds in the following two steps for an $N$-card deck:
    \begin{enumerate}[label = (\Roman*)]
    	\item 
    	\label{it:GSR-binom}
    	Remove the top $n$ cards from the top of the deck, where $n\sim \Bin(N,1/2)$ is a binomial variable.
    	\item 
    	\label{it:GSR-interleave}
    	Combine the two resulting piles via a uniformly random interleaving among all $\binom{N}{n}$ possibilities.
    \end{enumerate}

    The mixing time of the GSR shuffle was pinpointed in \cite{bayer1992trailing} following \cite[Example~4.17]{aldous1983random}: $\big(\frac{3}{2}\log_2N\big)\pm O(1)$ shuffles are necessary and sufficient for mixing in total variation.
    The analysis relies on the study of \emph{rising sequences}, which yield a sufficient statistic for mixing when (and only when) $n\sim \Bin(N,1/2)$.   
    Later \cite{lalley2000rate,mark2022cutoff} extended this result to obtain cutoff for \emph{asymmetric} binomial cuts, where $n\sim \Bin(N,p)$ for $p\neq 1/2$ and the interleaving step~\ref{it:GSR-interleave} above is unchanged.
    Although rising sequences cannot be used, the binomial distribution of the cuts remains crucial in the analysis.

    These results left open many natural questions. For example:
    \begin{enumerate}[label=Q\arabic*]
    	\item Does cutoff occur at the same time whenever the deck is cut ``about in half'' at each step? 
    	For example one may consider exact bisections $n=N/2$, or larger fluctuations $n=N/2 \pm o(N)$.
    	\item Are the existence and location of cutoff robust to a small fraction of ``erroneous'' cuts?
    	\item Does the cutoff phenomenon persist when $n/N$ does not concentrate, e.g.\ $n\sim \text{Unif}(\{0,1,2,\dots,N\})$?
    \end{enumerate}

    We answer all of these questions, showing that the delicate first condition \ref{it:GSR-binom} on the GSR shuffle is wholly unnecessary for the analysis of cutoff.
    Namely, let $\mu$ be a probability distribution on the interval $[0,1]$ satisfying the non-degeneracy condition $\mu(\{0,1\})<1$.
    Suppose that the (deterministic or random) sequence of cut sizes $(n^{(1)},n^{(2)},\dots,n^{(K)})$ is such that the time-parametrized empirical distribution
    \[
    \plim_{N\to \infty}\frac{1}{K}\sum_{1\leq t\leq K} \delta_{(n^{(t)}/N,t/K)}
    \]
	of cut sizes converges to a product measure $\mu\times  \text{Leb}([0,1])$ as $N\to\infty$ in probability.
	Then we show cutoff occurs after $\oC_{\mu}\log N$ shuffles where $\oC_{\mu}$ is an explicit $\mu$-dependent positive constant.
	Namely $K\geq \oC_{\mu}(1+\eps)\log N$ shuffles suffice for total variation mixing while $K\leq \oC_{\mu}(1-\eps)\log N$ do not. 
	This encompasses all of the settings highlighted above. 
	As explained in more detail in the next subsection, our result extends also to the multi-partite riffle shuffles introduced in \cite{diaconis1992analysis} where the deck is cut into $k$ piles in each step.

	\subsection{Main results}

	As most of our results hold conditional on the pile sizes, we first consider riffle shuffles with deterministic pile sizes. To define a shuffle of a deck $N$ with pile sizes $(n_0,\cdots,n_{k-1})$, where $\sum_{0\leq i\leq k-1}n_i=N$, one splits the $N$ cards into $k$ piles by taking the top $n_0$ cards off the top to form the first pile, the next $n_1$ cards to form the second pile, and so on.
	One then interleaves the piles by choosing one of the $\binom{N}{n_0,\dots, n_{k-1}}$ possible interleavings uniformly at random.
	Equivalently at each time, if the current remaining pile sizes are $A_0,\dots,A_{k-1}$, we drop a card from one of these piles independently from previous shuffles, such that the next card is dropped from pile $i$ with probability
	\begin{equation}
		\label{eq:interleave}
		\frac{A_i}{A_0+A_1+\dots+A_{k-1}}.
	\end{equation}
	We call the above process a \emph{riffle shuffle with pile sizes} $\vec\bn=(n_0,\cdots,n_{k-1})$, which defines a Markov transition matrix on the $N$-permutation group $\mathsf S_N$:
	\begin{equation*}
		\bP(\vec\bn):\Prob(\mathsf S_N)\to \Prob(\mathsf S_N),
	\end{equation*}
	where we write $\Prob(\mathsf S_N)$ as the space of probability measures on $\mathsf S_N$.

	We study the following shuffling process. Fix $K=K(N)$ the number of total shuffles, satisfying that $\lim_{N\to \infty}K(N)=\infty$, $k$ the number of piles, and non-negative integers $\vec\bn^{(t)}=(n_0^{(t)},\cdots,n_{k-1}^{(t)})$ for each $1\leq t\leq K$, with $\sum_{i=0}^{k-1} n_i^{(t)}=N$ for each $t$. 
	Then at time $t$, we perform a shuffle with pile sizes $\vec\bn^{(t)}=(n_0^{(t)},\cdots,n_{k-1}^{(t)})$. 

    To deal with general shuffling processes, fix a probability distribution $\mu$ on the simplex
    \begin{equation}\label{eq:def-Dk}
        \mathcal D_k=\{(p_0,\cdots,p_{k-1}):0\leq p_i\leq 1, p_0+\cdots+p_{k-1}=1\},
    \end{equation}
    with vertex set
    \begin{equation}\label{eq:def-Vk}
    V_k=\{(p_0,\cdots,p_{k-1})\in \mathcal D_k: p_i=1 \mbox{ for some }0\leq i\leq k-1\}.
    \end{equation}
    In a probability space $(\Omega,\mathcal F,\mathbb P)$, we say a shuffling process with pile sizes $(\bX^{(t)})_{t\in \mathbb Z^+}$ is \textbf{randomized-$\mu$-like} if the following convergence in probability holds in the space of probability measures on $\bbR^{k+1}$:
    \begin{equation}\label{def:random-mu-like}
        \plim_{N\to \infty}
        \frac{1}{K(N)}
        \sum_{1\leq t\leq K} \delta_{(\bX^{(t)}/N,t/K)}
        =
        \mu\times \text{Leb}([0,1]).
    \end{equation}
    Here $\delta_x$ denotes a point mass at $x\in \mathcal D_k$.
    
    For any positive integer $K$, the shuffle process with pile sizes $\vec \bn=(\vec\bn^{(t)})_{t\in \mathbb Z^+}$ until time $K$ defines a Markov transition matrix on $\mathsf S_N$:
        \begin{equation*}
		\bP(\vec\bn,K):\Prob(\mathsf S_N)\to \Prob(\mathsf S_N).
	\end{equation*}
    Equivalently, the transition matrix is defined by
    \[
    \bP(\vec\bn,K)=\bP(\vec\bn^{(K)})\bP(\vec\bn^{(K-1)})\cdots \bP(\vec\bn^{(1)}).
    \]
    Moreover, for shuffling process with randomized pile sizes $(\bX^{(t)})_{t\in \mathbb Z^+}$, let $\bP(\bX,K)=\bbE^{\vbn\sim\bX} \bP(\vbn,K)$ denote the distribution of the resulting $K$-times shuffle deck, averaged over the randomness of both the pile sizes and the interleavings.
    (We emphasize that both $\bX,K$ depend implicitly on $N$ here.)
	
    The stationary measure for these transition matrices is the uniform distribution on $\mathsf S_N$, which we denote by $\nu_{\Unif}$.
    Our interest will be in the rate of convergence to $\nu_{\Unif}$.
    We establish cutoff for all $\mu$-like riffle shuffles, at a universal time depending only on $\mu$.
    To state the result, we recall some constants defined similarly for $\bp$-shuffles in \cite{lalley2000rate,mark2022cutoff}. 
	For each $\bp\in \mathcal D_k$, with arbitrary tie-breaking, set $i_{\Max}=\arg\max_{i\in \{0,1,\dots,k-1\}}(p_i)$ and $p_{\Max}=p_{i_{\Max}}$.
    Define the functions 
	\begin{equation}
	\label{eq:phi-psi}
	\phi_{\bp}(x)=\sum_{i=0}^{k-1} p_i^x,\quad\quad \psi_{\bp}(x)=-\log\phi_{\bp}(x),\quad\quad 
    \psi_{\mu}(x)=\mathbb E_{\mu}\psi_{\bp}(x).
	\end{equation}
For each probability measure $\mu$ on $ \mathcal D_k$ with $\mu(V_k)<1$, define the positive constant $\theta_{\mu}$ by the identity $\psi_{\mu}(\theta_{\mu})=2\psi_{\mu}(2)$. 
Observe that $\phi_{\bp}$ is strictly monotone, hence so is $\psi_{\bp}$ and thus also $\psi_{\mu}$; in particular $\theta_{\mu}$ is uniquely determined.
Noting that  
    $\psi_{\bp}(3) \leq 2\psi_{\bp}(2) \leq \psi_{\bp}(4)$ holds for all $\bp \in \mathcal D_k$ (see \cite[Proof of Proposition~1.1]{mark2022cutoff}), we have
    \begin{equation} \label{eq:theta-bound}
        \theta_{\mu} \in [3,4).
    \end{equation} Then we define
	\begin{equation}
	\label{eq:C-p}
		C_{\mu}=\frac{3+\theta_{\mu}}{4\psi_{\mu}(2)}=\frac{3+\theta_{\mu}}{2\psi_{\mu}(\theta_{\mu})},
		\quad\quad
		\widetilde C_{\mu}=\frac{1}{\E_{\mu}\log(1/p_{\Max})},
		\quad\quad
		\overline{C}_{\mu}=\max(\widetilde C_{\mu},C_{\mu}).
	\end{equation}
	   We now present our first main result: randomized-$\mu$-like shuffles undergo (total variation) cutoff after $\overline{C}_{\mu}\log N$ steps.
     Note the case that $\mu$ is a delta mass at $\bp$ already includes multinomial shuffles $\mu_N^{(t)}\sim\Mult (N,\bp)$, recovering in particular the main result of \cite{mark2022cutoff}.
    
        \begin{thm}
        \label{thm:weak-convergence}
            Recalling \eqref{eq:def-Dk} and \eqref{eq:def-Vk}, let $\mu$ be a probability measure on $\cD_k$ with $\mu(V_k)<1$. 
            Then any sequence of randomized $\mu$-like shuffles undergoes total variation cutoff after $\overline{C}_{\mu}\log N$ steps.
            Equivalently, for any randomized $\mu$-like shuffle with pile sizes $\bX=(\bX^{(t)})_{t\in \mathbb Z^+}$ and any $\varepsilon>0$,
		\begin{align}
			\label{eq:random-LowerBound}
			\lim_{N\to\infty} d_{\TV}(\bP(\bX,\lfloor(1-\varepsilon)\overline{C}_{\mu}\log(N)\rfloor),\nu_{\Unif})&=1,\\
			\label{eq:random-UpperBound}
			\lim_{N\to\infty} d_{\TV}(\bP(\bX,\lfloor(1+\varepsilon)\overline{C}_{\mu}\log(N)\rfloor),\nu_{\Unif})&=0.
		\end{align}
        \end{thm}

    \begin{rmk}
    \label{rmk:endpoint-case}
        In the degenerate case $\mu(V_k)=1$, we may naturally define $\overline C_{\mu}=\infty$. Indeed as explained in  Remark \ref{rmk-all-vertex}, in this case one has $t_{\mix}\geq\omega(\log N)$ also from our proof of \eqref{eq:random-LowerBound}, i.e. for any fixed $C>0$:
        \begin{equation}\label{eq:all-vertex}
            \lim_{N\to \infty} d_{\TV}(\mathbf{P}(\bX,C\log N),\nu_{\Unif})=1.
        \end{equation}
    \end{rmk}
    
    \begin{rmk}\label{rmk:p-like}
    A simple example of a randomized-$\mu$-like shuffle process is to take $(\bX^{(t)})_{t\in \mathbb Z^+}$ a sequence of IID random vectors $\bX^{(t)}=(\bX_0^{(t)},\cdots\bX_{k-1}^{(t)})\in(\mathbb Z^+)^k$, with $\bX_0^{(t)}+\cdots+\bX_{k-1}^{(t)}=N$ and $\bX^{(t)}\sim \mu$.
    However the condition \eqref{def:random-mu-like} is much more general and only requires homogeneity on macroscopic time intervals.
    For instance, this includes periodic cases where e.g.\ $\bn^{(2t)}=(N/3,2N/3)$ and $\bn^{(2t+1)}=(N/2,N/2)$.
    
    For IID shuffles, an equivalent interpretation of Theorem~\ref{thm:main} (and Remark~\ref{rmk:endpoint-case}) is as follows.
    Fix $k$ and consider any sequence $(\mu_N)_{N\geq 1}$ of probability measures on $k$-way cuts of a size $N$ deck.
    Then for any distance $d$ metrizing the compactification $[0,\infty]$ of $\bbR_+$ and any $\eps\in (0,1)$, we have with the obvious notations:
    \[
    \lim_{N\to\infty}
    d(t_{\mix}(\eps; N,\mu_N)/\log N , \oC_{\mu_N})
    =
    \lim_{N\to\infty}
    d(t_{\mix}(1-\eps; N,\mu_N)/\log N , \oC_{\mu_N})
    =
    0.
    \]
    (Here as usual $t_{\mix}(\eps)$ is the number of shuffles needed to reach total variation distance $\eps$ from stationarity.)
    \end{rmk}

    We finally provide a counterexample.
    As motivation, it is natural to expect some convexity properties for the rate of mixing, under averaging transition kernels.
    For example, if one dealer performs $\mu$-like shuffles and a second dealer performs $\mu'$-like shuffles, then choosing a random dealer in each step yields a $(\mu+\mu')/2$-like shuffler.
    At least for reversible chains, the spectral gap is concave under such averaging, which suggests that mixing should also be sped up.
    Surprisingly, this intuition can fail for (total variation mixing of) riffle shuffles: there exist pairs of dealers who shuffle less efficiently when working together!

    \begin{prop}
    \label{prop:non-convexity}
    There exists $k\in\bbN$ and $\bp,\breve\bp\in \cD_k$ such that $\mu=(\delta_{\bp}+\delta_{\breve\bp})/2$ satisfies
    \[
    \oC_{\mu}>\max(\oC_{\bp},\oC_{\breve\bp}).
    \]
    \end{prop}

    The proof (deferred to the Appendix) characterizes the closure of all possible $\psi_{\mu}$ up to rescaling, extends the definition of $\oC_{\mu}$ to this space, and finds an explicit counterexample within it. The value $k$ resulting from our proof is presumably quite large.
    We also note that $\widetilde C_{\mu}$ is trivially convex, and $C_{\mu}$ is convex on $\{\mu\in\cD_k:\theta_{\mu}=\theta_*\}$ for any fixed $k,\theta_*$, both of which may make Proposition~\ref{prop:non-convexity} more surprising.

    \subsection{Consequences and Discussion}

    As previously mentioned, the original GSR shuffle analysis \cite{bayer1992trailing} relies on the somewhat brittle notion of rising sequences.
    Our results are built instead on the cold spots picture introduced by Lalley \cite{lalley2000rate} and also used in \cite{mark2022cutoff}.
	Thus a conceptual message of our work is that cold spots give a robust and sharp characterization of the mixing time for general riffle shuffles.

    As discussed in \cite[End of Chapter 8]{diaconis2011magical} and \cite[Pages 200, 465]{todhunter2014history}, for dividing a set of $N$ items into parts of size $(n,N-n)$, the question of the ``most natural'' distribution of $n$ has been debated for centuries and traces back to the origins of probability. Some may argue that the ``most natural'' way is to put each item to each part independently with probability $1/2$, which leads to the binomial distribution $\Bin(N,1/2)$. On the other hand, some may believe that the size should be distributed uniformly on $[0,N]\cap \mathbb Z$.

    Our result is the first to handle the natural case of uniformly random cuts. It is straightforward that the corresponding shuffling process is randomized-$\mu$-like for $\mu=\Unif([0,1])$. Using \eqref{eq:C-p} and direct calculation, we find that $\oC_{\mu}\approx 3.606$ in this case (yielding the approximation $\oC_{\mu}\log(52)\approx 14.25$ for a real deck).
    Numerics for other natural choices of $\mu$ are given in Table~\ref{table:beta}.

    \begin{table}[htbp]
    \centering
\begin{tabular}{lcccccc}
\toprule
Measure $\mu$ on $\cD_k$ & $k$ & $\theta_\mu$ & $\psi_\mu(2)$ & $C_\mu$ & $\widetilde C_\mu$ & $\overline C_\mu$ \\
\midrule
$k=2$ Beta$(1,1)$ (Uniform on $\mathcal D_2$) & 2 & 3.197 & 0.430 & \textbf{3.606} & 3.256 & 3.606 \\
$k=2$ Beta$(1/2,1/2)$  & 2 & 3.237 & 0.316 & \textbf{4.932} & 4.554 & 4.932 \\
$k=2$ Beta$(2,2)$  & 2 & 3.149 & 0.525 & \textbf{2.926} & 2.562 & 2.926 \\
$k=2$ Beta$(2,16)$  & 2 & 3.812 & 0.207 & 8.230 & \textbf{8.256} & 8.256 \\
$k=2$ Uniform on $\{(p,1-p):p\in [1/4,3/4]\}$ & 2 & 3.121 & 0.613 & \textbf{2.496} & 2.097 & 2.496 \\
$k=3$ Dirichlet$(1,1,1)$ (Uniform on $\mathcal D_3$) & 3 & 3.190 & 0.724 & \textbf{2.138} & 1.927 & 2.138 \\
$k=3$ Dirichlet$(1/2,1/2,1/2)$  & 3 & 3.232 & 0.551 & \textbf{2.830} & 2.607 & 2.830 \\
$k=3$ Dirichlet$(2,2,2)$  & 3 & 3.141 & 0.863 & \textbf{1.779} & 1.552 & 1.779 \\
$k=4$ Dirichlet$(1,1,1,1)$ (Uniform on $\mathcal D_4$) & 4 & 3.183 & 0.948 & \textbf{1.631} & 1.465 & 1.631 \\
Delta-mass on $(1/k,\dots,1/k)$
& $k$ & $3$ & $\log k$ & $\mathbf{\frac{3}{2\log k}}$ & $\frac{1}{\log k}$ & $\frac{3}{2\log k}$
\\
\bottomrule
\end{tabular}
\caption{Numerical estimates of $\theta_\mu$, $C_\mu$, $\widetilde{C}_\mu$, and $\overline{C}_\mu$ for several Beta distributions and Dirichlet distributions. 
The last line $\mu=\delta_{(1/k,...,1/k)}$ includes the symmetric multinomial $k$-shuffles analyzed in \cite{bayer1992trailing}.
}
\label{table:beta}
\end{table}

    In fact, among all measures on $\mathcal D_k$, by definition of $\theta_{\mu}$ and \eqref{eq:theta-bound}, we see that $\theta_{\mu}$ reaches its minimum when $\mu$ is a $\delta$-mass on $(1/k,\cdots,1/k)$. On the other hand, by Cauchy's inequality we see that $\phi_{\bp}(2)$ reaches its minimum, and thus $\psi_{\bp}(2)$ reaches its maximum at $\bp=(1/k,\cdots,1/k)$. Therefore, we conclude that $C_{\mu}$ reaches its minimum when $\mu$ is a $\delta$-mass on $(1/k,\cdots,1/k)$. It is also straightforward that $\widetilde C_{\mu}$ is minimized in this case. 
    In other words, no distribution of cuts mixes asymptotically faster than GSR shuffles.
    
    Other riffle shuffling chains have also been studied and would be very interesting to understand better.
    In light of our work, the main remaining challenge is to modify the uniform interleaving condition~\ref{it:GSR-interleave} of the GSR shuffle.
    An especially natural model is the ``clumpy'' shuffle where cards are more likely to drop from the same pile as the previous card.
    Polylogarithmic mixing for this chain was shown in \cite{jonasson2015rapid}.
    Another model is the Thorp shuffle where one cuts the deck exactly in half and ensures the $m$-th cards from each pile end up adjacent to each other, yielding $2^{N/2}$ possible permutations from a single shuffle \cite{thorp1973nonrandom,morris2008mixing,morris2009improved,morris2013improved}.
    Cutoff is not known for any version of these chains.
    Many more results and problems on card shuffling can be found in \cite{diaconis2003mathematical,diaconis2023mathematics}.

    \subsection{Proof Techniques and Challenges}

    The proof of Theorem~\ref{thm:weak-convergence} naturally splits into lower and upper bounds.
    We discuss these separately; the lower bound is easier but more clearly illustrates the underlying mechanism.

    \paragraph{Lower Bound via Cold Spot Hypothesis Testing}

    The mixing time lower bound requires pinpointing the fundamental obstruction to mixing. 
    This was discovered in the multinomial case $\mu=\delta_{\bp}$ by Lalley \cite{lalley2000rate}, who identified a fractal set of \emph{cold spots} within $[N]$ where the cards maintain an unusually strong correlation with their initial ordering.
    (More precisely, cold spots lead to the lower bound $C_{\mu}$ in \eqref{eq:C-p}, which is the more interesting obstruction.)
    To describe cold spots, we must identify each card with a sequence in $\{0,1,\dots,k-1\}^K$ of its piles in each shuffle.
    As described further in Subsection~\ref{subsec:shuffle-graph}, the \emph{inverse} permutation of a riffle shuffle can be described by sorting the $N$ such strings in lexicographic order, yielding a sequence of strings $s_1\leq_{\lex}s_2\leq_{\lex}\dots\leq_{\lex}s_N$ in $[k]_0^K$.
    Equal adjacent strings $s_j=s_{j+1}$ then remain in order (by virtue of being in the same pile during all $K$ shuffles), yielding extra ascents in the inverse permutation.
    The main contribution turns out to come from strings containing each digit $j$ with frequency proportional to $p_j^{\theta_{\mu}}$ within their initial segment.
    The lengths of these initial segments are chosen to have approximately deterministic location, which are called cold spots.
    By counting ascents within the cold spots, \cite{lalley2000rate} constructed an explicit hypothesis test to detect under-shuffled decks and thus prove the sharp mixing time lower bound (for $\bp\approx (1/k,\dots,1/k)$ sufficiently close to uniform, a restriction removed in \cite{mark2022cutoff} via truncation).

    To extend this lower bound, we first approximate the measure $\mu$ by a discrete distribution (namely one uses $\sum_{0\leq i\leq z}\mu(D_i)\delta_{\bp^{(i)}}$ to approximate $\mu$, where the $D_i$ are as in \eqref{eq:partition-Dk}). 
    Then we focus on strings containing each digit $j$ among the ``$\bp^{(i)}$-like shuffles'' with frequency proportional to $(p_j^{(i)})^{\theta_{\mu}}$ in an initial segment; cold spots are the (deterministic) typical locations of such strings (see \eqref{def:coldspot}).
	If the pile sizes are random with fluctuations of order $O(N^{1/2+o(1)})$ as in the binomial case, then the same argument works with only minor changes.
	To handle larger sublinear fluctuations of pile sizes in the randomized case, we construct a many-to-one hypothesis test by exhausting all (suitably nicely behaved) sequences of pile sizes. 
	For such a pile size sequence $(n_l^{(t)})_{1\leq t\leq K,0\leq l\leq k-1}$, we test whether the deck appears more like it was shuffled according to such a sequence, or generated uniformly at random.
	By ensuring that each individual test has very small $e^{-N^{c}}$ probability to incorrectly detect non-uniformity (see Lemma \ref{hypo-test2}), we find that the resulting many-to-one test handles the random case, extending the lower bound to pile sizes with general distributions.

    \paragraph{Upper Bound via Truncated Exponential Moments}
    
    For the more difficult upper bound, it suffices to consider deterministic $\vec\bn$. 
    Identifying each card with a string in $\{0,1,\dots,k-1\}^K$ as described above, we consider the random subgraph of the $N$-vertex path graph by connecting $(j,j+1)$ when $s_j=s_{j+1}$.
    Via a truncated $\chi^2$ argument, upper-bounding the total variation distance is reduced to upper-bounding the exponential moment of the number of shared edges between two IID such graphs by $1+o(1)$. (I.e. the number of shared edges should be stochastically dominated by a geometric random variable with mean $o(1)$.)
    This strategy was developed in \cite{mark2022cutoff} for $\Mult(N,\bp)$ shuffles, which relied on convenient properties of multinomial distributions in many places, including concentration, anti-concentration, and simple behavior of conditional laws. 
    Even the case $\mu=\delta_{\bp}$ requires significant new work, due to the much more complicated conditional laws that arise within the analysis and the presence of $o(\log N)$ non-$\bp$-like shuffles.

    We first show the first moment is $N^{-\Omega(1)}$, assuming that the first shuffle is ``good'' in that it has at least $2$ piles with $\Omega(N)$ cards.
    To achieve this, we divide the shuffle graph into carefully chosen blocks in Lemma \ref{lem:partition} (corresponding to different possible prefix strings in $[k]_0^M$ for varying $M$) and argue that the edges are distributed ``approximately uniformly'' within each block. There are several key challenges.

    First, understanding the fluctuation for the location of the strings with a given prefix (after conditioning) is an important part of the first moment proof. (See Section \ref{subsec:fluc}.) In \cite{mark2022cutoff}, the order of fluctuation for the location can be determined by initial consecutive $0$-digits or $1$-digits. In other words, by comparing the expected numbers of strings smaller and larger than this prefix, the fluctuation is given by the square root of the minimum of these two quantities. However, in our case, since we allow some shuffles to put $o(1)$ or $1-o(1)$ fraction of cards into certain piles, it is not immediately clear whether the fluctuation is determined by the left or the right endpoint. 
    To adapt the proof, we carefully choose the information to reveal within the conditioning arguments to apply anti-concentration for hypergeometric variables (see Lemmas~\ref{lem:anticoncen} and \ref{lem:anticoncen-hyper}). By a delicate prefix decomposition, we construct the relevant hypergeometric variable $\Hyp(n,m,n_1)$ to satisfy $\max(m,n_1)/n\leq 1-\Omega(1/\log N)$.
    The fluctuation of this hypergeometric variable is of order the square root of its mean, up to this additional $\log N$ factor (which turns out to not affect the leading order mixing time).

    Second, in \cite{mark2022cutoff} each block is defined by all strings with a certain prefix, such that its location exhibits fluctuations of approximately the same order as its size, which can be done by recursively partitioning until a simple stopping criterion is met. Then one combines those blocks into classes according to the number of each digits, and proves that the expected number of edges is $N^{-c}$ in each of a total of $O(\log N)$ classes.
    This becomes significantly more challenging in our case, especially because we allow some shuffles to put $o(1)$ fraction of cards into certain piles. 
    For example, two blocks may behave very differently if one places an $N^{-1/2}$ fraction of cards in a certain pile, while the other places an $N^{-3/4}$ fraction.
    To handle this subtlety, we first define a set of prefixes such that the order of fluctuations is fixed once a string has a prefix in this set (see definition \ref{eq:def:admissible-prefixes}). After truncating very rare prefixes, there are at most $N^{c'}$ admissible prefixes for some sufficiently small $c'$. We then further decompose each prefix block into sub-blocks with fluctuations slightly larger than its size (see lemma \ref{lem:partition}). We then group the sub-blocks according to both the counts of each digit and the precise pattern of ``small pile'' digit occurrences (definition \ref{def:digit-profile}), and union bound over the groups.

    Having established the corresponding first moment estimate, we estimate the exponential moment by exploring a pair of IID shuffle graphs on $[N]$ one-by-one from each direction (both forward and backward). First by some shifting tricks and monotonicity of $\TV$-distance under extra shuffles, we may assume some initial shuffles are ``good'' in that they have no extremely small piles (see lemma \ref{lem:good beginning}).
    The goal now is to argue that the first moment estimate applies conditionally to the unrevealed shared edges and conclude using a form of Portenko's lemma.
    Since the pile sizes for the first shuffle are fixed, the shuffle graph naturally divides into $k$ parts, corresponding to the $k$ piles in the first shuffle.
    Therefore, instead of considering two exploration processes as in \cite{mark2022cutoff} (forward and backward), we analyze $2k$ processes in Lemma \ref{lem:for-back-cover}, one in each direction within each part.
    Analyzing the exponential moment now reduces to analyzing the expected number of unrevealed edges, conditionally on the subgraph we have explored. 
    To estimate this conditional quantity, we use a simple decomposition of the possible unrevealed strings into $O(\log N)$ disjoint prefix blocks,  
    and then refine this decomposition further (see \eqref{eq:block-for}) so that each prefix block is immediately followed by a ``good'' shuffle.
    The number of resulting blocks is still $N^{o(1)}$, so we can afford to apply the first moment bound to each separately and union bound (the $o(1)$ exponent can be arbitrarily slow depending on the convergence rate in \eqref{def:random-mu-like}).
    
    For ``large'' blocks, we apply the first-moment estimate in Proposition \ref{prop:new-firstmoment}. Namely by carefully excluding specific events in each exploration step, we find that conditioning on the explored subgraph, the ``sub-shuffle'' in each block is still ``$\mu$-like''. Thanks to concentration results in Lemma \ref{lem:concentration}, the probability of each excluded event is at most $e^{-N^c}$.
    However, for ``small" blocks, the concentration result does not apply, so that we may no longer use the first moment estimate.
    Instead, we bound the number of unrevealed edges directly by block size.
    The main challenge now comes from the more complicated conditional distributions of the unrevealed strings.
    Here we use an explicit combinatorial characterization of shuffling process introduced in Section \ref{subsec:shuffle-graph}, in which the conditional distribution is uniform on a specific set of matrices.
    Using this fact, we may count the number of those matrices in Section~\ref{ref:sub2sec-pf-lem-resam} and therefore give an estimate for the conditional distribution.

\section{Preliminaries}
	\label{sec:prelims}
    In this section, we introduce notation and useful results. In Section \ref{subsec:reduce-deter}, we reduce the original problem to studying a shuffling process with deterministic pile sizes.
    In Section \ref{subsec:shuffle-graph}, we define the shuffle graph and reformulate the problem into a more manageable form. Then, in Section \ref{subsec:concen-res}, we establish some useful concentration inequalities. Finally, we conclude this section by detailing additional notation.

    \subsection{Reduction to deterministic pile sizes}
    \label{subsec:reduce-deter}
        To prove Theorem \ref{thm:weak-convergence}, we deal with shuffling processes with deterministic pile sizes. To this end, we define a class of deterministic shuffling processes that represent the ``typical'' behavior of a randomized-$\mu$-like shuffling process. Recall the definition of $\cD_k$ in \eqref{eq:def-Dk} and recall that $V_k\subset \mathcal D_k$ is the vertex set of the domain $\cD_k$.  For convenience, we write $V_k=\{v_0,\cdots,v_{k-1}\}$.
    For each $\chi>0$, consider a finite decomposition
    \begin{equation}\label{eq:partition-Dk}
        \mathcal D_k=D_0\cup D_1\cup\cdots\cup D_z
    \end{equation}
    of $\cD_k$ into parts with disjoint interiors such that the following holds:
    \begin{enumerate}
        \item $z\leq O(\chi^{-2k})$;
        \item $D_i=\{\bp: d(\bp,v_i)<\chi\}$ for $0\leq i\leq k-1$;
        \item $D_i$ is a simplex for $k\leq i\leq z$; 
        \item For all $k\leq i\leq z$, it holds that $\mathrm{diam}(D_i)<\chi^2.$
    \end{enumerate}
    Here $d$ and $\mathrm{diam}$ denote distance and diameter in $\mathcal D_k$ associated with the $L^{\infty}$ norm on $\cD_k\subseteq\mathbb R^k$. 
    For convenience, we write $\widetilde V_k=\cup_{0\leq i\leq k-1}D_i$ which is a neighborhood of $V_k$.
    For a construction obeying these properties, note first that since $\cD_k\backslash \widetilde V_k$ is a convex polytope, it can be triangulated into finitely many interior-disjoint simplices on the same vertex set (see e.g.\ \cite{lee2017subdivisions}). 
    These simplices can then be further triangulated into subsimplices with arbitrarily small diameter using iterated barycentric subdivision.

    We say a shuffle $\bn^{(t)}$ is $\chi$-bad if $\bn^{(t)}/N \in \widetilde V_k$; otherwise we call it $\chi$-good.
    We say the shuffle process with pile sizes $(\vec \bn^{(t)})_{t\leq K}$ is \textbf{$(\chi,\rho,\varphi)$-almost-$\mu$-like} if there exists\footnote{We point out the reason for defining such $t_*$ rather than defining $t_*=0$. In Sections 3.1 and 3.4, we will consider new shuffling processes which start at some time $\bar t$ between $1$ and $K$. The above definition conveniently ensures that these new shuffling processes remain $(\chi,\rho,\varphi)$-almost-$\mu$-like.}
    $0\leq t_*<\rho \log N$, such that for every $0\leq i\leq \lceil K/\rho \log N \rceil$,
    \begin{equation}\label{def:almost}
        \bigg|\frac{1}{\rho\log N}|\{t_*+(i-1)\rho\log N< t\leq t_*+i\rho \log N: \vec\bn^{(t)}/N\in D_i\}|-\mu(D_i) \bigg| < \varphi.
    \end{equation}
    The next theorem shows deterministic almost-$\mu$-like shuffles undergo cutoff after $\overline{C}_{\mu}\log N$ steps.
    We state it separately from Theorem~\ref{thm:weak-convergence} because we will focus on the deterministic case first, and because dependence on the explicit parameters $(\chi,\rho,\varphi)$ will be tracked carefully in the proof.
    
	\begin{thm}
	\label{thm:main}
    For fixed $\mu$ such that $\mu(V_k)<1$ and for fixed $\varepsilon$, there exists $\chi=\chi(\mu,\varepsilon),\rho=\rho(\mu,\varepsilon),\varphi=\varphi(\mu,\varepsilon)>0$ such that the following holds. For any sequence of $(\chi,\rho,\varphi)$-almost-$\mu$-like shuffles with deterministic pile sizes $\bn^{(t)}=(n_0^{(t)},\cdots,n_{k-1}^{(t)})$,
		\begin{align}
			\label{eq:LowerBound}
			\lim_{N\to\infty} d_{\TV}
            \big(\bP(\vec\bn,\lfloor(1-\varepsilon)\overline{C}_{\mu}\log(N)\rfloor),\nu_{\Unif}\big)
            &=
            1,\\
			\label{eq:UpperBound}
			\lim_{N\to\infty} d_{\TV}
            \big(\bP(\vec\bn,\lfloor(1+\varepsilon)\overline{C}_{\mu}\log(N)\rfloor),\nu_{\Unif}\big)
            &=
            0,
		\end{align}
		where $\nu_{\Unif}$ is the uniform measure on $\mathsf S_N$ and the limits are uniform in $\vbn$.
    \end{thm}

     \subsection{Shuffle graph}\label{subsec:shuffle-graph}

	As in \cite{lalley2000rate,mark2022cutoff}, it will be convenient to work with the inverse permutations of riffle shuffles, which admit a simpler description we recall now.
    Let $[N]=\{1,2,\cdots, N\}$ and $[k]_0=\{0,1,\cdots,k-1\}$.
    For $M>0$, we write $[k]_0^M$ as the set of strings with digits in $[k]_0$ and with length $M$. 
	We then define a $\vbn$-dependent distribution on length $N$ sequences of length $K$ strings 
	\[
	S_K=(s_1,\cdots,s_N)\in ([k]_0^K)^N.
	\] 
    Namely, we consider all $N\times K$ matrices $S\in ([k]_0)^{[N]\times [K]}$ with entries in $[k]_0$ such that for all $t\in [K]$ and $l\in [k]_0$, 
    \[
    n_l^{(t)}=|\{i\in [N]:S[i,t]=l\}|
    \]
    is exactly the number of digits $l$ appearing in the $t$-th column.
    Here we write $S[i,t]$ for the entry of $S$ in row $i$ and column $t$.
    We call such matrices \emph{shuffle matrices} associated with $\vbn$. 
    Choosing such a matrix uniformly at random, each row of said matrix forms a string with length $K$, and we write $\overline S_K=\{\overline s_1,\cdots, \overline s_N\}$ as the multi-set of these strings. 
    Then $S_K$ is obtained by sorting $\overline S_K$ into lexicographically increasing order
	\[
	s_1\leq_{\lex}s_2\leq_{\lex}\dots\leq_{\lex}s_N.
	\]
    Here $x<_{\lex}y$ indicates that $x$ has a smaller digit at the first entry where $x$ and $y$ differ. 
    We often omit the subscript $K$ for simplicity, writing $S$ for $S_K$.
    (We note that if $\vbn^{(t)}$ are IID multinomial as in \cite{lalley2000rate,mark2022cutoff}, then the $NK$ digits of $\overline S_K$ are IID which is a major simplification.)

	Next we generate the \emph{shuffle graph} $G=G(S)$. The vertex set is defined by 
    \[
    [N]=\{1,2,\cdots,N\},
    \] and the edge set is defined by 
    \begin{equation}
    \label{eq:shuffle-graph-edges}
    \{(i,i+1):s_i=s_{i+1}, i=1,2,\cdots,N-1\}.
    \end{equation}
    Thus $G(S)$ is always a subgraph of the path graph on $N$ vertices.
	We call a connected component in $G$ a $G$-component. For $\sigma\in \mathsf S_N$ and a shuffle graph $G$, define $\sigma^G$ by within each $G$-component, sorting the values $\sigma(i)$ into increasing order.
	
	\begin{prop}
	\label{prop:shuffle-graph}
		Suppose that $\nu_{\Unif}$ is the uniform distribution on $\mathsf S_N$, $\pi\sim \nu_{\Unif}$, and $G$ is the shuffle graph generated as above. Then we have
		\begin{equation}
		d_{\TV} (\pi^G,\pi)
            =
            d_{\TV}\big(\Pi_{1\leq i\leq K}
            \bP(n_0^{(i)},\cdots,n_{k-1}^{(i)}),\nu_{\Unif}\big)
            .  
		\end{equation}
	\end{prop}
	
	\begin{proof}
		For all $i=1,2,\cdots, N$, let $\tilde s_i$ denote a $[k]_0$-valued string of length $K$, such that $\tilde s_i[j]=l\in [k]_0$ if and only if card $i$ belongs to pile $l$ in the $j$-th shuffle. Consider the shuffle graph $G(\tilde S)$ defined as in \eqref{eq:shuffle-graph-edges} with $\tilde S=(\tilde s_1,\cdots, \tilde s_N)$. Now we sample $X_K\sim \Pi_{t\leq K}\bP(n_0^{(i)},\cdots,n_{k-1}^{(i)})$. Direct calculation shows that, conditioned on $\tilde S$, the \emph{inverse} permutation $X_K^{-1}$ is distributed uniformly on the set 
        \[
        \{\sigma\in \mathsf S_N: \sigma=\sigma^{G(\tilde S)}\}.
        \] 
        (See \cite[Lemma~3]{lalley2000rate} for more discussion of this last assertion.)
		
		It suffices to show that $S_K\overset{d}{=}\tilde S_K$. We prove this by induction on $K$. The case when $K=1$ is obvious. Suppose that we have proved the claim for $K$. With the discussion in the above paragraph, to determine the inverse permutation of the deck, we first generate $\pi\in \mathsf S_N$ uniformly, then sort them in order in each $G(\tilde S_K)$-component, and write $j$ at the $(K+1)$-th digits of $s_i$ if 
        \[
        \pi^{G(\tilde S_K)}(i)
        \in
        \Big[\sum_{k<j}n^{(K+1)}_k+1,\sum_{k\leq j}n^{(K+1)}_k\Big].
        \]
		This process is equivalent to sampling a uniform vector with $n^{(K+1)}_j$ digits $j$ (for all $j=0,1,\cdots,k-1$), and sorting the vector in lexicographic order within each $G(\tilde S_K)$-component. 
        (Indeed generating a $N\times (K+1)$ matrix and then sorting in lexicographic order is equivalent to first generating a $N\times K$ matrix with a fixed number $n^{(i)}$ of $j$ digits in column $i$ and sorting it in lexicographic order, then generating column $K+1$ and sorting it within the rows with same first $K$ digits.)
        Combined with the inductive hypothesis that $S_K\overset{d}{=}\tilde S_K$, we deduce that $S_{K+1}\overset{d}{=}\tilde S_{K+1}$. 
	\end{proof}

	\subsection{Concentration results} \label{subsec:concen-res}
        In this subsection, we demonstrate concentration results for $Y_x$, the number of strings with prefix $x$ in $S_K$.
	Since we work with fixed pile sizes most of the time, many relevant quantities will be hypergeometric random variables.
	Let $\Hyp(n_1,n,m)$ be the distribution of the intersection size between two independent uniformly random subsets of $[n]$ with cardinalities $n_1$ and $m$.
	The following very useful lemma shows concentration for such distributions.

	\begin{lem}[\hspace{-0.05em}{{\cite[Theorem~1]{hush2005concentration}}}]
	\label{concen-hypergeo}
		Let $X\sim \Hyp(n_1,n,m)$, for all $x>2$ we have
		\begin{equation}
			\mathbb P(|X-\mathbb E X|\geq x)\leq \exp(-2\alpha_{n_1,n,m}(x^2-1)),
		\end{equation}
		where
		\[\alpha_{n_1,n,m}=\max\lt(\frac{1}{n-n_1+1}+\frac{1}{n_1+1},\frac{1}{n-m+1}+\frac{1}{m+1}\rt).\]
	\end{lem}

    We give stronger versions of the concentration for hypergeometric distribution.
    \begin{lem} \label{lem:concen-hypergeo-stronger}
        Let $X\sim \Hyp(n_1,n,m)$ with $\mathbb EX=\frac{n_1m}{n}\geq 10$. There exists absolute constants $C,c$ such that for all $\delta<1/2$,
		\begin{equation} \label{eq:concen-hypergeo-stronger}
			\mathbb P(|X-\mathbb E X|\geq (\mathbb EX)^{1/2+\delta})\leq C\exp(-c(\mathbb EX)^{2\delta}).
		\end{equation}
    \end{lem}
    \begin{proof}
        We follow the proof of \cite[Theorem~1]{hush2005concentration}. 
        Write $p(k)=\mathbb P(X=k)$ and \[q(k)\equiv p(k+1)/p(k)=\frac{(m-k)(n_1-k)}{(k+1)(n-n_1-m+k+1)}.\]
        It is clear that $q(k)$ is decreasing in $k$. Denote $d(k)=\sum_{i\geq k} p(i)$, thus we have
        \[d(k+1)=\sum_{i\geq k}p(i+1)=\sum_{i\geq k}q(i)p(i)\leq q(k)\sum_{i\geq k}p(i)=q(k)d(k).\]
        It suffices to bound $q(k)$ to obtain a decay rate for $d(k)$. We extend $q$ to be a function on $\mathbb R^+$ by defining
        \[q(x)=\frac{(m-x)(n_1-x)}{(x+1)(n-n_1-m+x+1)}.\]
        Taking the log derivative, we find 
        \[
        \frac{\mathrm{d}}{\mathrm{d}x}\log q(x)=-\frac{1}{m-x}-\frac{1}{n_1-x}-\frac{1}{x+1}-\frac{1}{n-n_1-m+x+1}<-\frac{1}{x}.
        \]
        Writing $\kappa=(n_1m+m+n_1-1-n)/(n+2)$, we have $q(\kappa)=1$.        
        For $x\in [\kappa,2\kappa]$, we then have 
        \[
        \frac{\mathrm{d}}{\mathrm{d}x}\log q(x)\leq -\frac{1}{2\kappa}\mbox{ and } \log q(x)\leq -\frac{x-\kappa}{2\kappa},
        \] 
        which leads to the bound 
        \[\log d(k)\leq \log d(\lceil\kappa\rceil) +\sum_{i=\lceil\kappa\rceil}^{k-1}\log q(i)\leq -\sum_{i=\lceil\kappa\rceil}^{k-1}\frac{i-\kappa}{2\kappa}<-\frac{1}{2\kappa}((k-\kappa-1)^2-1).
        \]
        Notice that $|\mathbb EX-\kappa|\leq 1$, we then obtain the upper bound \[\mathbb P(X-\mathbb EX\geq (\mathbb EX)^{1/2+\delta})\leq C\exp(-c(\mathbb EX)^{2\delta}).\]
        By a similar argument we then obtain the lower bound, which then completes the proof.
    \end{proof}

    \begin{lem}\label{lem:concen-hypergeo-smallmean}
        Let $\delta\in (0,1/2),\mathsf C>0$, $X\sim \Hyp(n_1,n,m)$ with $\mathbb EX=\frac{n_1m}{n}\leq  \mathsf CN^\delta$. There exists constants $C,c$ depending on $\mathsf C$ such that
            \begin{equation}\label{eq:concen-hypergeo-smallmean}
                \mathbb P(X\geq 2a)\leq C\exp(-ca) \mbox{ for all } a\geq 3\mathsf CN^q.
            \end{equation}
    \end{lem}
    \begin{proof}
        By the calculation in the previous lemma, we see that $\log q(x)\leq -\log \frac{x}{\kappa}$. We have $q(a)\geq 2$ for $a\geq 2N^\delta$. Thus we have $\log d(a)\leq -(a-2\mathsf C N^\delta)\log 2\geq ca$ for $a\geq 3\mathsf CN^\delta$, which completes the proof.
    \end{proof}
	
    Lemmas \ref{lem:concen-hypergeo-stronger} and \ref{lem:concen-hypergeo-smallmean} give concentration for hypergeometric distributions in two different regimes. Lemma \ref{lem:concen-hypergeo-stronger} states that if a hypergeometric variable has a macroscopic mean, it will concentrate around its mean with high probability. Lemma \ref{lem:concen-hypergeo-smallmean}, on the other hand, states that a hypergeometric variable with a small mean is unlikely to take large values. With these two lemmas as inputs, we can deduce the following concentration results for the number of strings with fixed prefix.

	\begin{lem}\label{concen-prefix}
		For $K=O(\log N)$, let $Y_x$ denote the number of strings with prefix $x$ in $S_K$. Let $T_q$ be the set of all strings $x$ such that $\mathbb EY_x\geq N^q$.
		 Then there exist constants $C,c$ such that for all $a>0$,
	\begin{equation}\label{concen-prefix-ineq}
			\mathbb P(|Y_x-\mathbb EY_x|\leq CK(\mathbb EY_x)^{1/2+a})\geq 1- C\exp(-c(\mathbb EY_{x})^{2a}).
		\end{equation}
        Moreover,
        \begin{equation}\label{concen-prefix-union-ineq}
			\mathbb P(|Y_x-\mathbb EY_x|\leq CK(\mathbb EY_x)^{1/2+a}, \mbox{ for all }x \in T_q)\geq  1 - C\exp(-cN^{2qa}).
		\end{equation}
	\end{lem}
    
	\begin{proof}
        For a string $x=x[1]\cdots x[M]$, we define $x_t=x[1]\cdots x[t]$ for $1\leq t\leq M$. We have immediately that condition on $Y_{x_t}$, $Y_{x_{t+1}}\sim \Hyp(Y_{x_t},N,n_{t+1}^{(t+1)})$. 
        Suppose that $Y_{x_t} \geq \frac{1}{2} \E Y_{x_t}$, and we have
        \[
        \E[ Y_{x_{t+1}}|Y_{x_t}]=Y_{x_t}\cdot\frac{n_{x[t+1]}^{(t+1)}}{N}  \geq \frac{1}{2} \E Y_{x_{t+1}} \geq 10.
        \] 
        Therefore by Lemma \ref{lem:concen-hypergeo-stronger}, we obtain that for some positive absolute constant $C,c$,  
        \begin{align*}
        \mathbb  P\lt(\left.|Y_{x_t+1}-\E[ Y_{x_{t+1}}|Y_{x_t}]|\geq \lt(\E[ Y_{x_{t+1}}|Y_{x_t}]\rt)^{1/2+a}\right|Y_{x_t}\rt)
        &\leq C\exp \lt(-c\lt(\E[ Y_{x_{t+1}}|Y_{x_t}]\rt)^{2a}\rt)\\
        &\leq \exp \lt(-c\lt(\E Y_{x_{t+1}}\rt)^{2a}\rt).
        \end{align*}
        Thus conditionally on $Y_{x_t}$, with probability at least $1-\exp (-c(\E Y_{x_{t+1}})^{2a})$, 
        \[
        Y_{x_{t+1}}\in \bigg[Y_{x_t}\cdot\frac{n_{x[t+1]}^{(t+1)}}{N}(1-c(\E Y_{x_{t+1}})^{-1/2+a}),`
        ~
        Y_{x_t}\cdot\frac{n_{x[t+1]}^{(t+1)}}{N}(1 +c(\E Y_{x_{t+1}})^{-1/2+a})\bigg].\]
		By an induction on $t$, we deduce that with probability at least $1-(t+1)\exp(-c(\mathbb EY_{x_{t+1}})^{a})$, the following holds:
        \begin{equation}
            Y_{x_{t+1}}\in 
            \bigg[\E [Y_{x_{t+1}}]\cdot \prod_{i \leq t+1} (1-c(\E Y_{x_{i}})^{-1/2+a}), 
            ~
            \E [Y_{x_{t+1}}]\cdot \prod_{i \leq t+1} (1+c(\E Y_{x_{i}})^{-1/2+a})\bigg].
        \end{equation}
        In particular, since $\E Y_{x_{i}} \geq \E Y_{x} \geq N^q$ and $t+1 \leq K= O(\log N)$, we have $Y_{x_{t+1}} \geq \frac{1}{2} \E Y_{x_{t+1}}$. Thus the induction is sound.
        The proof is then completed by a union bound over $x$. (Notice that the factor accounting for the number of possible $x$ can be absorbed by adjusting $c$.)
    \end{proof}

    Similarly, one may use Lemma \ref{lem:concen-hypergeo-smallmean} to obtain the following lemma.
    \begin{lem}\label{lem:concen-block-smallmean}
        For $K=O(\log N)$, let $Y_x$ denote the number of strings with prefix $x$ in $S_K$. If $\mathbb EY_x\leq N^q$, then
        \begin{equation}
            \mathbb P(Y_x\geq 3N^q)\leq \exp(-cN^q).
        \end{equation}
    \end{lem}
	\begin{proof}
	    For a string $x=x[1]\cdots x[M]$ with $\mathbb E Y_x\leq N^q$. We define $x_t=x[1]\cdots x[t]$ for $1\leq t\leq M$. Suppose $t$ is the smallest integer such that $\mathbb EY_{x_t}\leq N^q $. By \eqref{concen-prefix-ineq}, we have
        \[
        \mathbb P(Y_{x_{t-1}}\geq 2\mathbb EY_{x_{t-1}})\leq \exp(-cN^q).
        \]
        On the event $\{Y_{x_{t-1}}\leq 2\mathbb EY_{x_{t-1}}\}$, we have $\mathbb E[Y_{x_t}|Y_{x_{t-1}}]\leq 2N^q$. We then obtain by Lemma \ref{lem:concen-hypergeo-smallmean} that
        \[
        \mathbb P(Y_{x_t}\geq 3N^q)\leq \exp(-cN^q),
        \] 
        completing the proof since $Y_x\leq Y_{x_t}$.
	\end{proof}

    \subsection{Notation}
	
    We write $o_{\chi}(\cdot)$ to indicate quantities which are $o(\cdot)$ as $\chi$ tends to 0.
    We write e.g.\ $0<\delta_1\ll\delta_2<1$ if $\delta_1$ is chosen sufficiently small depending on $\delta_2$, and $1<m_1\ll m_2<\infty$ if $m_2$ is sufficiently large given $m_1$.
    Throughout this paper, we work with several constants. First, we choose a large constant $\mathsf C$ satisfying
    \begin{equation}\label{eq:mathsfC}
        \mathsf C \gg \frac{K}{\log N},
    \end{equation}
    i.e. $C$ is sufficiently large depending on $\mu$ (with the implicit assumption that $\frac{K}{\oC_{\mu}\log N}\in [1/2,2]$, say).
    Then we choose small constants in the order:
    \begin{equation}
    \label{eq:small-params}
    0<\zeta\ll\varphi\ll\rho\ll\chi\ll\xi\ll\delta\ll\eps.
    \end{equation}
     In other words $\eps$ small depending on $(\mu,\mathsf C)$, then choose $\delta$ small depending on $(\mu,\mathsf C,\eps)$, then choose $\xi$ arbitrarily small depending on $(\mu,\mathsf C,\eps,\delta)$, and so on.
    The constants $\chi,\rho,\varphi,\eps$ will correspond to the statement of Theorem~\ref{thm:main}. For fixed $\chi$, we will fix a partition of $\mathcal D_k$ as in \eqref{eq:partition-Dk}. We write 
    \begin{equation}\label{def:hipi}
        h_i=\mu(D_i) \mbox{ and choose } \bp^{(i)}=(p_{0}^{(i)},\cdots,p_{k-1}^{(i)})\in D_i^{\circ}, \mbox{ for all }0\leq i\leq z,
    \end{equation}
    which will be treated as fixed values. 
    (Here $D_i^{\circ}$ denotes the interior of $D_i$.)
    We also introduce the set
    \begin{equation}\label{def:set-t}
        T=\{(i,l):p_l^{(i)}>\chi\}.
    \end{equation}
    The constant $\delta$ features crucially in the first moment analysis of Section~\ref{sub2sec: 1stmoment} and the lower bound analysis of Section~\ref{sec:lower-bound}. The large constant $\mathsf C$ will appear in Definition \ref{def:class}, where we need to truncate the prefix set.
    Finally $\xi$ and $\zeta$ will appear primarily in Lemma~\ref{lem:unreveal-edges} near the end of the proof.
    All of these constants are fixed as $N\to\infty$, i.e.\ we also have $1/\log N\ll \zeta$.
    In particular they do not depend on the precise cut sizes $\vbn$.

    Consider a shuffle process with pile sizes $\vbn=(\vec\bn^{(t)})_{t\in \mathbb Z^+}$, treated here as deterministic (by conditioning on $\vbn$, in the randomized case).

    \noindent
    With arbitrary tie-breaking, we define $l_{\Max}^{(t)}=\arg\max_{i\in [k]_0}n_i^{(t)}$.
    Since $\sum_{i \in [k]_0} n^{(t)}_{i}=N$, we have that \begin{equation}
    \label{eq:l-Max}
    n^{(t)}_{ l_{\Max}^{(t)}} \geq \frac{N}{k}.
    \end{equation} 

    \noindent
    For a prefix $x\in [k]_0^M$, we write $B_x$ as the set of strings in $[k]_0^K$ with prefix $x$. Moreover, we define 
	\begin{equation}
        \label{eq:numstr-less-x}
		\iota(x)=|\{i\in [N]:s_i<_{\lex}x\}|+1,\quad\quad\tau(x)=|\{i\in [N]:s_i<_{\lex}x \text{ or } s_i\in B_x\}|.
	\end{equation}
	These bookend the (random) discrete interval
	\begin{equation}
	\label{eq:I-Bx}
	\cI(B_x)=\{i\in [N]:s_i\in B_x\}
	=
	\{\iota(x),\iota(x)+1,\dots,\tau(x)\}
	\end{equation} 
	consisting of the indices of strings in $B_x$.
    Note that $Y_x=|\cI(B_x)|$ (as defined in Lemma~\ref{concen-prefix}).
    Moreover, we define $G_{B_x}$ to be the induced subgraph of $G$ with vertex set $\cI(B_x)$, which retains the edges $(i,  i + 1) \in E(G)$ such that $s_i = s_{i+1} \in B_x$. Denote its edge set by $E(G_{B_x})$.
    We also define $t_x$ and $\lambda_x$ in $[0,1]$ to be the expected endpoints of $\cI(B_x)$:
	\begin{equation} 
	\label{eq:exp-numstr-less-x} 
		t_x = \frac{1}{N}\mathbb E \iota(x), 
        \quad\quad
        t_x+\lambda_x = \frac{1}{N} \mathbb E \tau(x).
	\end{equation}
    By the definition of $S_K$, for any string $x=x[1]\cdots x[M]$, we have
    \begin{equation}\label{eq:lambda-x-repre}
        \lambda_x=\prod_{t=1}^{M}\frac{n^{(t)}_{x[t]}}{N}.
    \end{equation}
	For a prefix $x$, we let 
	\begin{equation}
	\label{eq:Jx}
		J_x = [t_x, t_x+\lambda_x)
	\end{equation}
	be the expected location of the strings with prefix $x$ in $S_K$.

        For convenience, we write $i^n=ii\cdots i$ for a length $n$ string consisting solely of the digit $i$.
        For two strings $x=x[1]\cdots x[a]$ with length $a$ and $y=y[1]\cdots y[b]$ with length $b$, we write $[xy]$ to denote their concatenation, i.e.\ the length $a+b$ string given by
        \[
        [xy]=x[1]\cdots x[a]y[1]\cdots y[b].
        \]
        Recall the definition of $D_i$ in \eqref{eq:partition-Dk}. We consider the following partition of the set $[K] \times [k]_0=\mathsf P\cup \mathsf Q= \mathsf P(\vbn)\cup \mathsf Q(\vbn)$, where
        \begin{equation}
        \label{eq:P-Q-def}
            \mathsf P_i= \{(t,l)\in \mathsf P: \bn^{(t)}/N \in D_i,  p_l^{(i)} \geq \chi\},\,\,
            \mathsf P= \cup_{0\leq i\leq z} \mathsf P_i \mbox{  and  } \mathsf Q=([K]\times [k]_0)\setminus \mathsf P.
        \end{equation}
        It is elementary that $|\mathsf P|+|\mathsf Q|=Kk$, and
        \[
        |\{l:(t,l) \in \mathsf P\}|+|\{l:(t,l) \in \mathsf Q\}|=k, \quad \forall t \in [K].
        \]

        \smallskip
        Given any $k$-tuple $(a_0,a_1, \ldots, a_{k-1})$ of non-negative real numbers with sum $a_{\tot}>0$, let 
	\begin{equation*}
		H(a_0,a_1,\ldots, a_{k-1})=\frac{\sum_{i=0}^{k-1}a_i\log (\frac{a_i}{a_{\tot}})}{a_{\tot}}
	\end{equation*}
	be the entropy of the discrete distribution with weights $(\frac{a_i}{a_{\tot}})_{0 \leq i \leq k-1 }$.
	The following entropic approximation for multinomial coefficients will be useful.
	
	\begin{lem}[\hspace{-0.05em}{{\cite[Lemma~2.2]{csiszár2004information}}}]
    \label{lem:entropy}
		For any fixed $A>0$ and any non-negative real numbers $a_0,a_1,\ldots a_{k-1} \in [0,A]$, satisfying $a_i\log N \in \mathbb{Z}$, 
		\begin{equation*}
			N^{a_{\tot} H(a_0,\dots,a_{k-1})-o_N(1)}\leq \binom{a_{\tot}\log(N)}{a_0\log(N),\dots,a_{k-1}\log(N)}\leq N^{a_{\tot} H(a_0,\dots,a_{k-1})},
		\end{equation*}
		where the term $o_N(1)$ tends to $0$ for any fixed $A$ as $N\to\infty$, uniformly in $a_0,\dots,a_{k-1}\in [0,A]$.
	\end{lem}
	
	Let $\bp^{t}$ be the probability distribution given by $(\bp^{t})_i=\frac{p_i^t}{\phi_{\bp}(t)}$. Define
	\begin{equation}\label{def:entropy}
		I(\bp,\bp^{t})\equiv D_{\term{KL}}(\bp^t\mid\mid\bp)+H(\bp^t)=\sum_{i=0}^{k-1} (\bp^{t})_i \log(1/p_i)=\sum_{i=0}^{k-1} \frac{p_i^t \log(1/p_i)}{\phi_{\bp}(t)}>0.
	\end{equation}
    We have
    \begin{equation}\label{eq:entropy-psi}
        H(\bp^t)=t\cdot I(\bp,\bp^t)-\psi_{\bp}(t).
    \end{equation}
	This quantity appears in calculations associated with the cold spot obstruction.

    \section{Upper Bound}
    \label{sec:upper-bound}
    In this section, we provide upper bounds on the mixing time. 
    Throughout, we fix a deterministic sequence of pile sizes $(\vbn^{(1)},\vbn^{(2)},\dots)$ (except for Subsection~\ref{subsec:rand-pile-upperbound} which addresses the random pile size case).
    We begin by proving \eqref{eq:UpperBound}, the case of deterministic almost $\bp$-like shuffles. 
    In Section~\ref{subsec:3-shuffle-reduction}, we show that the initial shuffles can be assumed to behave well, which is needed later for technical reasons.
    In Section \ref{subsec:upbound-deterministic}, we reduce the proof of \eqref{eq:UpperBound} to verifying Lemma~\ref{lem:for-back-cover}, which provides an estimate on the ``truncated exponential moment” of the number of shared edges of two independent samples, and Lemma~\ref{lem:L-sparsity}, which asserts that the shuffle graph is ``sparse” with high probability.
    Finally, we conclude this section by proving that, with high probability, the random series of pile sizes $\bX$ is $(\chi,\rho,\varphi)$-almost-$\mu$-like and deriving \eqref{eq:random-UpperBound} from \eqref{eq:UpperBound}.
    In fact, in Sections~\ref{sub2sec: 1stmoment} and \ref{sub2sec: expmoment}, we will upper-bound the first moment and the exponential moment of the number of shared edges respectively, thereby establishing Lemmas~\ref{lem:for-back-cover} and \ref{lem:L-sparsity}.

    \subsection{Ensuring initial shuffles are good}
    \label{subsec:3-shuffle-reduction}
    For technical reasons in Section~\ref{sub2sec: expmoment}, we need to assume some condition on initial shuffles.
    The following lemmas explain why this assumption can be made without loss of generality.
    We first observe that ignoring an initial subsequence of shuffles only slows mixing.   
    \begin{lem}
    \label{lem:mixing-subsequence}
    For any sequence of pile sizes and any $1\leq a\leq K$, we have:
    \[
    d_{\TV}(\Pi_{1\leq i\leq K}
    \bP(n_0^{(i)},\cdots,n_{k-1}^{(i)}),\nu_{\Unif})
    \leq 
    d_{\TV}(\Pi_{a\leq i\leq K}
    \bP(n_0^{(i)},\cdots,n_{k-1}^{(i)}),\nu_{\Unif})
    \]
    \end{lem}

    \begin{proof}
    Let $\pi_{[1,a)}=\Pi_{1\leq i\leq a-1}
    \bP(n_0^{(i)},\cdots,n_{k-1}^{(i)})$ and $\pi_{[a,K]}=\Pi_{a\leq i\leq K}
    \bP(n_0^{(i)},\cdots,n_{k-1}^{(i)})$.
    Then the claim follows from:
    \[
    d_{\TV}(\pi_{[1,a)}*\pi_{[a,K]},\nu_{\Unif})
    =
    d_{\TV}(\pi_{[1,a)}*\pi_{[a,K]},\pi_{[1,a)}*\nu_{\Unif})
    \leq 
    d_{\TV}(\pi_{[a,K]},\nu_{\Unif}).
    \]
    Here the first step holds because $\nu_{\Unif}$ is the uniform distribution, while the second is a general property of convolutions.
    \end{proof}

    Next, we show there must exist three $\chi$-good shuffles within the first $\frac{\varepsilon}{2}\log N$ shuffles which are relatively close together. 
    This condition guarantees that for any prefix $x$ whose length $\mathsf A$ corresponds to the third $\chi$-good shuffle, the expected size of $\mathcal I(B_x)$—denoted by $\lambda_x$—admits a constant lower bound.
    Essentially speaking, this ensures that there remains sufficient “space’’ in $[k]_0^M$ for the future strings that appear in the exploration process.
    For related arguments and detailed estimates, we refer the reader to Section~\ref{sec:exp-moment}, in particular Section~\ref{ref:sub2sec-pf-lem-resam} and equation~\eqref{eq:den-lobd}.
    Recall the definition of $t_*$ in \eqref{def:almost}. Here and throughout the rest of the paper we set
    \[
    \mathsf{A}(\mu) = 2(1-\mu(\widetilde V_k))^{-1}.
    \]

    \begin{lem} \label{lem:good beginning}
        For any $(\chi,\rho,\varphi)$-almost-$\mu$-like shuffle process with pile sizes $(\vec \bn^{(t)})_{t\leq K}$, there exists $t_* \leq t_0 < t_1<t_2 \leq t_0 + \mathsf{A}(\mu) \leq t_* + \rho\log N$ such that $\vec \bn^{(t_0)}$, $\vec \bn^{(t_1)}$ and $\vec \bn^{(t_2)}$ are $\chi$-good.
    \end{lem}

    \begin{proof}
        Suppose for contradiction that for all 
        \[
        t_*\leq t \leq t_*+\rho \log N - \mathsf{A}(\mu)
        \]
        such that $\vec \bn^{(t)}$ is $\chi$-good, there is at most one $\chi$-good shuffle among $\vec \bn^{(t+i)}$ ($1\leq i\leq \mathsf{A}(\mu)$).
        This implies a total of at least 
        \[\rho \log N \cdot \frac{\mathsf{A}(\mu)-1}{1+\mathsf{A}(\mu)}>\rho \log N\cdot(\mu(\widetilde V_k)+2\varphi)\]
        $\chi$-bad shuffles in $[t_*,t_*+\rho\log N]$.
        However, this results in a contradiction, since for any $(\chi,\rho,\varphi)$-almost-$\mu$-like shuffle process, at most $\rho \log N\cdot(\mu(\widetilde V_k)+\varphi)$ shuffles in $[t_*,t_*+\rho\log N]$ can be $\chi$-bad.        
    \end{proof}

    By Lemma~\ref{lem:good beginning}, we can find some $t_0 <2\rho \log N \leq \frac{\varepsilon}{2}\log N$ such that $\vec \bn^{(t_0)}$ is $\chi$-good, and $\vec \bn^{(t_0+i_1)}$, $\vec \bn^{(t_0+i_2)}$ are $\chi$-good for some $i_1<i_2\leq\mathsf{A}(\mu)$.
    Consider now the shuffle process with pile sizes
    \[
        \vec \bn^{(t),\new}=\vec \bn^{(t+t_0-1)}, \quad \mbox{for all } 1 \leq t \leq K^{\new} \overset{\text{def.}}{=} K-t_0+1.
    \]
    In this new shuffle process $(\vec \bn^{(t),\new})_{t \leq K^{\new}}$, the first shuffle is $\chi$-good, and there are two $\chi$-good shuffles in the interval $(1, \mathsf{A}(\mu)]$.
    Moreover, the new shuffle remains $(\chi,\rho,\varphi)$-almost-$\mu$-like, and $K^{\new} \geq (\overline{C}_{\mu}+\frac{\varepsilon}{2})\log N$. We emphasize here that by introducing $t_*$, the $(\chi,\rho,\varphi)$-almost-$\mu$-like property is preserved after shifting.
    By Lemma~\ref{lem:mixing-subsequence}, if we establish the desired upper bounds on the total variation distance for the new shuffle, we obtain the corresponding results for the original shuffle process.
    Therefore, without loss of generality, we may assume that the first shuffle, and two other shuffles before time $\mathsf A(\mu)$ are $\chi$-good.

    \subsection{Upper bound in the deterministic case} \label{subsec:upbound-deterministic}    
    Recall that $\pi\sim\nu_{\Unif}$ is uniformly distributed on the symmetric group $\mathsf S_N$ and $G=G(S_K)$ is the shuffle graph after $K$ shuffles.
    We denote by $\P^{\pi}$ the probability measure under which $\pi\sim\nu_{\Unif}$, and by $\E^{\pi}$ the corresponding expectation.
   	We first describe the Radon--Nikodym derivative of $\pi^G$ with respect to $\pi$, following \cite[Section 3]{mark2022cutoff}.
	Let $v_i$ denote the number of vertices in the $i$-th component of $G$.
   	 For fixed $\sigma\in \mathsf S_N$, recalling the definition of $(\cdot)^G$ from above Proposition~\ref{prop:shuffle-graph}, direct calculation shows that
    \[
    \mathbb P^\pi(\pi^G=\sigma)=\frac{1_{\{\sigma^G=\sigma\}}\Pi_i(v_i!)}{N!}.
    \]
    Therefore, the Radon--Nikodym derivative of $\pi^G$ with respect to $\pi$ is given by
    \[
    f_{G,\sigma}\equiv \frac{\mathbb P^\pi(\pi^G=\sigma)}{\mathbb P^{\pi}(\pi=\sigma)}=\frac{1_{\{\sigma^G=\sigma\}}}
    {\mathbb P^{\pi}(\pi^G=\pi)}
    \]
    We now consider $\sigma$ as a random variable with law $\nu_{\Unif}$, and denote by $\E^{\sigma}$ the corresponding expectation. 
    From the equation above, we can derive that the total variation distance to the uniform distribution after $K$ shuffles is given by 
    \[
    d_{\TV}\big(\Pi_{t\leq K}\bP(n_0^{(t)},\cdots,n_{k-1}^{(t)}),\nu_{\Unif}\big)
    =
    \frac{1}{2}\cdot \mathbb E^{\sigma}\left|\mathbb E^{S}[f_{G(S),\sigma}-1]\right|.
    \]
    Next, we apply a chi-squared upper bound for the total variation distance after excluding some exceptional rare events. Consider a partition $\cS= \cS^1 \cup \cS^0$, where $\cS$ denotes the space of all sequences generated by shuffle graphs of scale $N\times K$, and $\cS^1$ consists of all ``typical" sequences.     (We will not fix such a partition explicitly until Section~\ref{sec:exp-moment}, but one should just think for now that $\mu_K(\cS^0)$ is small and that $\cS^1$ will have some nice properties to be determined as needed.)
    Recall that $S_K$ distributes uniformly on $\cS$. Take $S'_K$ to be an independent copy of $S_K$ and define for any shuffle graphs $G,G'$,
    \[ f_{G,G'}\equiv \mathbb E^{\sigma}[f_{G,\sigma}f_{G',\sigma}].\]
    By \cite[Equations (3.2) and (3.3)]{mark2022cutoff}, we obtain
    \begin{equation} \label{eq:double-trick}
        \mathbb E^{\sigma}
        \left|\mathbb E^{S}[f_{G(S_K),\sigma}-1]\right| 
        \leq 
        \sqrt{
        \mathbb E^{S,S'}
        \left[
        (f_{G,G'}-1)
        1_{S_K,S'_K\in \cS^1}
        \right]
        }
        +
        2\mu_K(\cS^0).
    \end{equation}
    We now upper-bound $f_{G,G'}$ using the number of edges shared by $G$ and $G'$.
    Let $E(G,G')=E(G)\cap E(G')$. 
    For two shuffle graphs $G,G'$, we define the new shuffle graph $U$ to be their edge union and denote by $\mathcal C(U)$ the set of connected components of $U$.
    The next lemma shows how to upper-bound $f_{G,G'}$ based on $E(G,G')$ and the new shuffle graph $U$.

    \begin{lem}[\hspace{-0.05em}{\cite[Lemma~3.1]{mark2022cutoff}}]
        Suppose the $U$-components have vertex-sizes $(u_1,\dots,u_c)$. Then
        \begin{equation}\label{eq:fggbound}
            f_{G,G'}\leq \prod_{\substack{1\leq i\leq c,\\
            E(U_i)\cap E(G,G')\neq\emptyset}} (u_i!).
        \end{equation}
    \end{lem}

    We would like the right-hand side of \eqref{eq:fggbound} to be at most exponential in the number $E(G,G')$ of shared edges. 
    This will be ensured by a truncation involving the following ``local sparsity'' condition (where we will take $L$ large depending only on $(\mu,\eps)$).

    \begin{defi} \label{defi:L-sparsity}
        For $L\geq 10$ a positive integer, a shuffle graph $G$ is $\boldsymbol{L}$\textbf{-sparse} if within any discrete interval $\{i,i+1,\dots,i+L-1\}\subseteq [N]$ of $L$ consecutive vertices, at most $L/3$ (of the possible $L-1$) edges are in $E(G)$.
    \end{defi}

    \begin{lem}[\hspace{-0.05em}{\cite[Lemma~3.3]{mark2022cutoff}}]
    \label{lem:to-exp-moment}
        Suppose $G$ and $G'$ are $L$-sparse shuffle graphs. Then \[f_{G,G'}\leq (L!)^{|E(G,G')|}.\]
    \end{lem}

    To estimate the exponential moment, the idea will be to reduce to conditional first moment bounds via a form of Portenko's lemma \cite{portenko1975diffusion}.
    As in \cite{mark2022cutoff}, we explore $(G,G')$ one vertex at a time and argue that at any intermediate stage, the expected number of unrevealed shared edges is $o(1)$.
    However as observed there, this argument may break down near the end of the exploration (in the most extreme case, all remaining $Km$ digits could be $(k-1)$ when there are $m$ cards left).
    To mitigate this issue we label shared edges of $G$ and $G'$ as either ``forward" or ``backward" edges, such that every edge satisfies at least $1$ of the classifications (after truncating exceptional $G$ and $G'$); this allows each exploration to stop before getting too close to the end.
    Since in our analysis the number $n_l^{(1)}$ of strings starting with $[l]$ is deterministic, we in fact define ``forward" edges and ``backward" edges separately for strings starting with each $l \in [k]_0$.
    Moreover, since the number of digits equal to $0$ or $k-1$ may be unusually small in certain shuffles, unlike the setting in \cite{mark2022cutoff}, we introduce the notions of $\ileft^{(t)}$ and $\iright^{(t)}$ to ensure that the exploration process does not proceed too close to the endpoints.

    \begin{defi}\label{def:ileft-right}
        For all $t \in [K]$, we let
        \begin{equation}
            \ileft^{(t)}=\inf\Big\{l \geq 0:\frac{n^{(t)}_l}{N} \geq \chi\Big\},\quad 
             \iright^{(t)}=\sup\Big\{l \leq k-1:\frac{n^{(t)}_l}{N} \geq \chi\Big\}
        \end{equation}
        be the first and last piles with at least $\chi N$ cards in the $t$-th shuffle.
        Note that $\ileft^{(t)} < \iright^{(t)}$ if and only if $t$ is a $\chi$-good shuffle, and that if no pile sizes are small for shuffle $t$ then $\ileft^{(t)}=0$ and $\iright^{(t)}=k-1$.
    \end{defi}

    We can now introduce the ``forward" and ``backward" edges. For simplicity, we define
    \begin{equation}\label{eq:def-ak}
        a_1=1,\quad a_{k+1}=\inf\{t > a_k:\bn^{(t)}/N \notin \widetilde V_k\}.
    \end{equation}
    Note that $\ileft^{(t)} < \iright^{(t)}$ if and only if $t=a_k$ for some $k \geq 1$.
    Thus, by our assumption stated in Section~\ref{subsec:3-shuffle-reduction}, we have $a_3 \leq \mathsf{A}(\mu)$.

    \begin{defi}\label{def:s-for-back}
        For all $\p \in [k]_0$, we let 
        \[
             \mathsf S^\p_{\for}=\{x \in \cup_{a_3 \leq M \leq K }[k]_0^{M}:x[1]=\p \mbox{  and  } x[2]x[3]\cdots x[a_3] <_{\lex} \iright
             ^{(2)}\iright^{(3)}\cdots\iright^{(a_3)}\}.
        \]
        Let $E^\p_{\for}(G)$ consist of all edges $(i,i+1)\in E(G)$ for which $s_i=s_{i+1} \in \mathsf S^\p_{\for}$, and write 
        \[
        E^\p_{\for}(G,G')=E^\p_{\for}(G)\cap E^\p_{\for}(G').
        \]
        Similarly, for all $\p \in [k]_0$, we let
        \[
            \mathsf S^\p_{\back}=\{x \in \cup_{a_3 \leq M \leq K }[k]_0^{M}:x[1]=\p \mbox{  and  } x[2]x[3]\cdots x[a_3] >_{\lex} \ileft
             ^{(2)}\ileft^{(3)}\cdots\ileft^{(a_3)}\}.
        \]
        Define $E^\p_{\back}(G,G')$ in the same way but with $\mathsf S^\p_{\for}$ replaced by $\mathsf S^\p_{\back}$.
    \end{defi}

    Note that $E^\p_{\for}(G)\subset \big[\iota([\p]),\tau([\p])\big]$ is contained within an interval that does not depend on the shuffle graph $G$. 
    Therefore, $E^\p_{\for}(G) \cap E^{\p'}_{\for}(G') = \emptyset$ if $\p \neq \p'$. 
    In the next truncation condition we require the $3$ digits $a_1,a_2,a_3$ to have typical distribution, so that the forward and backward edges indeed cover all of $E(G,G')$.
    It is important for the truncation that this condition can be defined individually for $G$ and $G'$ separately rather than jointly.

    For simplicity, we let
    \begin{equation}
    \begin{aligned}
    \label{eq:lambda-lr}
        \lambda_{\text{left}}^\p &= \frac{1}{N}
        \mathbb E \lt|\lt\{i \in \big[\iota([\p]),\tau([\p])\big]:s_i[1]s_i[2]\cdots s_i[a_3] \leq_{\lex} \p\ileft
             ^{(2)}\ileft^{(3)}\cdots\ileft^{(a_3)}\rt\}\rt|;\\
        \lambda_{\text{right}}^\p &= \frac{1}{N}\mathbb E \lt|\lt\{i \in \big[\iota([\p]),\tau([\p])\big]:s_i[1]s_i[2]\cdots s_i[a_3] \geq_{\lex} \p\iright
             ^{(2)}\iright^{(3)}\cdots\iright^{(a_3)}\rt\}\rt|
    \end{aligned}
    \end{equation}

    \begin{defi} \label{def:regular}
        The sequence $S=(s_1,\dots,s_N)\in \cS$ of strings is $\vartheta$-\textbf{regular} if for all $\p \in [k]_0$
        \begin{align}\label{eq:def-regu-left}
            \lt|\lt\{i \in \big[\iota([\p]),\tau([\p])\big]:s_i[1]s_i[2]\cdots s_i[a_3] \leq_{\lex} \p\ileft^{(2)}\cdots\ileft^{(a_3)}\rt\}\rt| 
            &\leq
            N\Big( \lambda_{\mathrm{left}}^\p
            +\frac{\vartheta}{2}\frac{n_\p^{(1)}}{N}\Big);\\
        \label{eq:def-regu-right}
            \lt|\lt\{i \in \big[\iota([\p]),\tau([\p])\big]:s_i[1]s_i[2]\cdots s_i[a_3] \geq_{\lex} \p\iright^{(2)}\cdots\iright^{(a_3)}\rt\}\rt| 
            &\leq 
            N\Big(\lambda_{\mathrm{right}}^\p+\frac{\vartheta}{2}\frac{n_\p^{(1)}}{N}\Big).
        \end{align}
        for some $\vartheta>0$.
    \end{defi}

    \begin{lem}\label{lem:for-back-cover}
        If $S,S'\in \cS$ are $\vartheta$-regular for some $\vartheta>0$ such that for all $\p \in [k]_0$,  
        \begin{equation} \label{eq:exsit-mid}
            \lambda_{\mathrm{left}}^\p
             +\lambda_{\mathrm{right}}^\p< (1-\vartheta)\frac{n_\p^{(1)}}{N},    
        \end{equation} 
             then \[|E(G,G')|\leq \sum_{\p=0}^{k-1}\left[|E^\p_{\for}(G,G')|+|E^\p_{\back}(G,G')|\right]\] 
    \end{lem}
    \begin{proof}
        For simplicity, we let $E^\p(G,G')$ denote the shared edges of $G$ and $G'$ in $\big[\iota([\p]),\tau([\p])\big]$.
        Regularity implies that $E^\p_{\for}(G,G')$ contains all shared edges $(i,i+1)\in E^\p(G,G')$ with 
        \begin{equation}\label{eq:for-right-end}
            i\leq \tau([\p])-N\lt(\lambda_{\text{right}}^\p+\frac{\vartheta}{2}\frac{n_\p^{(1)}}{N}\rt), 
        \end{equation}
        and $E^\p_{\back}(G,G')$ contains all shared edges $(i,i+1)\in E^\p(G,G')$ with 
        \begin{equation}\label{eq:back-left-end}
            i\geq \iota([\p]) + N\lt(\lambda_{\mathrm{left}}^\p+\frac{\vartheta}{2}\frac{n_\p^{(1)}}{N}\rt).
        \end{equation}
        Thus, we have
        \begin{equation}\label{eq:compare-forback-end}
        \begin{aligned}
            &\Big(\tau([\p])-N\Big(\lambda_{\text{right}}^\p+\frac{\vartheta}{2}\frac{n_\p^{(1)}}{N}\Big)\Big) 
             -
             \Big(\iota([\p]) + N\Big(\lambda_{\mathrm{left}}^\p+\frac{\vartheta}{2}\frac{n_\p^{(1)}}{N}\Big) \Big)\\
             &=
             N\Big((1-\vartheta)\frac{n_\p^{(1)}}{N}-\lambda_{\text{left}}^\p-\lambda_{\text{right}}^\p\Big)\geq
             0.
        \end{aligned}
        \end{equation}
        Combining this with \eqref{eq:for-right-end} and \eqref{eq:back-left-end},
         we obtain 
        \[E^\p(G,G')\subset E^\p_{\for}(G,G')\cup E^\p_{\back}(G,G').\]
         The proof is then completed by observing that the following partition holds: 
         \[
         E(G,G')=\bigcup_{\p\in [k]_0}E^\p(G,G')
         .\qedhere
         \]
    \end{proof}

    We now specify our choice of $\vartheta$. Recall that $a_3\leq \mathsf{A}(\mu)$ as $N$ varies, where $\mathsf{A}(\mu)$ is a constant not depending on $N$. Let
    \begin{equation}
        \vartheta = \chi^{\mathsf{A}(\mu)}.
    \end{equation}

    \begin{lem}
    For all $\chi > 0$ and $\p \in [k]_0$, 
        \begin{equation}
        \lambda_{\mathrm{left}}^\p
             +\lambda_{\mathrm{right}}^\p< (1-\vartheta)\frac{n_\p^{(1)}}{N}.
        \end{equation}
    \end{lem}
    
    \begin{proof}
        By the definition of $\lambda_{\mathrm{left}}^l$ and $\lambda_{\mathrm{right}}^\p$ in \eqref{eq:lambda-lr}, we have
        \begin{equation}
           n_\p^{(1)}- N (\lambda_{\mathrm{left}}^\p
             +\lambda_{\mathrm{right}}^\p) 
            =
            \mathbb E \lt|i \in [N]:\p\ileft^{(2)}\cdots\ileft^{(a_3)}<_{\lex}s_i[1]s_i[2]\cdots s_i[a_3] <_{\lex}
            \p\iright^{(2)}\cdots\iright^{(a_3)}\rt|
        \end{equation}
        Recall the definition of $a_i$ in \eqref{eq:def-ak} and below, we have $\ileft^{(t)} < \iright^{(t)}$ if and only if $t=a_k$ for some $k \geq 1$.
        Therefore, we have
        \begin{equation}\label{eq:spe-str-lb}
            n_\p^{(1)}- N (\lambda_{\mathrm{left}}^\p
             +\lambda_{\mathrm{right}}^\p) 
             \geq 
             N \lambda_{\p\ileft^{(2)}\cdots \ileft^{(a_3-1)}\iright^{(a_3)}}.
        \end{equation}
        Note that 
        \begin{equation}
            \lambda_{\p\ileft^{(2)}\cdots \ileft^{(a_3-1)}\iright^{(a_3)}} = \frac{n^{(1)}_\p}{N} 
            \cdot \frac{n^{(a_3)}_{\iright^{(a_3)}}}{N}\cdot \prod_{2 \leq t \leq a_3-1} \frac{n^{(t)}_{\ileft^{(t)}}}{N} \overset{\eqref{eq:def-ak}}{\geq} \frac{n^{(1)}_\p}{N} \cdot \chi^{a_3-1} \geq \frac{n^{(1)}_\p}{N} \cdot \chi^{\mathsf{A}(\mu)}.
        \end{equation}
        Combining this with \eqref{eq:spe-str-lb} gives the desired result.
    \end{proof}

    \begin{lem} \label{lem:exp-moment}
        For all $K \geq (\oC_{\mu}+\varepsilon) \log N$, $t>0$, and $\p \in [k]_0$, there exists $\delta=\delta(\mu,\varepsilon)>0$ and two sequence sets $\cS_{\p}^{\for}$ and $\cS_{\p}^{\back}$ such that the following holds.
        \begin{itemize}
            \item[\emph{(a)}] The sequence sets $\cS_{\p}^{\for}$ and $\cS_{\p}^{\back}$ are typical, i.e.
            \[
            \P(S_K \in \cS_{\p}^{\for}) \geq 1- \exp(-N^{\delta}), \quad \P(S_K \in \cS_{\p}^{\back}) \geq 1- \exp(-N^{\delta}).
            \]
            \item[\emph{(b)}] The truncated exponential moment of number of shared edges is small: both of the inequalities
            \begin{equation} 
            \label{eq:exp-moment}
            \begin{aligned}
                \E(e^{t|E_{\for}^\p(G,G')|}\cdot 1_{S_K,S_K' \in \cS_{\p}^{\for}}) &\leq 1+ O(N^{-\delta}),\\ \E(e^{t|E_{\back}^\p(G,G')|}\cdot 1_{S_K,S_K' \in \cS_{\p}^{\back}}) &\leq 1+ O(N^{-\delta}),
            \end{aligned}
            \end{equation}
            hold for sufficiently large $N$.
            \item[\emph{(c)}] If the sequence $S_K \in \cS^{\dagger}$, where 
            \begin{equation}
                \cS^{\dagger}=\bigcap_{\p \in [k]_0} \lt( \cS_{\p}^{\for} \cap \cS_{\p}^{\back}\rt),
            \end{equation}
            then $S_K$ is $\vartheta$-regular.
        \end{itemize}
    \end{lem}

    \begin{lem}\label{lem:L-sparsity}
        For all $K \geq (\widetilde C_{\mu}+\varepsilon) \log N$, there exists $L=L(\mu,\varepsilon)$ such that $G(S)$ is $L$-sparse with probability at least $1-1/N$, for $N$ sufficiently large.
    \end{lem}

    We now give the proof of \eqref{eq:UpperBound} assuming Lemmas \ref{lem:exp-moment} and \ref{lem:L-sparsity}. We will complete the proof of Lemmas \ref{lem:exp-moment} and \ref{lem:L-sparsity} in Section \ref{sub2sec: expmoment}, after establishing a weaker first moment analog of \eqref{eq:exp-moment} in Section \ref{sub2sec: 1stmoment}.

    \begin{proof}[Proof of \eqref{eq:UpperBound}] 
        We take the integer $L$ in Lemma \ref{lem:L-sparsity}. Then we define $\cS^*$ to be the sequence set that consists of all $L$-sparse sequences, and we set $\cS^1=\cS^{\dagger}\cap \cS^*$, where  $\cS^{\dagger}$ is constructed in Lemma~\ref{lem:exp-moment} when $\vartheta=\widehat{\vartheta}$. Then we have
        \[\mathbb P(\cS^0)=O(1/N).\]
        By \eqref{eq:double-trick}, it suffices to upper-bound
        \[
            \mathbb E^{S,S'}\left[(f_{G,G'}-1)1_{S_K,S'_K\in \cS^1}\right]
        \]
        by a quantity which tends to $0$.
        Indeed we have 
        \[
        \begin{aligned}
            & \mathbb E^{S,S'}\left[(f_{G,G'}-1)1_{S_K,S'_K\in \cS^1}\right]\\
            \overset{\text{Lemma \ref{lem:to-exp-moment}}}&{\leq}\mathbb E^{S,S'}\left[((L!)^{|E(G,G')|}-1)1_{S_K,S'_K\in \cS^1}\right]\\
            \overset{\text{Lemma \ref{lem:for-back-cover}}}&{\leq} \mathbb E^{S,S'}\left[((L!)^{\sum_{\p\in[k]_0}|E^\p_{\for}(G,G')|+\sum_{\p\in[k]_0}|E^\p_{\back}(G,G')|}-1)1_{S_K,S'_K\in \cS^1}\right]\\
            &\leq 
            \frac{1}{2k}\mathbb E^{S,S'}\left[
            \Bigg(\sum_{\p\in[k]_0}
            \Big[((L!)^{2k|E^\p_{\for}(G,G')|}-1) + ((L!)^{2k|E^\p_{\back}(G,G')|}-1)\Big]
            \Bigg)
            1_{S_K,S'_K\in \cS^1}
            \right]\\
            \overset{\eqref{eq:exp-moment}}&{\leq} O(N^{-\delta}),
        \end{aligned} 
        \]
        where we use \eqref{eq:exp-moment} for $t=2k\log (L!)$ in the last line. Moreover, since $\delta$ from \eqref{eq:small-params} depends only on $\mu$ and $\varepsilon$, the upper bound is uniform in $\vec \bn$.
    \end{proof}

    \subsection{Upper bound for random pile sizes}
    \label{subsec:rand-pile-upperbound}
    In this subsection we extend the upper bound from deterministic to random pile sizes.
    Namely we will deduce \eqref{eq:random-UpperBound} (the second part of Theorem~\ref{thm:weak-convergence}) assuming the mixing time upper bound \eqref{eq:UpperBound} for deterministic cuts. 
    First we show that the shuffle process with pile sizes $(\bX^{(t)})_{t\leq K}$ is $(\chi,\rho,\varphi)$-almost-$\mu$-like with high probability.
    We also introduce the event
    \[
        \mathsf{Like}(\chi,\rho,\varphi)= \lt \{ \mbox{$(\bX^{(t)})_{t\leq K}$ is $(\chi,\rho,\varphi)$-almost-$\mu$-like} \rt\}.
    \]

    \begin{lem}\label{lem:num-q-like}
        Suppose that $(\bX^{(t)})_{t\in \mathbb Z^+}$ is randomized-$\mu$-like, we have
        \begin{equation} \label{eq:num-q-like}
            \mathbb P \lt( \mathsf{Like}(\chi,\rho,\varphi) \rt)=1-o(1).
        \end{equation}
    \end{lem}
    \begin{proof}
        Recall the definition of randomized-$\mu$-like in \eqref{def:random-mu-like}. We conclude immediately by the definition of weak convergence and testing \eqref{def:random-mu-like} with the indicator function $\mathbbm{1}_{D_i\times(\frac{t_*}{\log N}+\rho (j-1),\frac{t_*}{\log N}+\rho j]}$.
    \end{proof}

    We now obtain the mixing time upper bound \eqref{eq:random-UpperBound} in Theorem~\ref{thm:weak-convergence}:

    \begin{proof}[Proof of \eqref{eq:random-UpperBound}]
        For any shuffle process with pile sizes $\vec \bn$, we set
        \begin{equation}
            \begin{aligned}
                & \mbox{$\pi_{\vec \bn}$ is a random variable with law $\bP(\vec\bn,\lfloor(1+\varepsilon)\overline{C}_{\mu}\log(N)\rfloor)$;} \\
                & \mbox{$\sigma_{\vec \bn}$ is a random variable with law $\nu_{\Unif}$.}
            \end{aligned}
        \end{equation}
        By definition of the total variation distance, we can construct a coupling $\mathbf Q_{\vec \bn}$ of these two random variables such that under $\mathbf Q_{\vec \bn}$,
        \begin{equation}
            \mathbf Q_{\vec \bn}(\pi_{\vec \bn} \neq \sigma_{\vec \bn}) \leq d_{\TV}
            \big(\bP(\vec\bn,\lfloor(1+\varepsilon)\overline{C}_{\mu}\log(N)\rfloor),\nu_{\Unif}\big) +\frac{1}{N}.
        \end{equation}
        We now sample $(\pi_{\vec \bn} , \sigma_{\vec \bn})$ according to $\mathbf Q_{\vec \bn}$ independently, and sample the (random) pile sizes $\bX=(\bX^{(t)})_{t\in \mathbb Z^+}$ independent of $(\pi_{\vec \bn} , \sigma_{\vec \bn})$. 
        We now have
        \begin{equation}
            \begin{aligned}
                & \mbox{$\pi_{\bX}\overset{\text{def.}}{=}\sum_{\vec \bn}\pi_{\vec \bn}\cdot 1_{\{\bX=\vec \bn\}}$ is a random variable with law $\bP(\bX,\lfloor(1+\varepsilon)\overline{C}_{\mu}\log(N)\rfloor)$;} \\
                & \mbox{$\sigma_{\bX}\overset{\text{def.}}{=}\sum_{\vec \bn}\sigma_{\vec \bn}\cdot 1_{\{\bX=\vec \bn\}}$ is a random variable with law $\nu_{\Unif}$,}
            \end{aligned}
        \end{equation}
        where the sum is over all possible series of pile sizes $\vec \bn$.
        Therefore, we have
        \begin{equation}
        \begin{aligned}
            &\mathbb P(\pi_{\bX} \neq \sigma_{\bX}) \leq \mathbb P\lt((\mathsf{Like}(\chi,\rho,\varphi))^c\rt) + \sum_{\vec \bn \text{ is $(\chi,\rho,\varphi)$-almost-$\mu$-like}} \mathbb P(\bX=\vec\bn) \cdot \mathbf Q_{\vec \bn}(\pi_{\vec \bn} \neq \sigma_{\vec \bn})\\
            &\leq \mathbb P\lt((\mathsf{Like}(\chi,\rho,\varphi))^c\rt)+ \mathbb  \sup_{\vec \bn \text{ is $(\chi,\rho,\varphi)$-almost-$\mu$-like}}d_{\TV}\big(\bP(\vec\bn,\lfloor(1+\varepsilon)\overline{C}_{\mu}\log(N)\rfloor),\nu_{\Unif}\big) +\frac{1}{N} 
        \end{aligned}
        \end{equation}
        for all $\chi,\rho,\varphi>0$.
        We now apply Theorem~\ref{thm:main} (only \eqref{eq:UpperBound}) and Lemma~\ref{lem:num-q-like} to obtain \eqref{eq:random-UpperBound}.
    \end{proof}

\section{The first moment}\label{sub2sec: 1stmoment} 
    In this subsection we upper-bound the first moment $\mathbb E[|E(G,G')|]$.
    This estimate will later be applied conditionally to a subset of the full shuffle graphs (see the last step of Lemma~\ref{lem:unreveal-edges}), which is still approximately described by $\mu$. 
    The next technical definition will ensure that our first moment estimates can be applied in such a conditional setting.

    \begin{defi}\label{weak-almost}
        We say a shuffling process $\bn$ is \textbf{weak-$c$-$(\chi,\rho,\varphi)$-almost $\mu$-like} if there exists another shuffling process $\hat\bn$ which is $(\chi,\rho,\varphi)$-almost $\mu$-like such that
        \begin{equation}
        \label{eq:def:weak-almost-like}
            \left|
            \bn^{(t)}-\hat \bn^{(t)}
            \right|<N^{1-c},\quad 
            \forall t \in \mathsf{P}.
        \end{equation}
    \end{defi}
    In this Subsection we will always assume that $\bn$ is weak-$c$-$(\chi,\rho,\varphi)$-almost $\mu$-like, and in this subsection we define $\mathsf P_i$ with respect to $\hat\bn$. Recalling \eqref{eq:small-params}, we emphasize that $N^{-c}\ll\varphi,\chi,\rho$.
    We will prove the following crucial bound for the first moment; note that the quantiative rate $O(N^{-\delta})$ is important since we will later sum over $O(N^{\rho})$ applications of this bound.
    
    \begin{prop}\label{prop:new-firstmoment}
    Fix $c\in (0,1)$ and $\eps>0$. If $K\geq (\overline C_{\mu}+\varepsilon)\log N$, then there exists $\delta=\delta(\eps,c)$ and $\rho=\rho(\delta,c),\varphi=\varphi(\delta,c)$ such that the following holds. For all weak-$c$-$(\chi,\rho,\varphi)$-almost $\mu$-like shuffle process $\bn$, if the first shuffle is $\chi$-good, we have
    \begin{equation} \label{eq:new-firstmoment}
        \mathbb E|E(G,G^\prime)|\leq O(N^{-\delta}).
    \end{equation}
    \end{prop}

    Recall that we will choose parameters in the order \eqref{eq:small-params} and a large constant $\mathsf C$ satisfying \eqref{eq:mathsfC}.

    \subsection{Fluctuation of the interval $\mathcal I(B_x)$}\label{subsec:fluc}
    It is very useful to study the fluctuation of the interval $\mathcal I(B_x)$ (recall \eqref{eq:I-Bx}). To this end, we introduce the following notations.
    \begin{defi}\label{def:class}
    We define the following classes of prefixes:
    \begin{equation}
        \begin{aligned}
        \mathcal A_{\rmleft}&=\bigcup_{1<M\leq K}\left\{x\in [k]_0^M: x[M]>\ileft^{(t)} ,\, x[t]\leq \ileft^{(t)},\, \forall 1<t<M
        \right\} \\
        &\quad\quad\quad\quad\bigcup \quad
        \left\{x\in [k]_0^K: x[t]\leq \ileft^{(t)},\, \forall 1<t\leq K\right\},\\
        \mathcal A_{\rmright}&=\bigcup_{1<M<K}\left\{x\in [k]_0^M: x[M]<\iright^{(t)} ,\, x[t]\geq \iright^{(t)},\, \forall 1<t<M
        \right\} \\
        &\quad\quad\quad\quad\bigcup \quad
        \left\{x\in [k]_0^M: x[t]\geq \iright^{(t)},\, \forall 1<t\leq K\right\}.
    \end{aligned}
    \end{equation}
    Moreover, we define
    \begin{equation}\label{eq:def:admissible-prefixes}
        \begin{aligned}
        \widetilde {\mathcal A}_{\rmleft}&=
        \Big\{x\in \mathcal A_{\rmleft}:|\{t:x[t]< \ileft^{(t)}\}|<\frac{\mathsf{C}}{\log(1/\chi)}\log N\Big\},
        \\
        \widetilde{\mathcal A}_{\rmright}&=
        \Big\{x\in \mathcal A_{\rmright}^{(t)}:|\{t:x[t]> \iright^{(t)}\}|<\frac{\mathsf{C}}{\log(1/\chi)}\log N\Big\}.
    \end{aligned}
    \end{equation}
    \end{defi}

    We now define the quantity $c_F(x)$ which characterizes the order of fluctuations for location of the interval $\mathcal I(B_x)$.  
    For any $y\in \mathcal A_{\rmleft}$ with length $M_y$ and $z\in \mathcal A_{\rmright}$ with length $M_z$, we define
        \begin{equation}\label{def-Ayz}
            A_{y,z}=\Big\{x\in \bigcup_{M\geq 1}[k]_0^M: x \mbox{ has prefixes }y \mbox{ and }z\Big\}.
        \end{equation}
    Of course $A_{y,z}$ is only non-trivial when $y$ is a prefix of $z$ or vice-versa.
    When $A_{y,z}$ is non-trivial, define 
    \begin{equation}
        y\vee z =\begin{cases}
            y, \mbox{ if } M_y>M_z;
            \\
            z, \mbox{ if } M_y\leq M_z.
        \end{cases}
    \end{equation}
    Let $M_x$ be the length of $x$.
    For $x\in A_{y,z}$, where $y\in \mathcal A_{\rmleft}$ has length $M_y$ and $z\in \mathcal A_{\rmright}$ has length $M_z$, we have the following decomposition:
    \begin{equation} \label{eq:decom-left-int}
       \big\{s\in [k]_0^{M}:s<x\big\}=\bigcup_{1\leq i\leq M_x}P_{x,i},
    \end{equation}
    where 
    \begin{equation} \label{eq:def:left-strings}
    P_{x,i}=\big\{s\in [k]_0^{M_x}:s[j]=x[j] \mbox{ for all } j<i \mbox{ and } s[i]<x[i]\big\}.
    \end{equation}
    With these notations, we define $\Fluc_{\rmleft}(x)=(\max_{2\leq i\leq M_y}\mathbb E|P_{x,i}|)^{\frac{1}{2}}$, which represents the fluctuation seen from the left side. Similarly, we can define \begin{equation} \label{eq:decom-right-int}
       \big\{s\in [k]_0^{M_x}:s>x\big\}=\bigcup_{1\leq i\leq M}Q_{x,i}, 
    \end{equation}
    where 
    \begin{equation} 
    Q_{x,i}=\big\{s\in [k]_0^{M_x}:s[j]=x[j] \mbox{ for all } j<i \mbox{ and } s[i]>x[i]\big\},
    \end{equation}
    and define $\Fluc_{\rmright}(x)=(\max_{2\leq i\leq M_z}\mathbb E|Q_{x,i}|)^{\frac{1}{2}}$.
    With these notations, we are able to define 
    \begin{align}
    \label{eq:def:c_F}
        c_F(x)\equiv \min \big(\log_N \Fluc_{\rmleft}(x), \log_N \Fluc_{\rmright}(x)\big).
    \end{align}
    It is useful to point out that
    \begin{equation}\label{cF-cL-bound-1}
        \mathbb E|\mathcal I(B_x)|
        \leq 
        \frac{\min(\E|P_{x,M_y}|,\E|Q_{x,M_z}|)}{\chi}
        \leq 
        N^{2c_F(x)+o(1)}.
    \end{equation}

For convenience, we summarize some concentration results here.
\begin{lem}\label{lem:concentration}
    For any $y\in \mathcal A_{\rmleft}$, $z\in \mathcal A_{\rmright}$  and $x\in A_{y,z}$,
    if $\mathbb E |\cI(B_x)| < N^ {c_F+\delta}$, then for any $a>0$ independent of $N$:
    \begin{align}
	\label{eq:positionconcentration}
	\mathbb P\left[\big| \iota(x)-Nt_x]\big| \geq  N^{ \max(c_F,\delta)+\frac{a}{2}+\delta}\right]
	&\leq e^{-\Omega(N^{a-o(1)})},\\
	\label{eq:poscon2}
	\mathbb P\left[\big| \tau(x)-N(t_x+\lambda_x)\big| \geq N^{\max(c_F,\delta)+\frac{a}{2}+\delta}\right]
	&\leq e^{-\Omega(N^{a-o(1)})}.
	\end{align}
\end{lem}
\begin{proof}
    We first prove
    \begin{equation}
    \label{eq:sizeconcentration}
	\mathbb P\left[\Big| |\cI(B_x)|-\mathbb E|\cI(B_x)|\Big| \geq N^{\max(c_F,\delta)+\frac{a}{2}+\delta}\right]\leq e^{-\Omega(N^{a-o(1)})},
    \end{equation}
    When $\mathbb E|\cI(B_x)|\geq N^{\delta}$, \eqref{eq:sizeconcentration} holds immediately by \eqref{concen-prefix-ineq} with $q=\log_N \mathbb E|\cI(B_x)|>\delta$. Otherwise \eqref{eq:sizeconcentration} holds by Lemma \ref{lem:concen-block-smallmean}.
    For \eqref{eq:positionconcentration} and \eqref{eq:poscon2}, suppose $c_F(x)=\log_N \Fluc_\rmleft(x)$ without loss of generality.
    Recall that 
    \[
        \iota(x)=1+|\{i\in [N]:s_i<_{\lex}x\}|, \quad \{s\in [k]_0^{M}:s<_{\lex}x\}=\bigcup_{1\leq i\leq M_x}P_{x,i},
    \]
    where $P_{x,i}$ is defined in \eqref{eq:def:left-strings}.
    Now we have
    \begin{equation}
        \{s\in [k]_0^{M_y}:s<_\lex x\}=\{s:s[1]<_\lex x[1]\}\cup\bigcup_{2\leq i\leq M_x}P_{x,i}.
    \end{equation}
    Thus, we obtain
    \begin{equation}\label{eq:decomp-iota}
        \iota(x)=S_{x[1]-1}+\sum_{2\leq i\leq M_x}|P_{x,i}|.
    \end{equation}
    By definition of $\Fluc_{\rmleft}(x)$, we have $\mathbb E|P_{x,i}| \leq (\Fluc_{\rmleft}(x))^2$ for all $2 \leq i \leq M_y$. Moreover, for $M_y<i\leq M_x$, by \eqref{cF-cL-bound-1} we have 
    \[
    \mathbb E|P_{x,i}| \leq \mathbb E|\mathcal I(B_{y})|\leq N^{2c_F(y)+o(1)}=N^{2c_F(x)+o(1)}.
    \]
    Therefore, \eqref{eq:positionconcentration} holds by Lemmas \ref{concen-prefix}, \ref{lem:concen-block-smallmean} and a union bound. 
    For \eqref{eq:poscon2}, note that
    \[\tau(x)=\iota(x)+|\mathcal I(B_x)|, \quad \mathbb E|\mathcal I(B_x)|=N\lambda_x,\]
    \eqref{eq:poscon2} then holds by a combination of \eqref{eq:sizeconcentration} and \eqref{eq:positionconcentration}.
\end{proof}

Now we are ready to lower bound the fluctuations of the location of $|\mathcal I(B_x)|$.
The following lemma states that the endpoints of $\cI(B_x)$ do not concentrate on a small set, even under conditioning.
In \cite[Lemma 4.13]{mark2022cutoff}, the corresponding step was implemented by simply conditioning on the multi-set
\begin{equation}
\label{eq:U-x}
U_x=\{s\in S_K: s \mbox{ has prefix }x\}.
\end{equation}
This determines the size and ``internal structure'' of $\cI(B_x)$, but not its location.
For multinomial cuts, given $U_x$, the left endpoint $\iota(x)$ of $\cI(B_x)$ is conditionally binomial, so it is easy to analyze the homogenizing effect of this random shift.
For deterministic $\vbn$, it appears quite complicated to resample the entirety of the strings outside $U_x$. Instead we will resample only carefully chosen restricted information.

\begin{lem}\label{lem:anticoncen}
    For any $y\in \al, z\in \ar$ and $x\in A_{y,z}$, there exists $\mathcal S'$ with $\mathbb P(\mathcal S')\geq 1-\exp(-N^{\delta})$ such that the following holds.
    \begin{equation}
         \max_{j\in [N]}\mathbb P(\iota(x)=j,S_K\in\cS'|U_x) =O(N^{-c_F(x)+2\delta}),
    \end{equation}
    \begin{equation}
         \max_{j\in [N]}\mathbb P(\tau(x)=j,S_K\in\cS'|U_x) =O(N^{-c_F(x)+2\delta}).
    \end{equation}
\end{lem}

Before proving Lemma \ref{lem:anticoncen}, it is useful to state the anti-concentration estimate for hypergeometric variables. See e.g. \cite[Eq.~25 and Eq.~38]{pitman1997probabilistic} for a proof.
We note that our application of this estimate loses an ultimately irrelevant $\log N$ factor (from the depth of the prefix tree) as compared to the multinomial case, which is crucial to handle small pile sizes.

\begin{lem}\label{lem:anticoncen-hyper}
    Let $X\sim \Hyp(n_1,n,m)$ and let $\sigma^2=\frac{n_1m(n-n_1)(n-m)}{n^2(n^2-1)}$ be the variance of $X$. There exists an absolute constant $C$ such that for each $k\geq 0$,
    \begin{equation}\label{eq:anticoncen-hyper}
        \mathbb P(X=k)\leq \frac{C}{\sigma}.
    \end{equation}
\end{lem}

\begin{proof}[Proof of Lemma \ref{lem:anticoncen}.]
    When $c_F\leq2\delta$, the statements hold immediately. Now we assume that $c_F>2\delta$. 
    Our proof strategy is to condition on a $\sigma$-field that includes $U_x$, and to use the remaining randomness to obtain a ``shift'' of $U_x$. In fact, we plan to use the randomness of just one shuffle. 
    
    For any $y\in \mathcal A_{\rmleft}$ with length $M_y$ and $z\in \mathcal A_{\rmright}$ with length $M_z$, by definition of $c_F(x)$, there exist $2\leq u\leq M_y$ and $2\leq v\leq M_z$, such that $\Fluc_{\rmleft}(x)=(\mathbb E|P_{x,u}|)^{\frac{1}{2}}$ and $\Fluc_{\rmright}(x)=(\mathbb E|Q_{x,v}|)^{\frac{1}{2}}$. Without loss of generality, we suppose that $c_F(x)=\log_N \Fluc_{\rmleft}(x)$. To simplify notations, we write $N'=n^{(1)}_{x[1]}$. Since the first shuffle is $\chi$-good, we have $N'\leq (1-\chi)N$.
    Recall the partitions \eqref{eq:decom-left-int} and \eqref{eq:decom-right-int}, we have
    \begin{equation}
        N'=\sum_{i=2}^{M_x} \mathbb E|P_{x,i}| + \sum_{i=2}^{M_x} \mathbb E|Q_{x,i}| + \mathbb E |\cI(B_x)| \leq 2M_x \mathbb E |Q_{x,v}| + \chi ^{-1} \mathbb E|Q_{x,v}|.
    \end{equation}
    Therefore, we have
    \begin{equation}\label{eq:pigeon-hole}
        \mathbb E |Q_{x,v}| \geq \Omega\Big(\frac{N'}{\log N}\Big).
    \end{equation}
    We will divide the proof into three cases.

    \paragraph{Case $1$: $u>v$.}
    Recall the definition of a shuffle matrix in Section~\ref{subsec:shuffle-graph}. We first sample all columns except the $v$-th column.
    Let $R_1\subseteq [N]$ be the set of $i\in [N]$ such that the $i$-th row $\bar s_i$ satisfies
    \[
    \bar s[1:v-1]=x[1:v-1] \mbox{ and } \bar s_i[v+1:u]\geq x[v+1:u].
    \]
    Since we have sampled all columns except the $v$-th column, $R_1$ is fixed. Furthermore, we sampled all digits in the $v$-th column and in rows which belong to $R_1$. We write $\mathcal F_x$ for the $\sigma$-field associated with this conditioning. We see that $U_x$ is $\mathcal F_x$ measurable.
    Now we define
    \begin{align*}
        a_1&=|\{i:\bar s_i[1:v-1]<x[1:v-1]\}|;\\
        a_2&=|\{i:\bar s_i[1:v-1]=x[1] \cdots x[v-1] \mbox{ and } \bar s_i[v+1:u]<x[v+1:u]\}|;\\
        a_3&=|\{i\in R_1: \bar s_i[v]\leq x[v]\}|;\\
        a_4&=|\{i:\bar s[1:v-1]=x[1:v-1], \bar s_i[v]<x[v] \mbox{ and } \bar s_i[v+1:u]\geq x[v+1:u]\}|\\
        &\quad +|\{i:\bar s[1:u]=x[1:u] \mbox{ and } s[u+1:K]<x[u+1:K]\}|.
    \end{align*}
    Notice that $a_1,a_2$ do not depend on the $v$-th shuffle, and $a_3,a_4$ depend only on rows in $R_1$. Therefore, $a_1,a_2,a_3,a_4$ are all $\mathcal F_x$ measurable. Also define
    \begin{equation}
        Y_1=|\{i:\bar s[1:v-1]=x[1:v-1], \,\bar s_i[v]\leq x[v], \mbox{ and } \bar s_i[v+1:u]<x[v+1:u]\}|
    \end{equation}
    Now, conditioned on $\mathcal F_x$, we see that 
    \begin{equation}
        \iota(x)=a_1 + a_4 + Y_1 \,,\quad  \tau(x) =a_1 + a_4 + Y_1 + |U_x| \,, \mbox{ and } Y_1\sim \Hyp (N-|R_1|, \sum_{l\leq x[v]} n_l^{(v)} - a_3, a_2 ).
    \end{equation}
    Define the event 
    \begin{align*}
        \mathcal S'_1=&\lt\{\big||R_1|-\mathbb E|R_1|\big|\leq \max((\mathbb E|R_1|)^{1/2+\delta}, N^{\delta})\rt\}\\
        &\qquad\qquad\cap \lt\{\lt|a_3-\mathbb Ea_3\rt|\leq \max((\mathbb Ea_3)^{1/2+\delta}, N^{\delta})\rt\} \cap \lt\{ |a_2-\mathbb Ea_2|\leq (\mathbb Ea_2)^{1/2+\delta}\rt\}
    \end{align*}
    We have
    \begin{align*}
        &\frac{\sum_{l\leq x[v]} n_l^{(v)}}{N}= 1-\frac{\sum_{l>x[v]} n_l^{(v)}}{N}\leq 1- \frac{\mathbb E|Q_{x,v}|}{N}\leq 1-\Omega\Big(\frac{1}{\log N}\Big);\\
        &\mathbb E a_3=\mathbb E|R_1|\cdot \frac{\sum_{l\leq x[v]} n_l^{(v)}}{N};\\
        &a_2+|R_1|\leq N'\leq (1-\chi)N ;\\
        &\mathbb Ea_2\cdot \frac{\sum_{l\leq x[v]} n_l^{(v)}}{N}=\mathbb EY_1\geq \mathbb E|P_{x,u}|=N^{2c_F(x)}.
    \end{align*}
    By \eqref{concen-prefix-ineq} and \eqref{eq:concen-hypergeo-smallmean}, we have $\mathbb P(\mathcal S_1')\geq 1-\exp(-N^{\delta})$ since $\mathbb Ea_2\geq N^{2c_F}$.
    Therefore on the event $\mathcal S'_1$, 
    \begin{align*}
        a_2&\leq c(\chi)\cdot (N-|R_1|);\\
        \mathbb E(Y_1|\mathcal F_x) &= a_2\cdot \frac{\sum_{l\leq x[v]} n_l^{(v)} - a_3}{N-|R_1|}\\
        &=a_2\cdot \lt(\frac{\sum_{l\leq x[v]} n_l^{(v)}}{N}+\frac{\sum_{l\leq x[v]} n_l^{(v)}}{N(N-|R_1|)}(|R_1|-\mathbb E |R_1|)+\frac{\mathbb Ea_3-a_3}{N-|R_1|}\rt)\geq N^{2c_F(x)-\delta}.
    \end{align*}
    We then conclude by \eqref{eq:anticoncen-hyper}.

    \paragraph{Case $2$: $u<v$.} Similarly to Case 1, we first sample all columns except the $u$-th column.
    Let $R_2\subseteq [N]$ be the set of $i\in [N]$ such that the $i$-th row $\bar s_i$ satisfies
    \[
    \bar s[1:u-1]=x[1:u-1] \mbox{ and } \bar s_i[u+1:v]\leq  x[u+1:v].
    \]
    Since we have sampled all columns except the $u$-th column, $R_2$ is fixed. Furthermore, we sampled all digits in the $u$-th column and in rows which belong to $R_2$. We write $\mathcal G_x$ for the $\sigma$-field associated with this conditioning. We see that $U_x$ is $\mathcal G_x$ measurable.
    Now we define
    \begin{align*}
        b_1&=|\{i:\bar s_i[1:u-1]<x[1:u-1]\}|;\\
        b_2&=|\{i:\bar s_i[1:u-1]=x[1:u-1] \mbox{ and } \bar s_i[u+1:v]>x[u+1:v]\}|;\\
        b_3&=|\{i\in R_2:\bar s_i[u]<x[u]\}|;\\
        b_4&=|\{i:\bar s_i[1:u-1]=x[1:u-1], \bar s_i[u]\leq x[u] \mbox{ and } \bar s_i[u+1:v]<x[u+1:v]\}|\\
        &\quad +|\{i:\bar s_i[1:u-1]=x[1:u-1], \bar s_i[u]< x[u] \mbox{ and } \bar s_i[u+1:v]=x[u+1:v]\}|\\
        &\quad + |\{i:\bar s_i[1:v]=x[1:v] \mbox{ and } \bar s_i[v+1:K]<x[v+1:K]\}|
    \end{align*}
    Notice that $b_1,b_2$ do not depend on the $u$-th shuffle, and $b_3,b_4$ depend only on rows in $R_2$. Therefore, $b_1,b_2,b_3,b_4$ are all $\mathcal G_x$ measurable. Also define
    \begin{equation}
        Y_2=|\{i:\bar s_i[1:u-1]=x[1:u-1], \bar s_i[u]<x[u]\mbox{ and } \bar s_i[u+1:v]>x[u+1:v]\}|.
    \end{equation}
    Now, conditioned on $\mathcal G_x$, we see that 
    \begin{equation}
        \iota(x)= b_1+ b_4 +Y_2\,, \tau(x) = b_1+b_4 +Y_2 + |U_x| \,, \mbox{ and } Y_2 \sim \Hyp (N-|R_2|, \sum_{l < x[u]} n_l^{(u)} - b_3, b_2 ).
    \end{equation}
    Define the event
    \begin{align*}
        \mathcal S'_2=&\lt\{\big||R_2|-\mathbb E|R_2|\big|\leq \max((\mathbb E|R_2|)^{1/2+\delta}, N^{\delta})\rt\}\\
        &\qquad\qquad\cap\lt\{\lt|b_3-\mathbb Eb_3\rt|\leq \max((\mathbb Eb_3)^{1/2+\delta}, N^{\delta})\rt\}\cap \lt\{|b_2-\mathbb Eb_2|\leq (\mathbb Eb_2)^{1/2+\delta}\rt\}.
    \end{align*}
    By \eqref{concen-prefix-ineq} and \eqref{eq:concen-hypergeo-smallmean}, we have $\mathbb P(\mathcal S_2')\geq 1-\exp(-N^{\delta})$. Moreover, we have
    \begin{align*}
        &\frac{\sum_{l<x[u]} n_l^{(u)}}{N}\leq 1-\frac{n_{x[u]}^{(u)}}{N}\leq 1-\frac{\mathbb E|Q_{x,v}|}{N'}\leq 1-\Omega\Big(\frac{1}{\log N}\Big);\\
        &\mathbb E b_3=\mathbb E|R_2|\cdot \frac{\sum_{l<x[u]} n_l^{(u)}}{N};\\
        &b_2+|R_2|\leq N' \leq (1-\chi)N ;\\
        &\mathbb Eb_2\cdot \frac{\sum_{l< x[u]} n_l^{(u)}}{N}\geq \mathbb E|Q_{x,v}|\cdot \frac{\sum_{l< x[u]} n_l^{(u)}}{N}\geq \Omega\Big(\frac{N'}{\log N}\Big)\cdot \frac{\sum_{l< x[u]} n_l^{(u)}}{N}=N'\cdot \frac{\sum_{l< x[u]} n_l^{(u)}}{N} \cdot \Omega\Big(\frac{1}{\log N}\Big)\\
        &\hspace{25em}\geq \mathbb E|P_{x,u}|\cdot \Omega\Big(\frac{1}{\log N}\Big)\geq N^{2c_F(x)-\delta}.
    \end{align*}
    Therefore, on the event $\mathcal S'_2$, we have
    \begin{align*}
        &b_2\leq c(\chi)\cdot (N-|R_2|);\\
        &\mathbb E(Y_2|\mathcal G_x) = b_2\cdot \frac{\sum_{l< x[u]} n_l^{(u)} - b_3}{N-|R_2|}\\
        &\qquad=b_2\cdot \lt(\frac{\sum_{l< x[u]} n_l^{(u)}}{N}+\frac{\sum_{l< x[u]} n_l^{(u)}}{N(N-|R_2|)}(|R_2|-\mathbb E |R_2|)+\frac{\mathbb Eb_3-b_3}{N-|R_2|}\rt)\geq N^{2c_F(x)-\delta}.
    \end{align*}
    We then conclude by \eqref{eq:anticoncen-hyper}.

    \paragraph{Case $3$: $u=v$.} We sample all columns except the $v$-th column. Furthermore, we sample all the $x[v]$-digit in the $v$-th column. We write $\mathcal H_x$ for the $\sigma$-field associated with this conditioning. Define
    \begin{align*}
        d_1&=|\{i:\bar s_i[1:v-1]<x[1:v-1]\}|;\\
        d_2&=|\{i:\bar s_i[1:v-1]=x[1:v-1], \bar s_i[v]\neq x[v]\}|;\\
        d_3&=|\{i:\bar s_i[1:v]=x[1:v],\bar s_i[v+1:K]<x[v+1:K]\}|.
    \end{align*}
    Then $d_1,d_2,d_3$ are all $\mathcal H_x$-measurable. Also define
    \begin{equation}
        Y_3=|\{i:\bar s_i[1:v-1]=x[1:v-1] \mbox{ and } \bar s_i[v]<x[v]\}|.
    \end{equation}
    Now, conditioned on $\mathcal H_x$, we see that
    \begin{equation}
        \iota(x)= d_1+d_3 + Y_3\,,\quad  \tau(x) = d_1+d_3 + Y_3 + |U_x| \,, \mbox{ and } 
        Y_3 \sim 
        \Hyp \big(N-n^{(v)}_{x[v]}, \sum_{l < x[u]} n_l^{(v)}, d_2 \big).
    \end{equation}
    Define the event $\mathcal S_3'=\{|d_2-\mathbb Ed_2|\leq (\mathbb Ed_2)^{1/2+\delta}\}$. Since $\mathbb Ed_2\geq \mathbb E|P_{x,u}|=N^{2c_F(x)}\geq N^{2\delta}$, we have $\mathbb P(\mathcal S_3')\geq 1-\exp(-N^{\delta})$.
    Since $\mathbb E|P_{x,u}|\leq \mathbb E |Q_{x,v}|$, on the event $\mathcal S_3'$ we have
    \begin{align*}
        &\sum_{l < x[u]} n_l^{(v)}\leq \frac{1}{2}(N-n^{(v)}_{x[v]});\\
        &d_2\leq \mathbb Ed_2 +(\mathbb Ed_2)^{1/2+\delta}\leq \frac{N'}{N}\cdot( N-n^{(v)}_{x[v]}) + (\frac{N'}{N}\cdot( N-n^{(v)}_{x[v]}))^{1/2+\delta}\leq (1-c(\chi))\cdot ( N-n^{(v)}_{x[v]}).\\
        &\mathbb E(Y_3|\mathcal H_x)= d_2\cdot \frac{\sum_{l < x[u]} n_l^{(v)}}{N-n^{(v)}_{x[v]}}\geq c(\chi)\cdot\mathbb E d_2\cdot \frac{\sum_{l < x[u]} n_l^{(v)}}{N} \geq c(\chi) \cdot \mathbb E|P_{x,u}|\geq N^{2c_F(x)-\delta}.
    \end{align*}
    We then conclude by \eqref{eq:anticoncen-hyper}. By taking $\cS'=\cS_1'\cap\cS_2'\cap\cS_3'$, we finish the proof.
    \end{proof}

    \subsection{Digit profiles}
    \begin{defi}[digit profile]\label{def:digit-profile}
        Recall the definition of $\mathsf Q$ in \eqref{eq:P-Q-def} and definition of $T$ in \eqref{def:set-t}. For $y\in \mathcal A_{\rmleft}$ and $z\in \mathcal A_{\rmright}$ and $x\in A_{y,z}$, the \textbf{digit profile} of $x$ is
        \begin{equation}
            \lt(\{c_l^{(i)}(x)\}_{(i,l)\in T},\{(t,x[t]):(t,x[t])\in \mathsf Q\}\rt),
        \end{equation}
        where $c_l^{(i)}(x)\in \mathbb Z/\log N$ is such that after the first $\max(M_y, M_z)$ digits, $x$ contains $c_l^{(i)}\log N$ $l$-digits in $\mathsf P_i$.
    \end{defi}

    Each digit profile will contribute a certain number of expected shared edges, of order a power of $N$.
    Our strategy is to prove that the number of possible digit profiles (via a truncation) is bounded by a small power of $N$, and control the maximal contribution from any single digit profile.
    Rescaling by $\log N$ is convenient because the optimal digit profiles will come from solving a continuous optimization problem over the probability simplex $\cD_k$ which is independent of $N$.

    Our definition of digit profile builds on that of \cite{mark2022cutoff}.
    One change is that when $\vbn^{(1)}$ is deterministic, the sorting $S_K$ only starts to behave randomly in its second digit.
    In particular, this causes all strings of the form $j00\dots 0$ or $j(k-1)(k-1)\dots (k-1)$ to appear in highly predictable locations, so they have to be analyzed separately. In the multinomial case, this happens only for strings starting with many $0$'s or $(k-1)$'s, with no special exception for $y[1]$.
    Additionally in the multinomial case such ``near boundary'' prefixes suffice to completely determine the fluctuations, while our analysis must consider the much more combinatorially detailed information of pairs $(y,z) \in \al\times \ar$.

We next define constants depending on the digit profile of $x\in [k]_0^M$. From now on, we always suppose by symmetry that $M_y(x)\geq M_z(x)$ (i.e.\ $z$ is a prefix of $y)$). Let
\[
c_{\tot}^{(i)}(x)=\sum_{l:(i,l)\in T}c_l^{(i)}(x)
\] 
be the number of digits in $x$ that belong to $\mathsf P_i$ after $y$, and 
\[
c_{\tot}(x)=\sum_{0\leq i\leq z}c_{\tot}^{(i)}(x)
\] 
be the number of digits in $x$ that belong to $\mathsf P$ after $y$.
Also define
\begin{align}
c_L(x)\label{eq:def-cL}
&\equiv 
1-\sum_{t\leq M_y} \log_N(N/n^{(t)}_{x[t]})-\sum_{(i,l)\in T} c_l^{(i)} \log \lt(1/p_l^{(i)}\rt)-\sum_{M_y<t\leq M_x:(t,x[t])\in \mathsf Q}\log_N \lt(N/n^{(t)}_{x[t]}\rt),\\
c_E(x)
&\equiv
-\frac{1}{\log N}\sum_{t>M}\psi_{\bn^{(t)}/N}(2).\label{eq:def-cE}
\end{align}
The former dictates the typical size of $\cI(B_x)$, see \eqref{eq:cL} just below.
Note that by \eqref{cF-cL-bound-1} we have
\begin{equation}\label{eq:cF-cL-bound}
    c_F(x) \geq  \frac{1}{2}c_L(x)-o(1).
\end{equation}

\begin{lem}\label{lem:digit-hi-repre}
Recall the definition of $h_i$ and $\bp^{(i)}$ in \eqref{def:hipi}. We have
\begin{align}
\label{eq:cL}
c_L(x)&=\log_N \mathbb E|\mathcal I(B_x)|+O(\delta),\\ 
\label{eq:cE}
c_E(x)&= \lt(\frac{M_x-K}{\log N}\rt) \sum_{0\leq i\leq z} h_i \psi_{\bp^{(i)}}(2)+O(\delta),\\
\label{eq:c-tot-i}
c_{\tot}^{(i)}(x)&= 
\frac{1}{\log N}\Big[h_i (M_x-M_y) - \sum_{M_y<t\leq M_x:(t,x[t])\in \mathsf Q}1_{\{n^{(t)}/N\in D_i\}}\Big] + O(\chi). 
\end{align}
\end{lem}
\begin{proof}
We first prove \eqref{eq:cL}. Notice that when $\bn^{(t)}/N\in D_i$ for some $k\leq i\leq z$, we have that $\|\bn^{(t)}/N-\bp^{(i)}\|_{\infty}<\chi^2$. In this case, when $p_{x[t]}^{(i)}>\chi$ we have 
\begin{equation}\label{eq:diff}
\bigg|\log\Big(\frac{n^{(t)}_{x[t]}}{N}\Big)-\log \lt(p_{x[t]}^{(i)}\rt)\bigg| < O(\chi).
\end{equation}
When $\bn^{(t)}/N\in D_i$ and $p_{x[t]}^{(i)}>\chi$ for some $0\leq i\leq k-1$, we must have $p_{x[t]}^{(i)}>1-k\chi$, and \eqref{eq:diff} still holds.
Therefore,
\begin{align*}
&\,\,\log_N \mathbb E|\mathcal I(B_x)|-\log_N \mathbb E|\mathcal I(B_y)|=\frac{1}{\log N}\sum_{M_y<t\leq M_x} \log (n^{(t)}_{x[t]}/N)\\
&=\frac{1}{\log N} \sum_{M_y<t\leq M_x}\sum_{(i,l)\in T} 1_{\{x[t]=l,\bn^{(t)}\in D_i\}} \log(n^{(t)}_{x[t]}/N) +\sum_{M_y<t\leq M_x:(t,x[t])\in \mathsf Q}\log_N \lt(N/n^{(t)}_{x[t]}\rt)
\\
&=\frac{1}{\log N}\sum_{(i,l)\in T} c_l^{(i)}(x) \log N \lt( \log \lt(p_l^{(i)}\rt)+O(\chi) \rt)+ \sum_{M_y<t\leq M_x:(t,x[t])\in \mathsf Q}\log_N \lt(N/n^{(t)}_{x[t]}\rt)\\
&=\sum_{(i,l)\in T} c_l^{(i)}(x) \log \lt(p_l^{(i)}\rt)+\sum_{M_y<t\leq M_x:(t,x[t])\in \mathsf Q}\log_N \lt(N/n^{(t)}_{x[t]}\rt)+O(\chi), 
\end{align*}
Here we used the fact that $\sum_{(i,l)\in T} c_l^{(i)}(x)=c_{\tot}(x)$ is bounded by $K/\log N$.
\eqref{eq:cL} then holds by the fact
\begin{equation}
\log_N \mathbb E|\mathcal I(B_y)|=1-\sum_{t\leq M_y}\log_N(N/n^{(t)}_{x[t]}).
\end{equation}

Now we prove \eqref{eq:cE}. Again, note that if $\bn^{(t)}/N\in D_i$ for some $0\leq i\leq z$, then $\|\bn^{(t)}/N-\bp^{(i)}\|_{\infty}<\chi$, so
\begin{equation}
\bigg|\psi_{\bn^{(t)}/N}(2)-\psi_{\bp^{(i)}}(2)\bigg| < O(\chi).
\end{equation}
Recall the definition of $t_*$ in \eqref{def:almost}. Suppose that $j_1$ is the smallest integer such that $t_*+j_1\rho\log N>M_x$, and $j_2$ is the largest integer such that $t_*+j_2\rho\log N<K$. Thus we have $(j_2-j_1)\rho < \frac{M_x-K}{\log N} < (j_2-j_1)\rho +2\rho$. In addition, we see that $\psi_{\bp}(2)$ is bounded by $\log k$ for all $\bp\in \mathcal D_k$.
Therefore,
\begin{align*}
c_E(x)&=-\frac{1}{\log N}\sum_{t_*+j_1\rho \log N<t\leq t_*+j_2 \rho \log N} \psi_{\bn^{(t)}/N}(2)+O(\rho)\\
&=-\frac{1}{\log N}\sum_{t_*+j_1\rho \log N<t\leq t_*+j_2 \rho \log N} \sum_{i} 1_{\{\bn^{(t)}/N\in D_i\}}\psi_{\bp^{(i)}}(2)+O(\rho)+O(\chi)\\
&=-(j_2-j_1)\rho\sum_{0\leq i\leq z}h_i\psi_{\bp^{(i)}}(2)+O(\rho)+O(\chi)+O(\varphi z)
\\
&= \lt(\frac{M_x-K}{\log N}\rt) \sum_{0\leq i\leq z} h_i \psi_{\bp^{(i)}}(2)+O(\chi).
\end{align*}

For \eqref{eq:c-tot-i}, suppose that $j_3$ is the smallest integer such that $t_*+j_3\rho\log N>M_y$, and $j_4$ is the largest integer such that $t_*+j_4\rho\log N<M_x$. We have
\begin{align*}
c_{\tot}^{(i)}(x) + \frac{1}{\log N} \sum_{M_y<t\leq M_x:(t,x[t])\in \mathsf Q}1_{\{n^{(t)}/N\in D_i\}}
&=\frac{1}{\log N}\sum_{M_y<t\leq M_x} 1_{\{\bn^{(t)}/N\in D_i\}}
\\
&=\frac{1}{\log N}\sum_{t_*+j_3\rho \log N<t\leq  t_*+j_4\rho \log N} 1_{\{\bn^{(t)}/N\in D_i\}} + O(\rho)\\
&=h_i (j_4-j_3)\rho + O(\rho) +O(\varphi)\\
&=\frac{h_i}{\log N}(M_x-M_y)+ O(\rho) +O(\varphi),
\end{align*}
which completes the proof.
\end{proof}

    \subsection{Estimating the first moment}

    Now we control the probability that the edge $(i,i+1)$ belongs to $E(G,G^\prime)$. 
    As in \cite{mark2022cutoff}, the idea is to find a prefix $x$ such that $\cI(B_x)$ typically contains $(i,i+1)$, but the location of $\cI(B_x)$ has fluctuations of only slightly smaller order than the typical size $|\cI(B_x)|$ (recall the definition \eqref{eq:I-Bx} of $\cI(B_x)$).
    The fraction of edges near $i$ can then be modelled as locally homogeneous, with local density depending on the digit profile of $x$.
    Note that $x$ of course depends on $i$, and in fact will identify not just a single $x$ but a small set of them.
    In preparation for this argument, we construct a tree-structured partition of $S_K$ based on prefixes. First let $\mathcal T$ be a $k$-ary tree of depth $K$, which consists of all strings in $[k]_0^M$ for $0\leq M\leq K$. Here the children of $s\in [k]_0^M$ are $[si]$ for $i=0,1,\cdots,k-1$, and we let $\mathcal Fx$ be the parent of any vertex $x$ in $\mathcal T$. We then partition $\mathcal T$ as follows. 
    For $y\in \mathcal A_{\rmleft}$ and $z\in \mathcal A_{\rmright}$ such that $A_{y,z}$ is non-empty, suppose by symmetry that $z$ is a prefix of $y$. Therefore, define
    \begin{equation}\label{eq:Ay-stable}
        A_{y,z}^{\stable}=\begin{cases}
            \{y\}, \mbox{ if }c_L(y)-c_F(y)<\delta;\\
            \{x\in A_{y,z}:c_L(x)-c_F(x)<\delta \mbox{ and }c_L(\cF x)-c_F(y)\geq \delta\}, \mbox{ if }c_L(y)-c_F(y)\geq\delta.
        \end{cases}
    \end{equation}
    It is useful to point out that $c_L(x)$ is non-increasing down $\cT$, and when $c_L(y)-c_F(y)\geq\delta$ and $x\in A_{y,z}^{\stable}$,
    \[
    c_F(y)=c_F(x)=c_F(\mathcal F x)
    .
    \]
    With the notation above, have the following partition lemma.
    \begin{lem}
\label{lem:partition}
The following (resp. deterministic and random) partitions (i.e. disjoint unions) hold:
\begin{equation}\label{eq:partition}
    [k]_0^K=\bigcup_{y\in \mathcal A_{\rmleft},z\in \mathcal A_{\rmright}}\bigcup_{x\in A_{y,z}^{\stable}} B_x\quad\text{ and }\quad [N]=\bigcup_{y\in \mathcal A_{\rmleft},z\in \mathcal A_{\rmright}}\bigcup_{x\in A_{y,z}^{\stable}} \cI(B_x).
\end{equation}
\end{lem}
\begin{proof}
    By definition of $\mathcal A_{\rmleft}$ and $\mathcal A_{\rmright}$, we have
    \[
    [k]_0^K=\bigcup_{y\in \mathcal A_{\rmleft},z\in \mathcal A_{\rmright}}\bigcup_{x\in A_{y,z}} B_x.
    \]
    It suffices to establish the partition
    \[
    \bigcup_{x\in A_{y,z}} B_x=\bigcup_{x\in A_{y,z}^{\stable}} B_x.
    \]
    When $c_L(y)-c_F(y)<\delta$ the claim is obvious. When $c_L(y)-c_F(y)\geq\delta$, For $x'\in A_{y,z}$, the value $c_L(x')-c_F(x')=c_L(x')-c_F(y)$ is decreasing along any path in $\mathcal T$ from $y$ to a leaf. When $x=y$ is the root, we have $ c_L(x)-c_F(x)\geq \delta$.
    It suffices to check that $c_L(s)-c_F(s)<\delta$ for any leaf $s\in A_{y,z}$.
    In fact, by \eqref{eq:cF-cL-bound} and \eqref{eq:cL} we have
    \[c_L(s)-c_F(s)\leq \frac{1}{2}c_L(s)+o(1)\leq \frac{1}{2}\log_N \mathbb E |\cI(B_s)|-o(1).\]
    Since $K\geq (\widetilde C_{\mu}+\varepsilon)\log N$, we have
    \begin{align}
        \log_N \mathbb E |\cI(B_s)| \leq 1 - \sum_{1\leq t\leq K}\log_N(N/n_{l^{(t)}_{\Max}}^{(t)})=1-\frac{K}{\log N}\sum_{0\leq i\leq z} h_i\log (1/p^{(i)}_{\Max})\pm O(\delta)\\
        =1-\frac{K\mathbb E_{\mu}\log (1/p_{\Max})\pm O(\delta)}{\log N}<0,
    \end{align}
    which concludes the proof.
\end{proof}
    
    \begin{lem}\label{lem:new-blocksvalid}
    For each index $i\in [N]$ and for each $y \in \al, z \in \ar$, there exists a subset $A_{y,z,i}^{\stable}\subseteq A_{y,z}^{\stable}$ such that
    \begin{enumerate}
        \item The size of the subset is bounded: $|A_{y,z,i}^{\stable}| \leq 2k$.
        \item We have
        \begin{align}
            \mathbb P\Big[i\in \bigcup_{y \in \al, z \in \ar}\bigcup_{x \in A_{y,z,i}^{\stable}} \cI(B_x)\Big]\geq 1 - e^{-\Omega(N^{\delta})}. 
        \end{align}
    \end{enumerate}
    \end{lem}

    \begin{proof}
        For each $y \in \al, z\in \ar$, consider the set
        \[
        A_{y,z,i}^{\stable}=\{x\in  A_{y,z}^{\stable},\mathbb P(i\in \mathcal I(B_x))\geq e^{-N^{\delta}}\}.
        \]
        By \eqref{eq:partition}, we have immediately that
       \begin{equation} \label{eq:stable-union-bound}
           \mathbb P \Big[i\notin \bigcup_{y \in \al, z\in \ar}\bigcup_{x \in A_{y,z,i}^{\stable}} \cI(B_x)\Big]
           \leq 
           e^{-\Omega(N^{\delta})} \cdot |\mathcal A_{\rmleft}| \cdot |\mathcal A_{\rmright}| \cdot \sup_{y \in \mathcal A}|A_{y,z}^{\stable}|
       \end{equation}
       Claim 2 follows directly from the fact that $|\al|\vee|\ar|\vee|A_{y,z}^{\stable}|\leq k^K=N^{O(1)}$. It suffices to prove that $|A_{y,z,i}^{\stable}|\leq 2k$.

       Without loss of generality, we assume that $z$ is a prefix of $y$ and $c_L(y)-c_F(y)\geq \delta$ (otherwise $|A_{y,z,i}^{\stable}|\leq 1$).
       Consider the map (recall that $\mathcal Fx$ is the parent of $x$)
       \begin{equation}
       \begin{aligned}
           \Phi: A_{y,z,i}^{\stable}&\to\mathcal T;\\
                x&\mapsto \mathcal Fx,
       \end{aligned}
       \end{equation}
       we claim that the range of $\Phi$, denoted $\mathsf{Range}(\Phi)$, satisfies $|\mathsf{Range}(\Phi)| \leq 2$. 
       Therefore, since for each $y \in \mathsf{Range}(\Phi)$, we have $|\Phi^{-1}(y)| \leq k$, we conclude Claim 1.

       We prove the claim on $\mathsf{Range}(\Phi)$ by contradiction. Suppose that there are $3$ strings $y_1,y_2,y_3$ in $\mathsf{Range}(\Phi)$, such that $\mathbb P(i\in \mathcal I(B_{y_j}))\geq e^{-N^{\delta}}$ for each $j\in \{1,2,3\}$. By Lemma \ref{lem:partition} we see $J_{y_i}$'s are disjoint. Therefore, we suppose without loss of generality that $J_{y_1},J_{y_2},J_{y_3}$ lies from left to right on the interval $[0,1]$, and $\frac{i}{N}$ is not larger than the midpoint of $J_{y_2}$. Thus we have $i<(\iota_{y_2}+\tau_{y_2})/2$
       and $\mathbb P(i\in \mathcal I(B_{y_3}))\leq \mathbb P(\iota(y_2)<i)$. Note that by the definition of $A_{y,z,i}^{\stable}$, we have $ c_L(\mathcal Fx)- c_F(\mathcal Fx)\geq\delta$, which by \eqref{eq:cF-cL-bound} implies that $c_F(\mathcal Fx)>\delta-o(1)$ and $c_L(\mathcal Fx)>2\delta-o(1)$. By Lemma \ref{lem:concentration}, we obtain that $\mathbb P(\iota(y_2)<i)\leq \mathbb P(|\iota(y_2)-\mathbb E\iota (y_2)|\geq N^{c_L(y_2)}/2)\leq e^{-\Omega(N^{2\delta})}\leq e^{-N^\delta}$, which leads to a contradiction. Thus $\mathsf{Range}(\Phi)$ includes at most two elements as desired.
    \end{proof}

    \begin{lem}\label{lem:prefix-size}
    Recall the definition of $\tal$ and $\tar$ in Definition \ref{eq:def:admissible-prefixes}. With all assumptions in Proposition \ref{prop:new-firstmoment}, we have
        \begin{equation}\label{eq:prefix-size}
            |\tal|\vee|\tar| \leq N^{\delta}.
        \end{equation}
    \end{lem}

    \begin{proof}
        For fixed $M$ with $1 <M\leq K$, we first count the number of strings $x$ in $\tal\cap [k]_0^M$. Notice that if we fix the set $\{t:1<t<K,x[t]<\ileft^{(t)}\}$, then there are at most $k$ choices of $x$ (because of the first digit). Therefore, we have
        \begin{equation}
            |\tal|\leq \sum_{0\leq i\leq\frac{\mathsf{C}}{\log (1/\chi)}\log N}\binom{K}{i}\cdot k^{\frac{\mathsf{C}}{\log (1/\chi)}\log N}\cdot k\leq N^{\delta}.
        \end{equation}
        Here we used the fact that $\frac{\mathsf{C}}{\log (1/\chi)}\ll \delta$.
        By symmetry the same inequality holds for $|\tar|$.
    \end{proof}

    \begin{lem}\label{lem:bad-prefix} With all assumptions in Proposition \ref{prop:new-firstmoment}, we have for all $1\leq i\leq N-1$,
        \begin{equation}\label{eq:bad-prefix}
            \mathbb P\lt[(i,i+1)\in 
            \bigcup_{\substack{(y,z)\in (\al\times \ar); \\ (y,z)\notin (\tal\times\tar)} } \bigcup_{x\in A_{y,z,i}^{\stable}}E(G_{B_x}) \rt] \leq N^{-\mathsf{C}/2}.
        \end{equation}
    \end{lem}

    \begin{proof}
        For $(y,z)\in (\al\times \ar) \setminus (\tal\times\tar)$, suppose without loss of generality that $z$ is a prefix of $y$. We have $y\in \al\setminus \tal$, and
        \begin{align*}
            \mathbb P\lt[(i,i+1)\in \bigcup_{x\in A_{y,z,i}^{\stable}}G(B_x)\rt] 
            \leq 
            \mathbb P[(i,i+1)\in G(B_y)]\leq \mathbb E |\mathcal I(B_y)|\\
            \leq N\cdot \prod_{t:y[t]<\ileft^{(t)}}\frac{n_{y[t]}}{N}\leq N\cdot \chi ^{\frac{\mathsf{C}}{\log (1/\chi)} \log N}\leq N^{-\mathsf{C}+1}.
        \end{align*}
        Therefore, we conclude since $|\al\times\ar|\leq k^{2K}\leq N^{\mathsf{C}/3}$, the former via a union bound.
    \end{proof}

    \begin{lem}\label{lem:AMGM}
        With all assumptions in Proposition \ref{prop:new-firstmoment}, we have
        \begin{equation}\label{new-AMGM}
            \mathbb E[|E(G,G')|] \leq N^{-\mathsf{C}/2}+N^{O(\delta)}\cdot\sum_{y \in \tal, z \in \tar}\sum_{x \in A_{y,z,i}^{\stable}} \sum_{i=1}^{N-1}\mathbb P\lt[(i,i+1)\in E(G_{B_x})\rt]^2].
        \end{equation}
    \end{lem}

    \begin{proof}
        By definition of $E(G,G')$ we have
        \begin{align*}
            \mathbb E[|E(G,G')|] = \sum_{i=1}^{N-1} \mathbb P[(i,i+1)\in E(G,G')]= \sum_{i=1}^{N-1}\mathbb P[(i,i+1)\in E(G)]^2.
        \end{align*}
        Recall the definition of $A_{y,z,i}^{\stable}$ in Lemma \ref{lem:new-blocksvalid}, we have
        \begin{equation}
            \mathbb P[(i,i+1)\in E(G)]\leq e^{-\Omega(N^{\delta})}+\mathbb P\lt[(i,i+1)\in \bigcup_{y \in \al, z \in \ar}\bigcup_{x \in A_{y,z,i}^{\stable}} E(G_{B_x})\rt].
        \end{equation}
        By Lemma \ref{eq:bad-prefix} and a union bound, we have
        \begin{align*}
            \mathbb P[(i,i+1)\in E(G)]
            &\leq N^{-\mathsf C/2}+\mathbb P\lt[(i,i+1)\in \bigcup_{y \in \tal, z \in \tar}\bigcup_{x \in A_{y,z,i}^{\stable}} E(G_{B_x})\rt]\\
            &\leq N^{-\mathsf C/2}+\sum_{y \in \tal, z \in \tar}\sum_{x \in A_{y,z,i}^{\stable}}\mathbb P\lt[(i,i+1)\in E(G_{B_x})\rt].
        \end{align*}
        Therefore, by \eqref{eq:prefix-size} and the AM-GM inequality, we have
        \begin{align*}
            &\sum_{i=1}^{N-1}\mathbb P[(i,i+1)\in E(G)]^2\\
            &\leq \sum_{i=1}^{N-1} \lt[N^{-\mathsf C} + \sum_{y \in \tal, z \in \tar}\sum_{x \in A_{y,z,i}^{\stable}}\mathbb P\lt[(i,i+1)\in E(G_{B_x})\rt]^2\rt]\cdot N^{O(\delta)}\\
            &\leq N^{-\mathsf{C}/2}+N^{O(\delta)}\cdot\sum_{y \in \tal, z \in \tar}\sum_{x \in A_{y,z}^{\stable}} \sum_{i=1}^{N-1}\mathbb P\lt[(i,i+1)\in E(G_{B_x})\rt]^2,
        \end{align*}
        which completes the proof.
    \end{proof}

    Now it suffices to upper bound for each $(y,z)\in \tal \times \tar$ the sum
    \begin{equation}\label{AMGM-sum}
        \sum_{x\in A_{y,z}^{\stable}} \sum_{i=1}^{N-1} \mathbb P[(i,i+1)\in E(G_{B_x})]^2.
    \end{equation}

\begin{lem}
\label{lem:expect-block}
    With all assumptions in Proposition \ref{prop:new-firstmoment}, we have for any $x\in[k]_0^M$,
    \[
    \mathbb E\lt(|E(G_{B_x})|~\big\vert~ |\cI(B_x)|\rt)\leq |\cI(B_x)|^2 N^{c_E( x)+o(1)}.
    \]
\end{lem}
\begin{proof}
    It is easy to see that for $x\in [k]_0^M$, and for a random pair of rows $s,s'$,
    \begin{align*}
        \mathbb P[s=s'|s,s'\in B_x]
        = \prod_{M<t\leq K} 
    \Big(\sum_i\frac{n_i^{(t)}(n_i^{(t)}-1)}{N(N-1)}\Big)
    \leq
    N^{c_E( x)+o(1)}.
    \end{align*}
    Summing over at most $|\cI(B_x)|^2$ pairs of pre-sorted strings in $B_x$ completes the proof.
\end{proof}
    
    \begin{lem}\label{lem:edgedensity}
        With all assumptions in Proposition \ref{prop:new-firstmoment}, we have for all $y\in \tal,z\in \tar$ and $x \in A_{y,z}^{\stable}$, the following holds.
        \begin{enumerate}[label=(\roman*)]
            \item 
            \label{it:upbound1}
            If $c_L(x)\geq 2\delta$, we have
        \begin{equation}\label{upbound1}
            \sum_{i=1}^{N-1}\mathbb P[(i,i+1)\in E(G_{B_x})]^2\leq C N^{4 c_L(x)- c_F(x)+2c_E( x)+O(\delta)} + e^{-\Omega(N^{\delta-o(1)})}.
        \end{equation}
            \item
            \label{it:upbound2} 
            If $ c_L(x)< 2\delta$, $ c_F\geq\delta-(\log N)^{-1/2}$ and $A_{y,z}^{\stable}\neq\{y\vee z\}$, we have
              \begin{equation}\label{upbound2}
                  \sum_{i=1}^{N-1}\mathbb P[(i,i+1)\in E(G_{B_x})]^2\leq C N^{- c_F(x)+2c_E( x)+O(\delta)} + e^{-\Omega(N^{\delta-o(1)})}.
              \end{equation}
            \item 
            \label{it:upbound2.5}
            If $c_L(x)<2\delta$ and $ c_F(x)<\delta-(\log N)^{-1/2}$, then $x=y\vee z$ and $A_{y,z}^{\stable}=\{y\vee z\}$.
            \item 
            \label{it:upbound3}
            If $A_{y,z}^{\stable}=\{y \vee z\}$ and $c_L(y \vee z)<2\delta$, 
            we have $x=y\vee z$ and
            \begin{equation}\label{upbound3}
                \sum_{i=1}^{N-1}\mathbb P[(i,i+1)\in E(G_{B_{x}})]^2\leq C N^{2c_L(x)+2c_E(x)+O(\delta)} + e^{-\Omega(N^{\delta-o(1)})}.
            \end{equation}
        \end{enumerate}
    \end{lem}

    The idea behind Lemma~\ref{lem:edgedensity} is that conditional on the internal structure of $G(B_x)$ (i.e. the multiset of strings with prefix $x$), the location of $\cI(B_x)$ will be random with translational fluctuations of order $N^{c_F(x)}$. 
By definition of $A_{y,z}^{\stable}$, this is of almost the same order as $|B_x|$ itself.
Therefore one can expect to understand $\mathbb P[(i,i+1)\in E(G(B_x))]$ based on the local edge density estimated in Lemma~\ref{lem:expect-block}.
This is significantly easier in the binomial case because the random translation is a binomial random variable with fairly simple parameters, but requires much more care in our setting.
Thus to prove Lemma~\ref{lem:edgedensity}, we rely on Lemma \ref{lem:anticoncen}, an anti-concentration result for conditional hypergeometric distributions. 

\begin{proof}[Proof of Lemma \ref{lem:edgedensity}]
        Recall the definition of $U_x$ in \eqref{eq:U-x} and we take the event $\cS'$ from Lemma \ref{lem:anticoncen}.

        \vspace{5pt}
        \noindent \emph{Claim~\ref{it:upbound1}.}
        Since $x\in A_{y,z}^{\stable}$, we have $c_L(x)-c_F(x)<\delta$. 
        By Lemma \ref{lem:expect-block} we have
        \[
        \mathbb E\lt(|E(G_{B_x})|~\big\vert~ |\cI(B_x)|\rt)\leq |\cI(B_x)|^2 N^{c_E( x)+o(1)}.
        \]
        Furthermore, since we have by \eqref{eq:cF-cL-bound} and the assumption $c_L(x)\geq 2\delta$ that
        \[
        c_F(x)\geq c_L(y(x))/2-o(1)\geq c_L(x)/2-o(1)\geq \delta-o(1),
        \]
        Lemma \ref{lem:anticoncen} implies that $\max_{j\in [N]}\mathbb P(\iota(x)=j,S_K\in\cS'|U_x) =O(N^{- c_F(x)+2\delta})$. Therefore,
        \begin{equation}\label{eq:new-pointwise}
            \mathbb P((i,i+1)\in E(G_{B_x}),S_K\in \cS'~\big\vert~ |\cI(B_x)|)\leq C|\cI(B_x)|^2 N^{c_E(x)-c_F( x)+O(\delta)}.
        \end{equation}
        For $i\in \left[Nt_x-N^{ c_F(x)+2\delta},N(t_x+\lambda_x)+N^{ c_F(x)+2\delta}\right]$, Lemma \ref{concen-prefix} shows that 
        \[\mathbb P(|\mathcal I(B_x)|\geq 2N^{c_L(x)})\leq e^{-\Omega(N^\delta)-o(1)},\]
        and thus
        \begin{equation}\label{eq:blockedge}
        \mathbb P[(i,i+1)\in E(G_{B_x})]\leq CN^{2 c_L(x)-c_F(x)+c_E( x)+O(\delta)} +e^{-\Omega(N^{\delta-o(1)})}.
        \end{equation}
        While for $i\notin \left[Nt_x-N^{c_F(x)+2\delta},N(t_x+\lambda_x)+N^{c_F(x)+2\delta}\right]$, since $c_L(x)-c_F(x)<\delta$, \eqref{eq:positionconcentration} and \eqref{eq:poscon2} imply
    \begin{equation}\label{eq:blockedge1}
        \mathbb P[(i,i+1)\in E(G_{B_x})]\leq e^{-\Omega(N^{\delta-o(1)})},
    \end{equation}
    where we used $c_F(x)\geq \delta-o(1)$. We conclude that
    \begin{align*}
    &\sum_{i=1}^{N-1}\mathbb P[(i,i+1)\in E(G_{B_x})]^2\\
    &\leq 2N^{c_F(x)+2\delta}\lt(CN^{2 c_L(x)- c_F(x)+c_E(x)+O(\delta)} +e^{-\Omega(N^{\delta-o(1)})}\rt)^2+Ne^{-\Omega(N^{\delta-o(1)})}\\
    &\leq C N^{4 c_L(x)- c_F(x)+2c_E( x)+O(\delta)} + e^{-\Omega(N^{\delta-o(1)})}.
    \end{align*}

    \vspace{5pt}

    \noindent \emph{Claim~\ref{it:upbound2}.} By Lemma \ref{lem:concen-block-smallmean}, we have that $\mathbb P(|\mathcal I(B_x)|\geq 2N^{2\delta})<\exp(-\Omega(N^\delta)).$ For any $x\in \mathcal A_{y,z}^{\stable}$ with $ c_L(x)<2\delta$, $ c_F(x)\geq \delta-(\log N)^{-1/2}$ and $i\in \left[Nt_x-N^{ c_F(x)+2\delta},N(t_x+\lambda_x)+N^{ c_F(x)+2\delta}\right]$, by \eqref{eq:new-pointwise} we have
    \begin{equation}\label{eq:blockedge2}
        \mathbb P[(i,i+1)\in E(G_{B_x})]\leq CN^{- c_F(x)+c_E( x)+O(\delta)} +e^{-\Omega(N^{\delta-o(1)})}.
    \end{equation}
    Meanwhile for $i\notin \left[Nt_x-N^{ c_F(x)+2\delta},N(t_x+\lambda_x)+N^{ c_F(x)+2\delta}\right]$, since $c_L(x)-c_F(x)<\delta$, \eqref{eq:positionconcentration} and \eqref{eq:poscon2} imply
    \begin{equation}\label{eq:blockedge3}
        \mathbb P[(i,i+1)\in E(G_{B_x})]\leq e^{-\Omega(N^{\delta-o(1)})},
    \end{equation}
    where we used $c_F(x)\geq \delta-(\log N)^{-\frac{1}{2}}$. We conclude that
    \begin{align*}
    &\sum_{i=1}^{N-1}\mathbb P[(i,i+1)\in E(G_{B_x})]^2\\
    &\leq 2N^{c_F(x)+2\delta}\lt(CN^{- c_F(x)+c_E(x)+O(\delta)} +e^{-\Omega(N^{\delta-o(1)})}\rt)^2+Ne^{-\Omega(N^{\delta-o(1)})}\\
    &\leq C N^{-c_F(x)+2c_E(x)+O(\delta)} + e^{-\Omega(N^{\delta-o(1)})}.
    \end{align*}

    \vspace{5pt}

    \noindent \emph{Claims~\ref{it:upbound2.5} and \ref{it:upbound3}.}
    Without loss of generality we assume that $z$ is a prefix of $y$.
    By \eqref{cF-cL-bound-1}, we have $c_L(y) \leq 2c_F(x)+O((\log N)^{-1})$. Combining this with $c_F(y)=c_F(x) \leq \delta-(\log N)^{-1/2}$, 
    we obtain
    \begin{equation}
        \begin{aligned}
            c_L(y)-c_F(y) < \delta.
        \end{aligned}
    \end{equation}
    Thus, we obtain that $A_{y,z}^{\stable}=\{y\}$ by the definition of $A_{y,z}^{\stable}$ and conclude Claim~\ref{it:upbound2.5}.

    Lemma \ref{lem:anticoncen} implies that $\max_{j\in [N]}\mathbb P(\iota(x)=j,S_K\in\cS'|U_x) =O(N^{- \max(c_F(x),\delta)+2\delta})$. Therefore, by Lemma \ref{lem:expect-block} we have
        \begin{equation}\label{pointwise1}
            \mathbb P((i,i+1)\in E(G_{B_x}),S_K\in \cS'~\big\vert~ |\cI(B_x)|)\leq |\cI(B_x)|^2 N^{-\max(c_F(x),\delta)+c_E(x)+O(\delta)}.
    \end{equation}
    By Lemma \ref{lem:concen-block-smallmean} and the assumption that $c_L(x)<2\delta$, we have $\mathbb P(|\mathcal I(B_x)|\geq N^{4\delta})\leq e^{-\Omega(N^{\delta-o(1)})}$, and thus
    \begin{equation}
        \mathbb E|\mathcal I(B_x)|^2\leq N^{4\delta}\mathbb E|\mathcal I(B_x)| + e^{-\Omega(N^{\delta})} \leq N^{c_L(x)+O(\delta)}+e^{-\Omega(N^{\delta-o(1)})}.
    \end{equation}
    Therefore,
    \begin{equation}\label{pointwise2}
            \mathbb P((i,i+1)\in E(G_{B_x}))\leq N^{c_L(x)- \max(c_F(x),\delta) +c_E(x)+O(\delta)} + e^{-\Omega(N^{\delta-o(1)})}.
    \end{equation}
    Now, for $i\in \left[Nt_x-N^{\max(c_F(x),\delta)+2\delta},N(t_x+\lambda_x)+N^{\max(c_F(x),\delta)+2\delta}\right]$, we have
    \begin{equation}
        \mathbb P[(i,i+1)\in E(G_{B_x})]\leq N^{c_L(x)- \max(c_F(x),\delta) +c_E(x)+O(\delta)} + e^{-\Omega(N^{\delta-o(1)})},
    \end{equation}
    While for $i\notin \left[Nt_x-N^{\max(c_F(x),\delta)+2\delta},N(t_x+\lambda_x)+N^{\max(c_F(x),\delta)+2\delta}\right]$, \eqref{eq:positionconcentration} and \eqref{eq:poscon2} imply
    \begin{equation}
        \mathbb P[(i,i+1)\in E(G_{B_x})]\leq e^{-\Omega(N^{\delta-o(1)})}.
    \end{equation}
    We conclude that
    \begin{align*}
    &\sum_{i=1}^{N-1}\mathbb P[(i,i+1)\in E(G_{B_x})]^2\\
    &\leq 
    \lt(N^{c_L(x)+O(\delta)}+N^{\max(c_F(x),\delta)+O(\delta)}\rt)\lt(N^{c_L(x)-\max(c_F(x),\delta)+c_E(x)+O(\delta)} \rt)^2+Ne^{-\Omega(N^{\delta-o(1)})}\\
    &\leq C N^{2c_L(x)+2c_E(x)+O(\delta)} + e^{-\Omega(N^{\delta-o(1)})}.
    \qedhere
    \end{align*}
    \end{proof}

    To apply Lemma \ref{lem:edgedensity}, we need the following numerical lemma.

    \begin{lem}\label{lem:numerical} 
        If $\delta$, $\chi$, $\varepsilon$ are as in \eqref{eq:small-params}, the following properties hold for each $y\in \mathcal A$.
        \begin{enumerate}[label=(\alph*)]
            \item 
            \label{it:calcu1}
            If $x \in A_{y,z}^{\stable}$, then we have
            \begin{equation}\label{eq:calcu1}
                4 c_L(x)- c_F(x)+2c_E( x)+\sum_{0\leq i\leq z}c_{\tot}^{(i)}H(\{c_l^{(i)}\}_{(i,l)\in T})\leq-\Omega(\varepsilon).
            \end{equation}
            \item 
            \label{it:calcu2}
            If $x\in A_{y,z}^{\stable}$ and $A_{y,z}^{\stable}\neq\{y\vee z\}$, then we have
            \begin{equation}\label{eq:calcu2}
                - c_F(x)+2c_E( x)+\sum_{0\leq i\leq z}c_{\tot}^{(i)}H(\{c_l^{(i)}\}_{(i,l)\in T})\leq-\Omega(\varepsilon).
            \end{equation}
            \item 
            \label{it:calcu3}
            If $A_{y,z}^{\stable}=\{y\vee z\}$, we have 
            \begin{equation}\label{eq:calcu3}
                c_E(y\vee z)+c_L(y\vee z)\leq -\Omega(\varepsilon).
            \end{equation}
        \end{enumerate}
    \end{lem}

    \begin{proof}
    We assume without loss of generality that $z$ is a prefix of $y$. Recall $M_y$ is the length of $y$.

    \vspace{5pt}
    \noindent \emph{Claim~\ref{it:calcu1}.}
    To deal with the entropy term, we use the non-negativity of Kullback-Leibler divergence. For fixed $i$, define
    \begin{equation}\label{def-frakp}
        \mathfrak c^{(i)}=\lt(\frac{c_l^{(i)}}{c_{\tot}^{(i)}}\rt)_{l:(i,l)\in T} \quad\mbox{and}\quad\mathfrak p^{(i)}=\lt(\frac{(p_l^{(i)})^{\theta_{\mu}}}{\sum_{l:(i,l) \in T}(p_l^{(i)})^{\theta_{\mu}}}\rt)_{l:(i,l)\in T}.
    \end{equation}
    We have:
    \begin{equation} \label{eq:entropy-upbd}
        \begin{aligned} 
        H(\{c_l^{(i)}\}_{(i,l)\in T})
        &=\sum_{l:(i,l)\in T} \frac{c_l^{(i)}}{c_{\tot}^{(i)}} \log\lt(\frac{\sum_{l:(i,l)\in T}(p_l^{(i)})^{\theta_{\mu}}}{(p_l^{(i)})^{\theta_{\mu}}}\rt)-D_{\term{KL}}(\mathfrak c^{(i)},\mathfrak p^{(i)})
        \\
        &\leq 
        \sum_{l:(i,l)\in T} \frac{c_l^{(i)}}{c_{\tot}^{(i)}} \log\lt(\frac{\sum_{l:(i,l)\in T}(p_l^{(i)})^{\theta_{\mu}}}{(p_l^{(i)})^{\theta_{\mu}}}\rt)
        \\
        &=
        \log \lt(\sum_{l:(i,l)\in T}(p_l^{(i)})^{\theta_{\mu}}\rt)+\theta_{\mu}\sum_{l:(i,l)\in T} \frac{c_l^{(i)}}{c_{\tot}^{(i)}} \log\lt(\frac{1}{p_l^{(i)}}\rt)
        \\
        &= 
        -\psi_{\bp^{(i)}}(\theta_{\mu})+\theta_{\mu}\sum_{l:(i,l)\in T} \frac{c_l^{(i)}}{c_{\tot}^{(i)}} \log\lt(\frac{1}{p_l^{(i)}}\rt)\pm O(\chi).
        \end{aligned}
    \end{equation}
    Combining \eqref{eq:c-tot-i} and \eqref{eq:entropy-upbd}, we have
    \begin{equation}
        \begin{aligned}
            \sum_{0\leq i\leq z}c_{\tot}^{(i)}H(\{c_l^{(i)}\}_{(i,l)\in T})
            &\leq -c_{\tot}\sum_{0\leq i\leq z} h_i\psi_{\bp^{(i)}}(\theta_{\mu})+\theta_{\mu}\sum_{0\leq i\leq z}\sum_{l:(i,l)\in T}c_l^{(i)}\log\lt(\frac{1}{p_l^{(i)}}\rt)+O(\delta)
            \\
            &\leq -c_{\tot} \psi_{\mu}(\theta_\mu)+\theta_{\mu}\sum_{0\leq i\leq z}\sum_{l:(i,l)\in T}c_l^{(i)}\log\lt(\frac{1}{p_l^{(i)}}\rt)+O(\delta).
        \end{aligned}
    \end{equation}
    Here we used $\rho,\chi,\varphi\ll \delta$.
    Recall that $\psi_{\mu}(\theta_\mu)=2\psi_{\mu}(2)$, we have by \eqref{eq:cE},
    \begin{equation}\label{1stterm-ent}
        \begin{aligned}
            2c_E(x)+\sum_{0\leq i\leq z}&c_{\tot}^{(i)}H(\{c_l^{(i)}\}_{(i,l)\in T})\\
            &\leq 2\psi_{\mu}(2)\lt(\frac{M_x-K}{\log N}-c_{\tot}\rt)+\theta_{\mu}\sum_{0\leq i\leq z}\sum_{l:(i,l)\in T}c_l^{(i)}\log\lt(\frac{1}{p_l^{(i)}}\rt)+O(\delta).
        \end{aligned}
    \end{equation}
    Meanwhile, we set
    \begin{equation} \label{eq:q-positive}
        q\equiv\sum_{M_y<t\leq M_x:(t,x[t])\in \mathsf Q}\log_N \lt(N/n^{(t)}_{x[t]}\rt)\geq 0.
    \end{equation}
    Then we have:
    \begin{align}
       c_L(x) &=c_L(y)-\sum_{0\leq i\leq z}\sum_{l:(i,l)\in T}c_l^{(i)}\log\lt(\frac{1}{p_l^{(i)}}\rt)-q+O(\chi),\label{eq:cL-repre}\\
        c_F(x)&=c_F(y) \geq \frac{1}{2}c_L(y)-o(1)\label{eq:cF-half-cL}\\
        c_L(x)- c_F(x)&= c_L(y)-c_F(y)-\sum_{0\leq i\leq z}\sum_{l:(i,l)\in T}c_l^{(i)}\log\lt(\frac{1}{p_l^{(i)}}\rt)-q+O(\chi)<\delta.\label{eq:delta-stable-pf}
    \end{align}
    By \eqref{eq:delta-stable-pf}, we obtain
    \begin{equation}\label{eq:sum-of-cipi}
        \sum_{0\leq i\leq z}\sum_{l:(i,l)\in T}c_l^{(i)}\log\lt(\frac{1}{p_l^{(i)}}\rt) \geq c_L(y)-c_F(y)-q-\delta-O(\chi).
    \end{equation}
    Recall that by $\eqref{eq:theta-bound}$, we have $\theta_{\mu} \in [3,4)$. Therefore, we have
    \begin{equation}\label{eq:2ndterm-ent}
    \begin{aligned}
        &\quad 4 c_L(x)- c_F(x)+\theta_{\mu}\sum_{0\leq i\leq z}\sum_{l:(i,l)\in T}c_l^{(i)}\log\lt(\frac{1}{p_l^{(i)}}\rt)\\
        \overset{\eqref{eq:cL-repre},\eqref{eq:cF-half-cL}}&{=} 4 c_L(y)- c_F(y)-4q+(\theta_{\mu}-4)\sum_{0\leq i\leq z}\sum_{l:(i,l)\in T}c_l^{(i)}\log\lt(\frac{1}{p_l^{(i)}}\rt)\\
        \overset{\eqref{eq:sum-of-cipi},\eqref{eq:theta-bound}}&{\leq} -\theta_{\mu}q+\theta_{\mu}c_L(y)+(3-\theta_{\mu})c_F(y)+O(\delta)\\
        \overset{\eqref{eq:theta-bound},\eqref{eq:q-positive},\eqref{eq:cF-half-cL}}&{\leq} \frac{\theta_{\mu}+3}{2}c_L(y)+O(\delta).
    \end{aligned}
    \end{equation}
    Here we used the fact that $\chi\ll \delta$. By the definition of $M_y$, we have
    \begin{equation} \label{eq:cLy-bound}
        c_L(y)\leq 1-
        \Big(\frac{M_y}{\log N}-O(\rho)\Big) 
        \mathbb E_{\mu} \log (1/p_{\max})+O(\delta).
    \end{equation}
    Combining \eqref{1stterm-ent}, \eqref{eq:2ndterm-ent} and \eqref{eq:cLy-bound}, we obtain that
    \begin{equation}
        4 c_L(x)- c_F(x)+2c_E( x)+\sum_{0\leq i\leq z}c_{\tot}^{(i)}H(\{c_l^{(i)}\}_{(i,l)\in T})\leq f(M_y)+O(\delta),
    \end{equation}
    where
    \begin{equation}
       f(M_y)= \frac{\theta_{\mu}+3}{2} \lt[1-\frac{M_y}{\log N}\cdot \mathbb E_{\mu} \log (1/p_{\max})\rt]+2\psi_{\mu}(2)\cdot\frac{M_y-K}{\log N}.
    \end{equation}
    Since $f(M_y)$ is a linear function for $M_y \in [0,K]$ and $\delta \ll \varepsilon$, it suffices to verify that $f(0) \leq -\Omega(\varepsilon)$ and $f(K) \leq -\Omega(\varepsilon)$.
    In fact, recall the definition of $C_{\mu}, \widetilde C_{\mu} $ in \eqref{eq:C-p}, we have
    \[
        \begin{aligned}
            f(0)&=\frac{\theta_{\mu}+3}{2}-2\psi_{\mu}(2)\cdot\frac{K}{\log N}=\frac{\theta_{\mu}+3}{2} \lt(1-\frac{K}{C_{\mu}\log N}\rt);\\
            f(K)&=\frac{\theta_{\mu}+3}{2} \lt[1-\frac{K}{\log N}\cdot \mathbb E_{\mu} \log (1/p_{\max})\rt]=\frac{\theta_{\mu}+3}{2} \lt(1-\frac{K}{\widetilde C_{\mu}\log N}\rt).
        \end{aligned}
    \]
    Therefore, we can obtain the desired result by noticing that $\frac{K}{\log N} \geq \overline{C}_{\mu}+\varepsilon$, where $\overline{C}_{\mu}=\max\{C_{\mu}, \widetilde C_{\mu}\}$.
    
    \vspace{5pt}
    \noindent \emph{Claim~\ref{it:calcu2}.}
    We continue using the notations in Claim~\ref{it:calcu1} and find
    \begin{align*}
        &- c_F(x)+2c_E( x)+\sum_{0\leq i\leq z}c_{\tot}^{(i)}H(\{c_l^{(i)}\}_{(i,l)\in T})
        \\
        &\leq-c_F(y)+2\psi_{\mu}(2)\frac{M_y-K}{\log N}+\theta_{\mu}\sum_{0\leq i\leq z}\sum_{l:(i,l)\in T}c_l^{(i)}\log\lt(\frac{1}{p_l^{(i)}}\rt). 
    \end{align*}
    Recall the definition of $\mathcal Fx$ above Lemma~\ref{lem:partition}.  Since we have assumed that $y \notin A_{y,z}^{\stable}$, it holds that $c_F(\mathcal{F}x)= c_F(x)$ and $c_L(\mathcal F x) \leq c_L(y)$. We have 
    \begin{align*}
        \delta<c_L(\mathcal F x)-c_F(y)<c_L(y)-c_F(y)-q-\sum_{0\leq i\leq z}\sum_{l:(i,l)\in T}c_l^{(i)}\log\lt(\frac{1}{p_l^{(i)}}\rt)+o(1)\\
        \overset{\eqref{eq:cF-half-cL}}{\leq}\frac{1}{2}c_L(y)-\sum_{0\leq i\leq z}\sum_{l:(i,l)\in T}c_l^{(i)}\log\lt(\frac{1}{p_l^{(i)}}\rt)+o(1).
    \end{align*}
    Thus we have
    \begin{align*}
        &-c_F(y)+2\psi_{\mu}(2)\frac{M_y-K}{\log N}+\theta_{\mu}\sum_{0\leq i\leq z}\sum_{l:(i,l)\in T}c_l^{(i)}\log\lt(\frac{1}{p_l^{(i)}}\rt)\\ 
        &\leq -\frac{1}{2}c_L(y)+2\psi_{\mu}(2)\frac{M_y-K}{\log N}+\theta_{\mu}\sum_{0\leq i\leq z}\sum_{l:(i,l)\in T}c_l^{(i)}\log\lt(\frac{1}{p_l^{(i)}}\rt) \\
        &\leq \frac{1}{2}(\theta_{\mu}-1)c_L(y)+2\psi_{\mu}(2)\frac{M_y-K}{\log N}+O(\delta).
    \end{align*}
    By \eqref{eq:cLy-bound} (whose proof does not use the hypothesis of Claim~\ref{it:calcu1}), we have 
    \begin{multline}
        -c_F(y)+2\psi_{\mu}(2)\frac{M_y-K}{\log N}+\theta_{\mu}\sum_{0\leq i\leq z}\sum_{l:(i,l)\in T}c_l^{(i)}\log\lt(\frac{1}{p_l^{(i)}}\rt) \\
        \leq \frac{\theta_{\mu}-1}{2}\lt[1-\frac{M_y}{\log N}\cdot \log (1/p_{\Max})\rt]+2\psi_{\mu}(2)\frac{M_y-K}{\log N}+O(\delta).
    \end{multline}
     By taking $M_y=0$ and $K$ in the right-hand side, we conclude the proof of Claim~\ref{it:calcu2}.

    \vspace{5pt}
    \noindent \emph{Claim~\ref{it:calcu3}.}
    By \eqref{eq:cE} we have 
    \[c_E(y)=\lt(\frac{M_{y}-K}{\log N}\rt) \psi_{\mu}(2) +o_{\chi}(1).\]
    By \eqref{eq:cLy-bound} we have
    \[c_L(y)\leq 1-\frac{1}{\log N}(M_{y}-2\rho\log N)\mathbb E_{\mu} \log (1/p_{\Max})+O(\delta).\]
    Therefore, we have 
    \[
        c_L(y) +c_E(y) \leq \lt(\frac{M_{y}-K}{\log N}\rt) \psi_{\mu}(2) + 1-\frac{M_y}{\log N}\mathbb E_{\mu} \log (1/p_{\Max}) +O(\delta)
    \]
    By taking $M_y=0$ and $K$ in the right-hand side, we conclude the proof of Claim~\ref{it:calcu3}.
    \end{proof}

    \begin{proof}[Proof of Proposition \ref{prop:new-firstmoment}]
        By Lemma \ref{lem:AMGM}, it suffices to bound the sum \eqref{AMGM-sum}. 
        Let \[\mathsf q(x)= |\{(t,x[t]):(t,x[t])\in \mathsf Q\}|.\] 
        When $\mathsf q (x) > \frac{\mathsf C\log N}{\log (1/\chi)}$, we have immediately that
        \begin{equation}
            \sum_{i=1}^{N-1} \mathbb P[(i,i+1)\in E(G_{B_x})]^2\leq N\cdot \mathbb E|\mathcal I(B_x)|^2\leq N\cdot \chi^{\frac{2\mathsf C\log N}{\log (1/\chi)}}\leq N^{-2\mathsf C+1}.
        \end{equation}
        Therefore,
        \begin{equation}\label{eq:baddigitprofiles}
            \sum_{\mathsf q(x) > \frac{\mathsf C}{\log (1/\chi)}} \sum_{i=1}^{N-1} \mathbb P[(i,i+1)\in E(G_{B_x})]^2\leq k^{K}\cdot N^{-2\mathsf C+1}\leq N^{-\mathsf C}.
        \end{equation}
        It suffices to bound the sum \eqref{AMGM-sum} for those prefixes $x$ with $\mathsf q(x)\leq \frac{\mathsf C\log N}{\log (1/\chi)}$. We bound this term by grouping prefixes $x\in A_{y,z}^{\stable}$ according to their digit profiles. For $y\in \tal,z\in \tar$, consider those strings $x\in A_{y,z}^{\stable}$ with digit profile 
        \[\lt(\{c_l^{(i)}(x)\}_{(i,l)\in T},\{(t,x[t]):(t,x[t])\in \mathsf Q\}\rt).\]
        We call this group $A$ for all $x\in A_{y,z}^{\stable}$ with this digit profile.
        By Lemma \ref{lem:entropy}, the corresponding number of strings is at most
        \[
        \prod_{0\leq i\leq z}\binom{c_{\tot}^{(i)}\log N}{\{c_{l}^{(i)}\log N\}_{(i,l)\in T}}\leq 
        N^{\sum_{0\leq i\leq z}c_{\tot}^{(i)}H(\{c_l^{(i)}\}_{l:(i,l)\in T})+O(\delta)}.
        \]
        For each group $A$, since the digit profile is  fixed, we see that $c_F$ and $c_L$ are already fixed. Therefore, each group belongs to one of the following cases.

        \paragraph{Case 1: $c_L\geq 2\delta$ for all $x\in A$.} 
        By \eqref{upbound1} and \eqref{eq:calcu1}, we have
        \begin{align*}
        &\sum_{x\in A} \sum_{i=1}^{N-1} \mathbb P[(i,i+1)\in E(G_{B_x})]^2\\
        &\leq CN^{\sum_{0\leq i\leq z}c_{\tot}^{(i)}H(\{c_l^{(i)}\}_{(i,l)\in T})+o(1)} \cdot N^{4 c_L(x)- c_F(x)+2c_E( x)+O(\delta)} +e^{-\Omega(N^{\delta-o(1)})}\\
        &\leq CN^{-\Omega(\varepsilon)}+e^{-\Omega(N^{\delta-o(1)})}.
        \end{align*}

        \paragraph{Case 2. $c_L<2\delta$ and $c_F\geq \delta-(\log N)^{-1/2}$ for all $x\in A$, and $A_{y,z}^{\stable}\neq \{y\vee z\}$.} 
        By \eqref{upbound2} and \eqref{eq:calcu2},
        \begin{align*}
        &\sum_{x\in A}\sum_{i=1}^{N-1} \mathbb P[(i,i+1)\in E(G_{B_x})]^2\\
        &\leq CN^{\sum_{0\leq i\leq z}c_{\tot}^{(i)}H(\{c_l^{(i)}\}_{(i,l)\in T})+o(1)} \cdot N^{-c_F(x)+2c_E(x)+O(\chi)+13\delta} +e^{-\Omega(N^{\delta-o(1)})}\\
        &\leq CN^{-\Omega(\varepsilon)}+e^{-\Omega(N^{\delta-o(1)})}.
        \end{align*}

         \paragraph{Case 3. $c_L<2\delta$ and either $c_F\geq \delta-(\log N)^{-1/2}$ and also $A_{y,z}^{\stable}= \{y\vee z\}$, or else $c_F< \delta-(\log N)^{-1/2}$.}
        By Lemma~\ref{lem:edgedensity}\ref{it:upbound2.5}, $c_L(y\vee z)<2\delta$ and $A_{y,z}^{\stable}=\{y\vee z\}$ (in either subcase). Therefore, by \eqref{upbound3} and \eqref{eq:calcu3}, 
        \[
        \sum_{x\in A}\sum_{i=1}^{N-1} \mathbb P[(i,i+1)\in E(G_{B_x})]^2
        \leq N^{2c_L(y\vee z)+2c_E(y\vee z)+4\delta} +e^{-\Omega(N^{\delta-o(1)})}
        \leq CN^{-\Omega(\varepsilon)}+e^{-\Omega(N^{\delta-o(1)})}.\]

        \vspace{5pt}
        \noindent Notice that there are at most \[\log^{kz}(N)\cdot \sum_{\mathsf q \leq \frac{\mathsf C\log N}{\log (1/\chi)}}\binom{K}{\mathsf q}k^{\mathsf q}=N^{O(\delta)}\] total digit profiles with $\mathsf q(x)\leq \frac{\mathsf C\log N}{\log (1/\chi)}$. Therefore, by \eqref{eq:baddigitprofiles} and combining the above three cases, we have
        \begin{align*}
            \sum_{x \in A_{y,z,i}^{\stable}} \sum_{i=1}^{N-1}\mathbb P\lt[(i,i+1)\in E(G_{B_x})\rt]^2
            &\leq \sum_{\substack{x \in A_{y,z,i}^{\stable},
            \\
            \mathsf q(x)\leq \frac{\mathsf C\log N}{\log (1/\chi)}}{}} \sum_{i=1}^{N-1}\mathbb P\lt[(i,i+1)\in E(G_{B_x})\rt]^2] +N^{-\mathsf C}\\
            &\leq N^{O(\delta)}\cdot N^{-\Omega(\eps)}\leq N^{-\Omega(\eps)}.\\
        \end{align*}
        We conclude by \eqref{eq:prefix-size} and \eqref{new-AMGM} that 
        \begin{align*}
            \mathbb E[|E(G,G')|] 
            &\leq N^{-\mathsf{C}/2}+N^{O(\delta)}\cdot\sum_{y \in \tal, z \in \tar}\sum_{x \in A_{y,z,i}^{\stable}} \sum_{i=1}^{N-1}\mathbb P\lt[(i,i+1)\in E(G_{B_x})\rt]^2]\\
            &\leq N^{-\mathsf{C}/2}+ N^{O(\delta)}\cdot \sum_{y \in \tal, z \in \tar} N^{-\Omega(\eps)}\\
            &\leq  N^{-\mathsf{C}/2} + |\tal|\cdot |\tar|\cdot N^{O(\delta)}\cdot N^{-\Omega(\eps)}\leq O(N^{-\delta}).    
            \qedhere
        \end{align*}
    \end{proof}

\section{The exponential moment} \label{sub2sec: expmoment}
\label{sec:exp-moment}
	The main goal of this subsection is to prove Lemma \ref{lem:exp-moment}, namely that $|E(G,G')|$ is typically $o(1)$ small even in the sense of exponential moments. 
	Along the way we will also deduce Lemma \ref{lem:L-sparsity}. 
    Recalling Subsection~\ref{subsec:3-shuffle-reduction}, we will assume throughout that for some $1<t_1<t_2<\mathsf{A}(\mu)$, the first shuffle, the $t_1$-th shuffle and the $t_2$-th shuffle are $\chi$-good. This requirement appears in the discussion below \eqref{eq:def:good event}, and in arguments of Lemmas~\ref{lem:cond-weight-upperbd}, \ref{lem:cond-weight-lowerbd} and \ref{lem:cond-resam}.

    We now outline the organization of this section. Section~\ref{sec:prep-lem} collects several preparatory results. In Section~\ref{sec:exp-proc}, we establish Lemmas~\ref{lem:exp-moment} and \ref{lem:L-sparsity} using an “exploration process” argument. The proof of a key technical lemma required for Lemma~\ref{lem:exp-moment} is deferred to Sections~\ref{subsec:moment-esti} and \ref{ref:sub2sec-pf-lem-resam}.

    \subsection{Preparatory lemmas} 
    \label{sec:prep-lem}
    The following simple lemma states that if an event $\widetilde A \in \sigma(X,Y)$ occurs with high probability, we can construct another event $A \in \s(X)$ such that $A$ itself occurs with high probability, and that if $A$ holds, then $\widetilde A$ occurs with high probability even after conditioning on $X$.
    We will use this lemma to exclude rare events in the exploration process; see Lemma~\ref{lem:explore}.
    \begin{lem} \label{lem:steexpBaye}
       Fix any probability space $(\Omega,\mathcal F, \P)$ and sub-$\s$-algebra $\mathcal G \subset \mathcal F$. For any $\delta>0$ and  event $A \in \mathcal F$ such that $\mathbb P(A) \geq 1-p_0$, we have 
        \begin{equation}
        \label{eq:steexpBaye}
            \P( \P(A | \mathcal G) \geq 1- \sqrt{p_0}) \geq 1-\sqrt{p_0}
        \end{equation}
    \end{lem}
    
    \begin{proof}
        Note that by the tower property of conditional expectations,
        \begin{equation}
            \mathbb E \lt( \P (A^c | \mathcal G) \rt) \leq p_0.
        \end{equation}
        Now \eqref{eq:steexpBaye} follows from the Markov inequality.
    \end{proof}

    The next two easy lemmas from \cite{mark2022cutoff} allow us to bound complicated expected edge intersections by simpler ones.

    \begin{lem}[\hspace{-0.05em}{\cite[Lemma~5.2]{mark2022cutoff}}]\label{lem:mix0}
    Let $A$ and $B$ be independent random subsets of a finite set $\mathcal A$. Let $A'$ and $B'$ respectively be independent copies of $A$ and $B$. Then
    \[
    \mathbb E[|A\cap B|]\leq \frac{\mathbb E[|A\cap A'|]+\mathbb E[|B\cap B'|]}{2}.
    \]
    \end{lem}

    \begin{lem}[\hspace{-0.05em}{\cite[Lemma~5.3]{mark2022cutoff}}]\label{lem:mix}
    Let $A$ be a random subset of a finite set $\mathcal A$ and let $\mathcal F$ be a $\sigma$-algebra. Let $A'$ be an independent copy of $A$ and let $A_{\mathcal F}$ and $A_{\mathcal F}'$ be conditionally independent copies of $A$ given $\mathcal F$. Then
    \begin{equation}\label{eq:L2}
        \mathbb E[|A\cap A'|]\leq \mathbb E[|A_{\mathcal F}\cap A'_{\mathcal F}|].
    \end{equation}
    \end{lem}

    \subsection{Exploration process}
    \label{sec:exp-proc}
    In this subsection, we will prove the crucial Lemma \ref{lem:exp-moment} via an exploration process (except for the technical Lemma~\ref{lem:cond-12moment} which will be deferred to the next subsection).
    We also deduce Lemma~\ref{lem:L-sparsity} at the end of the subsection.
    Recall the definition of $\tau_{x}$ and $\iota_x$ in \eqref{eq:numstr-less-x} and $\mathcal I(B_x)$ in \eqref{eq:I-Bx}.
    Then, the interval $I(B_{[p]})=\{\iota([\p]),\dots,\tau([p])\}$ contains all strings with first digit $p$, and $\iota([\p]),\tau([p])$ are determined by the cut sizes $n_\p^{(t)}$ which we have fixed.
    For simplicity, we write $\iota(p)$ and $\tau(p)$ for $\iota([\p])$ and $\tau([p])$, respectively. 
    To estimate the exponential moment of $|E^\p_{\for}(G,G')|$ 
    on a typical event, we consider the exploration process on the interval $[\tau(p),\iota(p)]$ of strings beginning with $\p$, and prove that $|E^\p_{\for}(G,G')|$ can be dominated by a geometric random variable with mean $O(N^{-\delta})$ on this typical event.
    In the remainder of this subsection, we fix a digit $\p$ and always assume that $i \in \{\iota(p),\ldots,\tau(p)\}$ and $K=O(\log N)$.
    For brevity, we write 
    \begin{equation}
    \label{eq:bs-i}
    \mathbf a_i=(s_{\iota(p)},s_{\iota(p)+1},\ldots,s_i)
    \end{equation}
    for the strings in $\cI(B_{[p]})$ up to position $i$.
    Let $G_i$ denote the induced subgraph of $G$ on vertex set $\{\iota(p), \iota(p)+1,\ldots,i\}$.
    We also let $G'$ be an independent copy of shuffle graph $G$ and define $s'_i$, $\mathbf a'_i$ and $G_i'$ for $G'$ similarly. 
    We let $E_{\for}^p(G_i)$ consist of all edges $(i,i+1)\in E(G_i)$ for which $s_i=s_{i+1} \in \mathsf S^\p_{\for}$, and let $E_{\for}^p(G_i,G_i')=E_{\for}^p(G_i) \cap E_{\for}^p(G_i')$; see also Definition~\ref{def:s-for-back} for reference.

    We now explain how to upper-bound the exponential moment of $|E^\p_{\for}(G,G')|$ on a typical event, using a variant of Portenko's lemma \cite{portenko1975diffusion}.
    It requires truncating several atypical events, which ensure that the digits in our strings are evenly dispersed.
    Namely, we will consider a series of high-probability events $F_i \in \s (\mathbf a_i)$, under which we can upper-bound the expected number of unrevealed edges at each moment $i$.
    (The precise choice of $F_i$ will be made only in \eqref{eq:def-Fi-new} later.) 
    With slight abuse of notation, we also define $F'_i$ in the same way as $F_i$ but with $\mathbf a_i$ replaced by $\mathbf a_i'$ in the following lemma.
    At the endpoints, we define 
    \[
    {F}_{\tau(p)-1}=\{ s_{\tau(p)-1} \geq_{\lex} [\p\iright^{(2)}\cdots \iright^{(a_3)}] \},
    \quad\quad 
    {F}'_{\tau(p)-1}=\{ s'_{\tau(p)-1} \geq_{\lex} [\p\iright^{(2)}\cdots \iright^{(a_3)}] \},
    \]
    where $\iright^{(t)}$ and $a_k$ are defined in Definition~\ref{def:ileft-right} and below. 
    Also recall that $a_3 \leq \mathsf{A}(\mu)$.
    We let
    \[
        {F}^{(\p)}=\bigcap_{i=\iota(p)}^{\tau(p)-1} ({F}_i\cap{F}'_i),\quad\forall\,p\in [k]_0.
    \]

    \begin{lem} \label{lem:explore}
        Consider events ${F}_i$, ${F}'_i$ (for each $i \in \{\iota(p),\ldots,\tau(p)-2\}$), ${F}_{\tau(p)-1}$, ${F}'_{\tau(p)-1}$ and ${F}^{(\p)}$ of the form above. Suppose that there exists $\gamma>0$ such that for each $i \in [\iota(p),\tau(p)-2]$, the conditional estimate 
        \begin{equation}
        \label{eq:unrev-edge-cond}
            \E\big[|E^\p_{\for}(G,G')|-|E_{\for}^p(G_i,G_i')| \,\big\vert\, \mathbf a_i, \mathbf{a}'_i\big] \leq \gamma 
        \end{equation}
        holds on the event ${F}_i\cap{F}'_i$.
        Then
        \begin{equation} \label{eq:exp-moment-general}
            \E[e^{t E^\p_{\for}(G,G')}\cdot 1_{{F}^{(\p)}}] \leq 1+2e^t\gamma
        \end{equation}
        for any $N$ sufficiently large, $t>0$ and $e^t\gamma<\frac{1}{10}$.
    \end{lem}

    \begin{proof}
        Define 
        \[
            \mathcal F_i = \s(\mathbf a_i, \mathbf a'_i,1_{{F}^{(\p)}}),
        \]
        so that $(\mathcal F_i)_{i \in \mathcal{I}(B_x)}$ forms a filtration.

        We now consider the random variable $X \equiv |E^\p_{\for}(G,G')| \cdot  1_{{F}^{(\p)}}$ and show that it can be dominated by an geometric variable $Y$ with success probability $\gamma$.
        For $j>0$, let
        \[
            t_j=\inf\{i \in \{\iota(p),\iota(p)+1,\ldots,\tau(p)\}:|E^\p_{\for}(G_i,G'_i)|\cdot1_{{F}^{(\p)}}  \geq j\}
        \]
        be a stopping time with respect to the filtration $(\mathcal F_i)_{i \in \{\iota(p),\iota(p)+1,\ldots,\tau(p)\}}$, and then $s_{t_j} \in \mathcal{S}_{\for}^\p$ on the event ${F}^{(\p)}$.
        Therefore, on the event $\{t_j<\infty\}$, we have ${F}^{(\p)}$ occurs and $t_j \leq \tau(p)-2$.
        Hence,
        \[
        \begin{aligned}
            \P(X>j |\mathcal F_{t_j} ) &\leq \E[|E^\p_{\for}(G,G')|1_{{F}^{(\p)}}-|E_{\for}^p(G_i,G_i')|1_{{F}^{(\p)}}\vert\mathcal F_{t_j}].
        \end{aligned}            
        \]
        Moreover, \eqref{eq:unrev-edge-cond} implies that for all $i \leq \tau(p)-2$, 
        \[
           \E(|E^\p_{\for}(G,G')|-|E_{\for}^p(G_i,G_i')|\big\vert \mathbf a_i, \mathbf a'_i) \cdot 1_{F^{(\p)}} \leq \E(|E^\p_{\for}(G,G')|-|E_{\for}^p(G_i,G_i'))|\big\vert \mathbf a_i, \mathbf a'_i) \cdot 1_{F_i\cap F'_i} \leq \gamma, 
        \]
        thus we have (recalling the notation \eqref{eq:bs-i})
        \[
        \begin{aligned}
            &\quad\E[(|E^\p_{\for}(G,G')|-|E_{\for}^p(G_i,G_i')|)\cdot 1_{{F}^{(\p)}}\vert\mathcal F_{i}]\\ 
            &=\E[(|E^\p_{\for}(G,G')|-|E_{\for}^p(G_i,G_i')|)\vert\mathcal F_{i}]\cdot 1_{{F}^{(\p)}}\\
            &=
            \E[\E[|E^\p_{\for}(G,G')|-|E_{\for}^p(G_i,G_i')|\big\vert \mathbf a_i, \mathbf a'_i] \big\vert \mathcal{F}_i] \cdot 1_{F^{(\p)}} \\
            &=\E[\E[|E^\p_{\for}(G,G')|-|E_{\for}^p(G_i,G_i')|\big\vert \mathbf a_i, \mathbf a'_i] \cdot 1_{F^{(\p)}}\big\vert \mathcal{F}_i]\\ 
            &\leq \gamma
        \end{aligned}
        \]
        Since we have $t_j \leq \tau(p)-2$ on the event $\{t_j < \infty\}$, this implies that 
        \[
            \P(X>j |\mathcal F_{t_j} ) \leq \gamma.
        \]
        Note that $\{X \geq j\}=\{t_j < \infty\} \in \mathcal F_{t_j}$ is $\mathcal F_{t_j}$-measurable, so we have
        \begin{equation}
            \P(X>j |X \geq j ) \leq \gamma,
        \end{equation}
        which implies that the random variable $X$ is dominated by a geometric random variable $Y$ with 
        \[
            \P(Y=j)=\gamma^j(1-\gamma), \quad j \geq 0.
        \]
        Recall that $X = |E^\p_{\for}(G,G')| \cdot  1_{F^{(\p)}}$.
        Using the assumption $e^t\gamma < \frac{1}{10}$, we find 
        \[
        \begin{aligned}
            \E[e^{t |E^\p_{\for}(G,G')}|\cdot 1_{F^{(\p)}}] &\leq \E[e^{tX}] \leq \E[e^{tY}]\frac{1-\gamma}{1-e^t\gamma}\leq 1+2e^t\gamma,
        \end{aligned} 
        \]
        which then completes the proof.
    \end{proof}

    In the rest of this subsection, we focus on proving that 
    \[
        \E[|E^\p_{\for}(G,G')|-|E_{\for}^p(G_i,G_i')|\, \big\vert \, \mathbf a_i, \mathbf{a}'_i] \leq O(N^{-\zeta})
    \]
    holds for some $\zeta>0$ (as in \eqref{eq:small-params}), after a suitable truncation.
    The idea will be to decompose the set of possible unrevealed strings into $N^{O(\rho)}$ disjoint blocks, and apply the first-moment estimate from the previous section to each one.
    It is easy to show (see e.g.\ \cite[Lemma~5.5]{mark2022cutoff}) that the lexicographic interval
    \[
        I_{s_i,[\p(k-1)(k-1)]}\equiv\{s\in [k]_0^K:s_i<_{\lex}s<_{\lex}[\p\iright^{(2)}\cdots \iright^{(a_3)}]\}
    \]
    can be written as a disjoint union of blocks
    \begin{equation} \label{eq:block-for-old}
        I_{s_i,[\p(k-1)(k-1)]}=\bigcup_{\tilde x\in \widetilde{\mathrm{Block}}_{\for}^\p(s_i)} B_{\tilde x}
    \end{equation}
    for some set $\widetilde{\mathrm{Block}}_{\for}^\p(s_i)$ containing at most $2kK=O(\log N)$ strings, such that each $x\in\widetilde{\mathrm{Block}}_{\for}^\p(s_i) $ has length at most $K$. 
    
    For prefixes $\tilde x\in \widetilde{\mathrm{Block}}_{\for}^{\p}$ with length $M=M(\tilde x) \leq K-2\rho \log(N)$, there exists $j_0 \in \mathbb Z^+$ such that 
    \[
    M<j_0\rho\log N \leq M+\rho \log N \leq (j_0+1)\rho\log N<K.
    \]
    Since the shuffle process is $(\chi,\rho,\varphi)$-almost-$\mu$-like, there exists a $\chi$-good shuffle $\hat t=\hat{t}(\tilde x) \in (j_0\rho\log N,(j_0+1)\rho\log N]$, which implies $0 \leq \hat t-M < \rho \log N$ and $ \hat t \leq K.$
    We may therefore further decompose $B_{\tilde x}$ into the following disjoint union of blocks,
    \[
        B_{\tilde x} = \bigcup_{w \in [k]_0^{\hat t-M-1}} B_{[xw]}.
    \]
    The number of blocks in the union above is at most $k^{\hat t-M-1} \leq k^{\rho \log N}=N^{O(\rho)}$.
    
    Thus, the set of prefixes 
    \[
    \begin{aligned}
         \mathrm{Block}_{\for}^\p(s_i)\equiv&\{\tilde x \in \widetilde{\mathrm{Block}}_{\for}^\p(s_i): M(\tilde x) > K-2\rho \log(N)\} \cup \\
         &\{[xw]: \tilde x \in \widetilde{\mathrm{Block}}_{\for}^\p(s_i),\,M(\tilde x) \leq K-2\rho \log(N),\,w \in [k]_0^{\hat t-M}\}.
    \end{aligned}
    \]
    contains at most $O(\log N) \cdot N^{O(\rho)}=N^{O(\rho)}$ prefixes.
    Since the prefixes $\tilde x$ corresponds to disjoint blocks $B_{\tilde x}$, the equation \eqref{eq:block-for-old} implies that
    \begin{equation}\label{eq:block-for}
        I_{s_i,[\p(k-1)(k-1)]}=\bigcup_{ x\in {\mathrm{Block}}_{\for}^\p(s_i)} B_{x}
    \end{equation}
    is a disjoint union of blocks. In addition, every prefix in $\mathrm{Block}_{\for}^\p(s_i)$ has length at most $K$, and any prefix of length $M \leq K - 2\rho \log N$ satisfies that $(M+1)$ is a $\chi$-good shuffle.

    \medskip
    Recall the definition of $\mathsf S^p_\for$ in Definition~\ref{def:s-for-back}.
    We estimate the number of unrevealed edges in each block $B_x$ for $x \in \Blocks(s_i)$, conditionally on $\mathbf a_i $ after truncation, according to the expected size $\lambda_x = \tfrac{1}{N}\mathbb E|\cI(B_x)|$, where $\cI(B_x) = \{j \in [N] : s_j \in B_x\}$.
    
    When $\lambda_x \geq N^{-1+\xi}$, where $\xi > 0$ is a small constant as in \eqref{eq:small-params}, we will apply the upper bound from Proposition~\ref{prop:new-firstmoment}, established in the previous subsection, to control the number of unrevealed edges in $B_x$. (See \eqref{eq:case3-K-M} and below.)
    To use Proposition~\ref{prop:new-firstmoment}, it is necessary to ensure that, for every prefix $x$ with $\lambda_x \geq N^{-1+\xi}$, the continuation of the shuffle process after $x$ is conditionally “$\mu$-like” within the interval $\cI(B_x)$. To guarantee this, we exclude certain rare events (included in the $F_i$ below).
    Recall that $M$ denotes the length of the prefix $x$. For brevity, denote by
    \begin{equation} \label{eq:def-x-digits}
    n^{(t),x}_{l} = \sum_{j \in \cI(B_x)} 1_{\{s_j[t] = l\}},
    \end{equation}
    the number of digits equal to $l$ in the $t$-th shuffle among strings in $B_x$, for $t > M$ and $l \in [k]_0$.
    In Lemma~\ref{lem:cond-12moment} below, we show that both $|\cI(B_x)|$ and $n_l^{(t),x}$ concentrate sharply for typical $\mathbf a_i$, which is sufficient to guarantee that the shuffle after prefix $x$ remains conditionally “$\mu$-like.”
    Note that suitable concentration must hold even for quite small $\lambda_x$.
    
    On the other hand, if $\lambda_x < N^{-1+\xi}$, we instead bound either the first moment of $|\cI(B_x)|$ or the number of edges in $B_x$ to obtain the desired estimate.
    In fact, for typical strings $\mathbf a_i$, we can establish upper bounds on both the first moment of $|\cI(B_x)|$ and the number of edges in the block $B_x$, conditionally on $\mathbf a_i $.
    More precisely, there exists an event $F_i \in \sigma(\mathbf a_i)$ on which these quantities can be effectively bounded. 
    The results are summarized in the following lemma.
    Recall that $E(G_{B_x})$ denotes the edges in $G$ with vertices in $B_x$.
    Also recall the definition of $\bp^{(i)}$ in \eqref{def:hipi}.

    \begin{lem} \label{lem:cond-12moment}
        For any $K \geq (\widetilde C_{\mu}+\varepsilon) \log N$ and $i\in \{\iota(p),\ldots,\tau(p)-2\}$, and any $\xi>0$, there exists an event $F_i=F_i(\xi) \in \sigma(\mathbf a_i)$ and $\zeta=\zeta(\mu,\xi)$, such that the following holds.
        \begin{enumerate}[label=(\alph*)]
            \item 
            \label{it:cond-12moment-a}
            We have
            \begin{equation} 
                \mathbb P(F_{i}) \geq 1- \exp(-N^{\zeta}).
            \end{equation}
            \item 
            \label{it:cond-12moment-b}
            There exists $C(\mu,\chi)>0$ such that if $s_i \in \mathsf S^\p _{\for}$, we have that for all $x \in \mathsf S^\p_{\for}$ and $x>_{\lex} s_i$, 
            \begin{equation}
            \label{eq:cond-moment}
                \mathbb E\left[\left|\cI(B_x) \right|\big\vert \mathbf a_i\right] \leq C(\mu,\chi)N(\lambda_{x})^{1-o_{\chi}(1)}\cdot N^{O(\chi)}.
            \end{equation}
            on the event $F_i$.
            \item 
            \label{it:cond-12moment-c}
            There exists $C(\mu,\chi)>0$ such that if $s_i \in \mathsf S^\p _{\for}$, we have
            \begin{equation}
            \label{eq:cond-moment-equal}
                \mathbb E\left[\left|\{j>i:s_j=s_i \}\right|\big\vert \mathbf a_i\right] \leq C(\mu,\chi)N(\lambda_{s_i})^{1-o_{\chi}(1)}\cdot N^{O(\chi)}.
            \end{equation}
            on the event $F_i$.
            \item
            \label{it:cond-12moment-d}
            There exists $C(\mu,\chi)>0$ such that if $s_i \in \mathsf S^\p _{\for}$, we have that for all $x \in \mathsf S^\p_{\for}$ and $x>_{\lex} s_i$, 
            \begin{equation} \label{eq:edges-in-Bx}
                \mathbb E\left[\left|E(G_{B_x}) \right|\big\vert \mathbf a_i\right] 
                \leq C(\mu,\chi)N^{2-\frac{K-M}{\log N}\sum_{1 \leq i \leq z} h_i\psi_{\bp^{(i)}}(2)+o_\chi(1)} \cdot(\lambda_{x})^{2-o_\chi(1)} 
            \end{equation}
            on the event $F_i$,
            where $M$ is the length of prefix $x$.
            \item 
            \label{it:cond-12moment-e} For any prefix $x$ such that $\lambda_x \geq N^{-1+\xi}$, we have 
            \begin{equation}
                \P \lt( \big|\,|\cI(B_x)|-N \lambda_x \big| \leq (N \lambda_x)^{\frac{4}{5}} \big|\mathbf a_i \rt)\geq 1-\exp(-N^{\xi/2})
            \end{equation}
            on the event $F_i$.
            \item 
            \label{it:cond-12moment-f} 
            For any prefix $x$ such that $\lambda_x \geq N^{-1+\xi}$, we have 
            \begin{equation}
                \P \Big(\Big|n^{(t),x}_{ l}-N \lambda_x \cdot \frac{n^{(t)}_{ l}}{N}\Big| \leq (N \lambda_x)^{\frac{4}{5}} \mbox{  for all $(t,l) \in \mathsf{P}$ and $ M<t\leq K$}\big|\mathbf a_i \Big)\geq 1-\exp(-N^{\xi/2})
            \end{equation}
            on the event $F_i$, where $M$ is the length of prefix $x$.
        \end{enumerate}
    \end{lem}

     We postpone the proof of Lemma~\ref{lem:cond-12moment} to Section \ref{subsec:moment-esti}, after completing the proofs of Lemmas \ref{lem:exp-moment} and \ref{lem:L-sparsity}, as it is fairly involved and largely independent of the other arguments in this subsection.

    \medskip
    Recall that $G'$ is an independent copy of the shuffle graph $G$ and $s'_i$, $\mathbf a'_i$ and $F_i'$ are defined similarly for $G'$.  
    We now prove an upper bound of the expected number of the ``unrevealed" edges in $E^p_{\for}(G,G')$, conditionally on revealing the first $i$ strings in $G$ and $G'$, on the event that the first $i$ strings are typical. 
    Recall that $G_i$ is the induced subgraph of $G$ with vertex set $\iota(p),\iota(p)+1,\ldots,\tau(p)$ and for brevity define the ``unrevealed'' forward subgraphs
    \[
        G_{u,1}=E^\p_{\for}(G)\setminus E_{\for}^p(G_i)), 
        \quad\quad 
        G'_{u,1}=E^\p_{\for}(G')\setminus E_{\for}^p(G_i').
    \]
    Note that we have suppressed $i$; at this point, we have fixed $i$ and aim to check the condition of Lemma~\ref{lem:explore} for that value of $i$.

    \begin{lem} \label{lem:unreveal-edges}
         For all $K \geq (\widetilde C_{\mu}+\varepsilon) \log N$, $\chi,\rho,\varphi>0$ as in \eqref{eq:small-params} sufficiently small, $(\chi,\rho,\varphi)$-almost-$\mu$-like shuffle $(\vec \bn^{(t)})_{t\leq K}$ and $i \in \{\iota(p),\ldots,\tau(p)-2\}$, there exists $\zeta=\zeta(\mu,\varepsilon)>0$ such that on the event $F_{i} \cap (F_{i})'$ constructed in Lemma~\ref{lem:cond-12moment},
         we have
         \begin{equation} \label{eq:unreveal-edges}
             \E[|E(G_{u,1},G_{u,1}')|\big\vert \mathbf{a}_i,\mathbf{a}'_i] \leq O(N^{-\zeta}).    
         \end{equation}
    \end{lem}

    \begin{proof}
        Without loss of generality, suppose that $s_i,s'_i$ is smaller than $[\p\iright^{(2)}\cdots \iright^{(a_3)}]$.
        Define the graph $G_{u,2}$ (resp.\ $G_{u,2}'$) as a conditionally independent copy of $G_{u,1}$ (resp.\ $G_{u,1}'$) given 
        \[\widetilde{\mathcal F}_i\overset{\textrm{def.}}{=} \s(\mathbf{a}_i,\mathbf{a}'_i,(\mathcal{I}(B_x))_{x \in \mathrm{Block}_{\for}^\p(s_i)}) ,\]
        and $\widetilde G_{u,1}$ (resp.\ $\widetilde G_{u,1}'$) as a conditionally independent copy of $G_{u,1}$ (resp.\ $G_{u,1}'$)
        given $\mathbf{a}_i,\mathbf{a}'_i$.
        We now show that at any time $i$, the expected number of unrevealed common edges 
        \begin{equation} \label{eq:change-cond}
            \E[|E(G_{u,1},G_{u,1}')|\big\vert \mathbf{a}_i,\mathbf{a}'_i] \leq 
            \E[|E(G_{u,1},G_{u,2})|\big\vert \mathbf{a}_i].
        \end{equation}
        Actually, we have 
        \[
        \begin{aligned}
            \E[|E(G_{u,1},G_{u,1}')|\big\vert \mathbf{a}_i,\mathbf{a}'_i]
            \overset{\text{Lemma~\ref{lem:mix0}}}&{\leq} \frac{\E[|E(G_{u,1},\widetilde  G_{u,1})|\big\vert \mathbf{a}_i,\mathbf{a}'_i]+\E[|E(G_{u,1}',\widetilde G_{u,1}')|\big\vert \mathbf{a}_i,\mathbf{a}'_i]}{2}\\
            \overset{\text{Lemma~\ref{lem:mix}}}&{\leq}
            \frac{\E[|E(G_{u,1}, G_{u,2})|\big\vert \widetilde{\mathcal F}_i]+\E[|E(G_{u,1}',G_{u,2}')|\big\vert \widetilde{\mathcal F}_i]}{2}
        \end{aligned}
        \]
        By symmetry, we obtain \eqref{eq:change-cond}, so to complete the proof, it suffices to verify that $$\E[|E(G_{u,1},G_{u,2})|\big\vert \mathbf{a}_i] \leq O(N^{-\zeta}).$$

        Recall the definition of $\mathrm{Block}_{\for}^\p(s_i)$ in \eqref{eq:block-for}.
        For any prefix $x \in \mathrm{Block}_{\for}^\p(s_i)$, let $E(G_{u,1},G_{u,2},x)$ be all shared edges in $E(G_{u,1},G_{u,2})$ whose vertices in $G_{u,1}$ and in $G_{u,2}$ are all in $B_x$.
        Since $G_{u,1}$ and $G_{u,2}$ have the same $(|\mathcal{I}(B_x)|)_{x \in \mathrm{Block}_{\for}^\p(s_i)}$ and $\mathrm{Block}_{\for}^\p(s_i) \in \sigma(\mathbf a_i)$, we have the decomposition
        \begin{equation} \label{eq:decomp-com-edges}
            \E[|E(G_{u,1},G_{u,2})|\big\vert \mathbf{a}_i]=\E[|j \in \{i+1,\ldots,\tau(p)\}:s_j=s_i|\big\vert \mathbf{a}_i]+ \sum_{x \in \mathrm{Block}_{\for}^\p(s_i)} \E[|E(G_{u,1},G_{u,2},x)| \big\vert \mathbf{a}_i].
        \end{equation}
        
        We start with estimating the first term on the right-hand side of \eqref{eq:decomp-com-edges}.
        Recall the definition of $t_*$ in \eqref{def:almost}.
        Let $j_0$ be the biggest integer such that $t_*+j_0\rho\log N \leq K$.
        We now have 
        \begin{equation} \label{eq:lambdax-pmax-bound}
        \begin{aligned}
            \log_N \lambda_{{s}_i} &= \frac{1}{\log N}\sum_{1 \leq t \leq K} \log \frac{n^{(t)}_{s_i[t]}}{N} \\
            &\leq \frac{1}{\log N}\sum_{t_* \leq t \leq t_*+j_0\rho\log N} \sum_{1 \leq i \leq z} \log \frac{n^{(t)}_{s_i[t]}}{N} \cdot 1_{\{t \in \mathsf P_i\}}\\
            & \leq \frac{1}{\log N}\sum_{1 \leq i \leq z} \sum_{t_* \leq t \leq t_*+j_0\rho\log N, t \in \mathsf P_i}  (\log p_{\max}^{(i)}+O(\chi))\\
            & \leq  \frac{K}{\log N}\sum_{1 \leq i \leq z} h_i\log p^{(i)}_{\max}+O(\rho)+O(\chi)+O(\varphi)\\
            &=-\frac{K}{\log N}\mathbb E_{\mu}(\log (1/p_{\max}))+o_{\chi}(1),
        \end{aligned}
        \end{equation}
        where in the fourth line we have used that $K/(\log N)\leq 2\oC_{\mu}$ is bounded without loss of generality, and \eqref{def:almost}, and in the last line we have used \eqref{eq:small-params}.
        Since 
        \[
        K \geq \lt(\frac{1}{\mathbb E_{\mu}(\log (1/p_{\max}))}+\varepsilon\rt)\log N,
        \]
        we have $\lambda_{{s}_i} \leq N^{-1-\Omega(\varepsilon)}$ for all $\chi,\rho,\varphi$ sufficiently small.
        Therefore, by \eqref{eq:cond-moment-equal}, we have 
        \begin{equation}\label{eq:si-bound}
            \E[|j \in \{i+1,\ldots,\tau(p)\}:s_j=s_i|\big\vert \mathbf{a}_i] \leq C(\mu,\chi)N^{1-(1+\Omega(\varepsilon))(1-o_{\chi}(1))-O(\chi)}.
        \end{equation}
        Since $o_\chi(1), \chi \ll \Omega(\varepsilon)$, by \eqref{eq:si-bound}, there exists $\zeta>0$ as in \eqref{eq:small-params} such that
        \begin{equation}
            \E[|j \in \{i+1,\ldots,\tau(p)\}:s_j=s_i|\big\vert \mathbf{a}_i] \leq O(N^{-\zeta}).
        \end{equation}
        
        To estimate the second term on the right-hand side of \eqref{eq:decomp-com-edges}, we will show that for all $x \in \mathrm{Block}_{\for}^\p(s_i)$, 
       \begin{equation} \label{eq:resamp-unrev-edges}
           \E[|E(G_{u,1},G_{u,2},x)| \big\vert \mathbf{a}_i] \leq O(N^{-\zeta})
       \end{equation}
        for some $\zeta>0$. Since $\mathrm{Block}_{\for}^\p(s_i)=O(\log N)$, this suffices to conclude the proof.
        We now split into three cases according to the size of $\lambda_x$ and write $M$ for the number of digits in $x$ in each case.
        The main one is Case $3$ where we apply the first moment bound to the smaller system consisting of strings within $B_x$.
        As in \eqref{eq:small-params}, we let $\xi>0$ be such that $\xi \ll \varepsilon$, and suppose that $\chi,\rho,\varphi \ll \xi$ and $\zeta \ll \chi,\rho,\varphi$.
        
        \paragraph{Case 1. $\lambda_x \leq N^{-1-\xi}$.} In this case, we have 
        \begin{equation}
        \begin{aligned}
            \E[|E(G_{u,1},G_{u,2},x)| \big\vert \mathbf{a}_i] 
            &\leq 
            \E[|\cI(B_x)| \big\vert \mathbf{a}_i] \\
            \overset{\eqref{eq:cond-moment}}&{\leq} 
            C(\mu,\chi,\zeta) N\lambda_x^{1-o_\chi(1)}N^{O(\chi)}\\
            &\leq
            O(N^{1-(1+\xi)(1-o_\chi(1))-O(\chi)})\\
            &\leq O(N^{-\Theta(\xi)})\\
            & \leq O(N^{-\zeta})
        \end{aligned}
        \end{equation}
        
        \paragraph{Case 2. $N^{-1-\xi} \leq \lambda_x \leq N^{-1+\xi}$.} 
        Imitating the calculation in \eqref{eq:lambdax-pmax-bound} gives
        \[
            \log_N \lambda_{x} \leq -\frac{M}{\log N}\mathbb E_{\mu}(\log (1/p_{\max}))+o_{\chi}(1).
        \]
        The condition $\lambda_x \geq N^{-1-\xi}$ now implies by rearranging that
        \[
        \begin{aligned}
             \frac{M}{\log N} \leq \frac{1+\xi}{\mathbb E_{\mu}(\log (1/p_{\max}))} +o_{\chi}(1),
        \end{aligned}
        \]
        where we have used that
        $\mathbb E_{\mu}(\log (1/p_{\max}))=1/\widetilde C_{\mu}>0$ depends only on $\mu$ in dividing (and that $\mu(V_k)<1$.
        Since $o_\chi(1)\ll \varepsilon$ is sufficiently small (recall \eqref{eq:small-params}), this implies $K-M \geq \Omega(\varepsilon )\log N$.

        Trivially, the shared edges in $G_{u,1}$ and $G_{u,2}$ are in particular edges in $G_{u,1}$. 
        Recall that $E(G_{B_x})$ is the set of edges in $G$ with vertices in $B_x$.
        Therefore, we have
        \begin{equation} \label{eq:com-to-one-graph}
        \begin{aligned}
            &\E[|E(G_{u,1},G_{u,2},x)| \big\vert \mathbf{a}_i]\\
            &\leq \E[|E((G_{u,1})_{B_x})| \big\vert \mathbf{a}_i]\\
            \overset{\eqref{eq:edges-in-Bx}}&{\leq} C(\mu,\chi)N^{2-\frac{K-M}{\log N}\sum_{1 \leq i \leq z} h_i\psi_{\bp^{(i)}}(2)+o_{\chi}(1)} \cdot(\lambda_{x})^{2-o_{\chi}(1)} \\
            &=N^{2-\frac{K-M}{\log N}\mathbb E_{\mu }\psi_{\bp}(2)+o_{\chi}(1)} (\lambda_{x})^{2-o_{\chi}(1)}
        \end{aligned}
        \end{equation}
        Since $K-M \geq \Omega(\varepsilon ) \log N$ and $\E_{\mu}(\psi_{\bp}(2))>0$ , we have
        \[
        \begin{aligned}
            \E[|E(G_{u,1},G_{u,2},x)| \big\vert \mathbf{a}_i] 
            \leq O(N^{2-2(1-\xi)(1-o_\chi(1))+o_{\chi}(1)-\Omega(\varepsilon)})
            \leq O(N^{-\zeta}).
        \end{aligned}        
        \]
        Here we used $\eqref{eq:small-params}$.

        \paragraph{Case 3. $\lambda_x \geq N^{-1+\xi}$.}
        Similarly to Case 2, we have $M \leq (1+o_{\chi}(1))\overline{C}_{\mu} \log (\lambda_x^{-1})$.
        If $|\cI(B_x)| \in[\frac{1}{2}N \lambda_x, 2N\lambda_x]$, we have 
        \begin{equation}\label{eq:case3-K-M}
            \begin{aligned} 
            K-M
            &\geq  (1+o_{\chi}(1))(\overline{C}_{\mu}+\varepsilon)\log(N\lambda_x)\\
             &\geq\lt(\overline{C}_{\mu}+\frac{\varepsilon}{2}\rt) \log(2N\lambda_x)\\
            &\geq \lt(\overline{C}_{\mu}+\frac{\varepsilon}{2}\rt) \log|\cI(B_{x})|\\
            &\geq\lt(\overline{C}_{\mu}+\frac{\varepsilon}{2}\rt)\xi \log N.
        \end{aligned}
        \end{equation}
        Here we use $\chi,\rho,\varphi \ll\xi \ll \varepsilon$ in the third line and $|\cI(B_x)| \geq \frac{1}{2}N \lambda_x \geq \frac{1}{2}N^{\xi}$.

        Recall the definition of $n^{(t),x}_l$ in \eqref{eq:def-x-digits}.
        We now apply a crucial self-reducibility argument.
        Conditioned on $\cI(B_x)$ and 
        $(n^{(T),x}_l)_{T>M,l\in [k]_0}$, suppose we do a riffle shuffle with $|\cI(B_x)|$ cards and pile sizes 
        \[
        (n^{(t+M),x}_0,n^{(t+M),x}_1,\ldots,n^{(t+M),x}_{k-1})
        \]
        for each $t \in [K-M]$.
        In this way we generate $|\cI(B_x)|$ strings; if we append $x$ to the beginning of each string, then the resulting $|\cI(B_x)|$ strings of length $K$ have the same distribution as the multi-set of strings starting with $x$.
        Thus we will apply the first moment result from Section~\ref{sub2sec: 1stmoment} with $|\cI(B_x)|$ cards.
        To do so, consider the event 
        \[
           \widehat F=\left\{ 
            \begin{aligned}
                &|\cI(B_x)| \in 
                \big[N \lambda_x/2,2N\lambda_x\big].
                \\
                & 
                \Big|n^{(t+M),x}_l-|\cI(B_x)|\frac{n^{(t+M)}_l}{N}\Big| \leq |\cI(B_x)|^{\frac{4}{5}} \mbox{  for all $(t+M,l) \in \mathsf P$}.
            \end{aligned}
            \right\}
        \]
        Note that on the event $\widehat{F}$, we have $\log_N |\cI(B_x)|=\xi+o(1)$ as $N \to \infty$. Therefore, by the second condition in the definition of $\widehat F$, for all $(t+M,l) \in \mathsf P$, the $t$-th shuffle in the new shuffle satisfies 
        \[
            \left|\frac{n_{l}^{(t+M),x}}{|\cI(B_x)|}-\frac{n^{(t+M)}_l}{N}\right| \leq |\cI(B_x)|^{-\frac{1}{5}}.
        \]
        We now consider the shuffle process 
        \[
            (\widehat{n}_0^{(t),x},\widehat{n}_1^{(t),x},\ldots,\widehat{n}_{k-1}^{(t),x}),\quad t \in [K-M],
        \]
        such that 
        \begin{equation}
        \begin{aligned}
            &\Big|\widehat{n}_l^{(t),x}- |\cI(B_x)|\frac{n^{(t+M)}_l}{N}\Big| \leq 2k, \quad \sum_{0 \leq l \leq k-1} \widehat{n}_l^{(t),x} = |\cI(B_x)|, \\
            &\mbox{$\frac{\widehat{n}_l^{(t),x}}{|\cI(B_x)|}$ and $\frac{n^{(t+M)}_l}{N}$ are in the same $D_i$}.
        \end{aligned}
        \end{equation}
        The existence of such a point $(\widehat{n}_l^{(t),x})_{0 \leq l \leq k-1}$ is guaranteed by the fact that $D_i$ is a simplex.
        Hence, the new shuffle is weak-$\frac{1}{6}$-$(\chi,\frac{\rho\log N}{K-M},\frac{\varphi\log N}{K-M})$-almost $\mu$-like.
        Note that \eqref{eq:case3-K-M} implies that $\frac{\log N}{K-M} \leq [(\overline{C}_{\mu}+\frac{\varepsilon}{2})\xi]^{-1}$.
        In addition, since $M \leq (1+o_{\chi}(1))\overline{C}_{\mu}$ and $\rho,\chi \ll \varepsilon$, we can pick $\rho$ and $\chi$ sufficiently small such that $M \leq K-2\rho\log N$.
        Thus, by the discussion below \eqref{eq:block-for}, the first shuffle of the new shuffle process is $\chi$-good.
        Therefore, we can now pick $\chi,\rho,\varphi$ sufficiently small and apply Proposition~\ref{prop:new-firstmoment} (with $\frac{\rho\log N}{K-M}$ in place of $\rho$ and $\frac{\varphi\log N}{K-M}$ in place of $\varphi$) to obtain
        \[
           | E(G_{u,1},G_{u,2},x)| \overset{\eqref{eq:new-firstmoment},\eqref{eq:case3-K-M}}{\leq} O(|\cI(B_x)|^{-\delta}) \leq O(N^{-\zeta})
        \]
        on the event $\widehat{F}$.
        In addition, by Lemma~\ref{lem:cond-12moment}, on the event $F_i$, $\widehat{F}$ occurs with probability higher than $1-\exp(-N^{\xi})$.
        Therefore, we can use the trivial bound $|E(G_{u,1},G_{u,2},x)| \leq N$ on the event $(\widehat{F})^c$ to derive \eqref{eq:resamp-unrev-edges}.
        This concludes the proof.
    \end{proof}

    Combining Lemmas \ref{lem:unreveal-edges} and \ref{lem:explore}, we give the upper bound for the exponential moment, i.e.\ Lemma \ref{lem:exp-moment}.
    Note that although we have focused on $E_{\for}^\p(G,G')$, exactly analogous results for $E_{\back}^\p(G,G')$ can be established by exploration in the opposite direction.
    Thus we only prove Lemma~\ref{lem:exp-moment}(a) and Lemma~\ref{lem:exp-moment}(b) in the cases that involve $\cS_{\p}^{\for}$, $E_{\for}^\p(G,G')$, $\mathsf S^\p_{\for}$, etc.
    
    To establish Lemma~\ref{lem:exp-moment}(c), we need to do the following further truncation.
    Recall that $\xi$ is a sufficiently small positive constant. 
    For all $x$ such that $\lambda_x \leq N^{-1+2\xi}/2$ and the length of $x$ is smaller than $K$, we introduce the following event
    \begin{equation}
    \begin{aligned}
        {J}_x=\{|\mathcal I(B_x)| \leq 3 N^{2\xi}\},\quad\quad
        J_{\xi}=\bigcap_{\substack{x \in \cup_{1 \leq M \leq K}[k]_0^{M}, \\ \lambda_x \leq N^{-1+2\xi}}} J_{x}.
    \end{aligned}
    \end{equation}
    By Lemma~\ref{lem:concen-block-smallmean}, we have that for all $x$ with $\lambda_x \leq N^{-1+2\xi}$:
    \[
        \mathbb P({J}_x) \geq 1-\exp(-cN^{2\xi}).
    \]
    Since the intersection in the definition of $J_{\xi}$ is over at most $N^{O(1)}$ events, each of which occurs with probability higher than $1-\exp(-cN^{2\xi})$, the event $J_{i,\xi}$ occurs with probability higher than $1-\exp(-N^{\xi})$.

    Similarly, we also introduce
    \begin{equation}
        H_{\xi} = \bigcap_{x : \lambda_x \geq N^{-1+\xi}} \big\{ \big| |\cI(B_x)|-N \lambda_x \big| \leq (N \lambda_x)^{\frac{4}{5}}\big\}.
    \end{equation}
    We have $\mathbb P(H_{\xi}) \geq 1-\exp(-N^{\xi})$.

    \begin{proof}[Proof of Lemma \ref{lem:exp-moment}]
        We take $\cS_{\p}^{\for}$ to consist of all sequences of strings $S$ such that the event 
        \begin{equation} \label{eq:typ-seq-set}
            \Big( \bigcap_{i=\iota(p)}^{\tau(p)-2}F_i \Big) \cap {F}_{\tau(p)-1}\cap J_{\xi} \cap H_{\xi}
        \end{equation}
        occurs, where the event ${F}_{\tau(p)-1}=\big\{ s_{\tau(p)-1} \geq [\p\iright^{(2)}\cdots \iright^{(a_3)}] \big\}$ is defined in Lemma \ref{lem:explore}, and $F_i$ ($i \in [\iota(p)+1,\tau(p)-2]$) is constructed in Lemma~\ref{lem:cond-12moment}. 
        It is elementary that 
        \[
        \begin{aligned}
            \{S_K \in \cS_{\p}^{\for}\}&=\Big( \bigcap_{i=\iota(p)}^{\tau(p)-2}F_i \Big) \cap {F}_{\tau(p)-1}\cap J_{\xi}\cap H_{\xi},\\
            \quad \{S'_K \in \cS_{\p}^{\for}\}&=\Big( \bigcap_{i=\iota(p)}^{\tau(p)-2}F_i' \Big) \cap {F}'_{\tau(p)-1}\cap J'_{\xi} \cap H'_{\xi}.
        \end{aligned}
        \]
        The inequality in Lemma~\ref{lem:exp-moment}(a) for $\cS_{\p}^{\for}$ follows from the union bound and the fact that $F_i$, ${F}_{\tau(p)-1},J_{\xi},H_{\xi}$ each occur with probability at least $1-\exp(-N^{\xi})$. 
        The set $\cS_{\p}^{\back}$ is defined by replacing $(k-1)$ with $0$ in the definition of $F_{\iota(p)+1}$ and $F_i$ for $i\in[\iota(p)+2,\tau(p)]$, and substituting the intersection of all these events for the event in \eqref{eq:typ-seq-set}.
        Lemma~\ref{lem:exp-moment}(a) in this case can be proved similarly.
        
        To demonstrate Lemma~\ref{lem:exp-moment}(b), we apply Lemma \ref{lem:explore} with: 
        \begin{equation}
        \label{eq:apply-lem-explore}
        \gamma=O(N^{-\zeta}),
        \end{equation}
        and $F_i$ as constructed in Lemma~\ref{lem:cond-12moment}.
        Note that $e^t\gamma<\frac{1}{10}$ for sufficiently large $N$.
        By Lemma \ref{lem:unreveal-edges}, \eqref{eq:unrev-edge-cond} holds.
        Therefore, the first line of \eqref{eq:exp-moment} holds for $\zeta$ in place of $\delta$ therein by the fact that
        \[
            1_{\{S_K,S_K' \in \cS_{\p}^{\for}\}} \leq 1_{F^{(\p)}},
        \]
        where $F^{(\p)}$ is defined above Lemma~\ref{lem:explore}.
        The second line of \eqref{eq:exp-moment} can be derived similarly.

        We now turn to the proof of Lemma~\ref{lem:exp-moment}(c).
        We show the regularity first.
        Recall the definition of $J_\xi$ and $H_{\xi}$ above the proof of this lemma. 
        We start with establishing \eqref{eq:def-regu-left}. Actually, we have
        \begin{multline}\label{eq:left-regu-decom}
           \lt|i \in \{\iota(x),\iota(x)+1,\ldots,\tau(x)\}:s_i[1]s_i[2]\cdots s_i[a_3] \leq_{\lex} [\p\ileft^{(2)}\cdots\ileft^{(a_3)}]\rt| \\
             = \sum_{2 \leq k \leq a_3}\sum_{0 \leq j < \ileft^{(k)}}|\mathcal I (B_{[\p\ileft^{(2)}\cdots\ileft^{(k-1)}j]})| + |\mathcal I(B_{[\p\ileft^{(2)}\cdots\ileft^{(a_3)}]})|. 
        \end{multline}
        For simplicity, we write
        \begin{equation}
            \mathcal S_{\mathrm{left}}^\p 
            \equiv
            \big\{[\p\ileft^{(2)}\cdots\ileft^{(k-1)}j]: 2 \leq k \leq a_3, 0 \leq j < \ileft^{(k)} \} \cup \{[\p\ileft^{(2)}\cdots\ileft^{(a_3)}]\big\},
        \end{equation}
        so that the right-hand side of \eqref{eq:left-regu-decom} can be written as $\sum_{x \in \mathcal S_{\mathrm{left}}^\p} |\mathcal I(B_x)|$.
        By the fact that $a_3 \leq \mathsf{A}(\mu)$ is bounded, we have $|\mathcal S_{\mathrm{left}}^\p|=O(1)$. 

        For the strings in $\mathcal S_{\mathrm{left}}^\p$ such that $\lambda_x \geq N^{-1+\xi}$, the definition of the event $H_\xi$ implies that
        \begin{equation}\label{eq:left-big}
            |\mathcal I(B_x)| \leq N \lambda_x + (N\lambda_x)^{\frac{4}{5}}\leq N \lambda_x(1+N^{-\xi/6})
        \end{equation}
        on the event $H_\xi$.
        For the strings in $\mathcal S_{\mathrm{left}}^\p$ such that $\lambda_x < N^{-1+\xi}$, the definition of the event $J_\xi$ implies 
        \begin{equation}\label{eq:left-small}
            |\mathcal I(B_x)| \leq 3N^{2\xi}.
        \end{equation}
        Combining \eqref{eq:left-big}, \eqref{eq:left-small} with \eqref{eq:left-regu-decom} gives
        \begin{equation}
            \lt|i \in \{\iota(x),\iota(x)+1,\ldots,\tau(x)\}:s_i[1]s_i[2]\cdots s_i[a_3] \leq_{\lex} \p\ileft^{(2)}\cdots\ileft^{(a_3)}\rt| 
            \leq (1+o(1)) N\sum_{x \in \mathcal S_{\mathrm{left}}^\p} \lambda_x + |\mathcal S_{\mathrm{left}}^\p| \cdot 3N^{2\xi}
        \end{equation}
        By \eqref{eq:left-regu-decom} and \eqref{eq:lambda-lr}, we have $\lambda_{\text{left}}^\p=\sum_{x \in \mathcal S_{\mathrm{left}}^\p} \lambda_x$. Moreover, since $a_3 \leq \mathsf{A}(\mu)$, we have $\lambda_{\text{left}}^\p \geq \chi^{\mathsf{A}(\mu)}$, a positive constant independent of $N$.
        Therefore, we further have
        \begin{equation}
            \lt|i \in \{\iota(x),\iota(x)+1,\ldots,\tau(x)\}:s_i[1]s_i[2]\cdots s_i[a_3] \leq_{\lex} \p\ileft^{(2)}\cdots\ileft^{(a_3)}\rt| 
            \leq (1+o(1)) N \lambda_{\text{left}}^\p,
        \end{equation}
        which implies the first line of \eqref{eq:def-regu-left}.
        The second line can be established similarly.
        Therefore, all sequences in $\cS^{\dagger}$ are regular.
    \end{proof}

    \begin{proof}[Proof of Lemma \ref{lem:L-sparsity}]
        Since $K\geq (\widetilde C_{\mu}+\varepsilon)\log N$, we have $N\lambda_{s_i}^{1-o_{\chi}(1)}\leq O(N^{-\delta})$ for all $s_i\in \cS_K$ and $\chi>0$ sufficiently small (see \eqref{eq:lambdax-pmax-bound} and \eqref{eq:si-bound}). By \eqref{eq:cond-moment-equal}, on the event $F_i$ (in Lemma~\ref{lem:cond-12moment} with sufficiently small $\delta>0$), we have 
        \begin{equation}
        \label{eq:small-change-for-next-string-to-be-same}
        \mathbb P[s_{i+1}=s_i|\mathbf a_i]\leq O(N^{-\delta}), 
        \quad\forall s_i \in \mathsf S^\p _{\for}.    
        \end{equation}
        Take $F=\cap_i F_i$ and $L>\mathsf{C}_0/\delta$. For each discrete sub-interval 
        \[
        \{i,i+1,\dots,i+L-1\}\subseteq \{\iota(x),\iota(x)+1,\ldots,\tau(x)\}
        \]
        of $L$ consecutive vertices, the probability for it to contain at least $a$ ordinary edges in $\mathsf S^p_{\for}$ 
        is at most
       \begin{equation}
           \mathbb P(F^c)+ \sum_{i< j_1<\cdots<j_a\leq i+L-1}
           \mathbb P(s_{j_1}=s_{j_1-1},\cdots,s_{j_a}=s_{j_a-1},F)
           \stackrel{\eqref{eq:small-change-for-next-string-to-be-same}}{\leq} 
           \mathbb P(F^c)+\binom{L}{a} (CN^{-\delta(1-o_{\chi}(1))})^{a}
       \end{equation}
       Taking $a=\lfloor L/12\rfloor$, we see that $\{i,i+1,\dots,i+L-1\}$ contains at most $L/12$ edges in $\mathsf S^\p_{\for}$ with probability at least
       \[1-CN\exp(-N^{\delta})-2^L(CN^{-\delta(1-o_{\chi}(1))})^{L/12}\geq 1-O\lt(\frac{1}{N^2}\rt).\] 
       By symmetry the same holds for edges formed by strings in $\mathsf S^\p_{\back}$. We obtain that for each sub-interval $\{i,i+1,\dots,i+L-1\}$ in $[\iota([\p]),S_{\p}]$, it contains at most $L/6$ edges with probability $1-O\lt(\frac{1}{N^2}\rt)$. Taking a union bound, we see that the above happens for all $i$ with probability at least $1-O\lt(\frac{1}{N}\rt)$. On this event, notice that each interval $\{i,i+1,\dots,i+L-1\}$ intersects with $\{\iota(x),\iota(x)+1,\ldots,\tau(x)\}$ for at most two $\p$'s, and thus contains at most $L/3$ edges. Thus we complete the proof.
    \end{proof}

    \subsection{Proof of Lemma~\ref{lem:cond-12moment}}
    \label{subsec:moment-esti}
    In this subsection, we prove a key lemma used in the proof of Lemma~\ref{lem:cond-12moment}, namely Lemma~\ref{lem:cond-resam}, which, roughly speaking, estimates the probability that, conditioned on $\mathbf a_i$, $m$ strings larger than $s_i$ all start with a given prefix $x$.
    We will then use Lemma~\ref{lem:cond-resam} to establish Lemma~\ref{lem:cond-12moment}\ref{it:cond-12moment-b}, \ref{it:cond-12moment-c} and \ref{it:cond-12moment-d} in the next subsection.

    Recall $\mathbf a_i$ defined in \eqref{eq:bs-i}.
    For convenience, we now introduce another way to characterize the distribution of the strings $s_j$ ($j \in [N]\setminus\{\iota(p),\ldots,i\}$) conditioned on $\mathbf a_i$.
    Suppose that there are $y_{p,i}$ strings in $\mathbf a_i$ equal to $s_i$, and let 
    \[
    n^{(t),\rem}_l=n_{l}^{(t)}-\sum_{\iota(p) \leq j \leq i} 1_{\big\{ s_j[t]=l\big\}}+y_{p,i}\cdot 1_{\{s_i[t]=l\}}
    \]
    denote the number of remaining $l$-digits in the $t$-th shuffle after sampling all strings strictly smaller than $s_i$ among $s_{\iota(p)},\cdots,s_i$.
    In particular, $i \in \{\iota(p),\iota(p)+1,\ldots,\tau(p)\}$, and $n_{\p}^{(1),\rem}=\tau(p)-i+y_{p,i}$. 
    For brevity, we set:
    \begin{equation} \label{eq:def-n-rem}
        n^{\rem}\equiv n^{(1),\rem}_{\p}=n_{\p}^{(1)}-(i-\iota(p)+1)+y_{p,i}.
    \end{equation}
    We now provide another way to describe the distribution of the strings $s_j$ ($j \in [N]\setminus\{\iota(p),\ldots,i\}$) conditioned on $\mathbf a_i$.
    We also include a diagram to illustrate this representation in Figure~\ref{fig:m-repre}.

    \begin{lem}\label{lem:another-expr}
        Consider the set $\mathcal M$ of $(N-i+\iota(p)+y_{p,i}-1)\times K$ matrices with entries in $[k]_0$ such that:
        \begin{enumerate}
            \item The first $n^{\rem}$ rows start with $\p$.
            \item The first $n^{\rem}$ rows are lexicographically greater than or equal to $s_i$.
            \item Among the first $n^{\rem}$ rows, at least $y_{p,i}$ are equal to $s_i$.
            \item For every $t\in [K]$ and $j\in [k]_0$, the number of $j$-digits in the $t$-th column is exactly $n_j^{(t),\rem}$ .
        \end{enumerate}
        We then select uniformly at random a matrix from $\widehat{\mathcal{M}}$.
        For each\footnote{Note the fourth condition above with $t=1$ and $j=\p$ implies that \emph{only} the first $n^{\rem}$ rows begin with $\p$; recall $n^{\rem}=n^{(1),\rem}_{\p}$.} $j \in [n^{\rem}]$, denote by $\widehat{s}_j$ the $j$-th row of the selected matrix. 
        If we sort $\{\widehat{s}_j\}_{j \in [n^{\rem}]}$ into increasing lexicographic order, writing the result as $(\overline{s}'_j)_{j \in [n^{\rem}]}$, then we have
        \[
            (\overline s'_{1}, \ldots, \overline s'_{n^{\rem}}) \overset{\mathrm{law.}}{=} (s_{i-y_{p,i}+1},s_{i-y_{p,i}+2},\ldots,s_{\tau(p)})
            \text{\emph{ conditioned on $\mathbf a_i$,}}
        \]
        which gives a natural coupling between the two sequences of strings with $\overline{s}'_{j}=s_{j+i-y_{p,i}}$ for each $j \in [n^{\rem}]$, and thereby between the set $\{\widehat{s}_j\}_{j \in [n^{\rem}]}$ and the sequence $(\overline{s}_{j+i-y_{p,i}})_{j \in [n^{\rem}]}$. 
        We denote this natural coupling by $Q$.
        In particular, the distribution of $\left|\cI(B_x)\right|$ conditioned on $\mathbf a_i$ coincides with that of $\left|\{j \in [n^{\rem}] : \widehat{s}_j \in B_{x}\}\right|$.      
    \end{lem}

    \begin{proof}
       Recall the definitions of $S_K$, $\overline{S}_K$, and $\overline{s}_j$ in Section~\ref{subsec:shuffle-graph}.
        The matrices in $\widehat{\mathcal M}$ are precisely those obtained by removing the $(i - \iota(p) - y_{p,i}+1)$ rows from $\overline{S}_K$ that correspond to $(s_j)_{j \in \{\iota(p),\iota(p)+1,\ldots, i - y_{p,i}\} }$, and then moving all rows whose first digit is $\p$ to the first $n^{\rem}$ positions, while preserving their relative order.
        The four conditions defining $\widehat{\mathcal M}$ ensure that, when the rows $(s_j)_{j \in \{\iota(p),\iota(p)+1,\ldots, i - y_{p,i}\} }$ are inserted back into any matrix in $\widehat{\mathcal M}$ as the first $(i - y_{p,i} - \iota(p)+1)$ rows, and then the first $n^{(1)}_{\p}$ rows are randomly interleaved with the remaining rows, the resulting matrix is a valid realization of $\overline{S}_K$, and the corresponding $S_K$ satisfies that its rows from $\iota(p)$ to $i$ are exactly $(s_j)_{j \in \{\iota(p),\iota(p)+1,\ldots, i - y_{p,i}\} }$.
        The lemma then follows directly from the definition of $s_j$; we omit further details.
    \end{proof}

    \begin{figure}[tb!]
        \centering
        \includegraphics[width=0.9\linewidth]{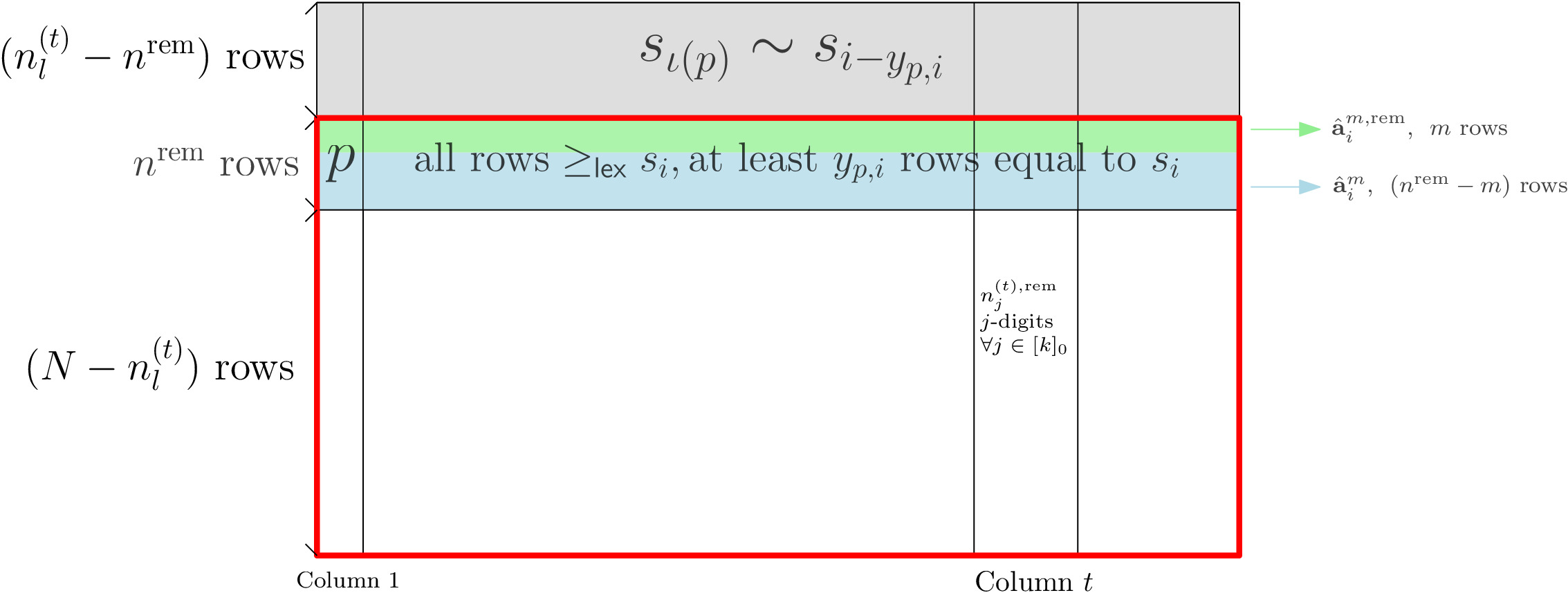}
        \caption{
        As defined in Lemma~\ref{lem:another-expr}, a matrix in $\widehat{\mathcal M}$ is obtained by taking a realization of $\overline{S}_K$ and then moving all rows lexicographically smaller than $s_i$ to the first $(i - y_{p,i} - \iota(p)+1)$ positions, followed by placing all remaining rows that start with $\p$ in the next $n^{\rem}$ positions, while preserving their relative order.
        The portion enclosed by red bold lines corresponds exactly to a matrix in $\widehat{\mathcal M}$, and the four defining conditions of $\widehat{\mathcal M}$ are illustrated in the figure.
        The crucial Lemma~\ref{lem:cond-resam}, which estimates the probability that all rows in the green part start with a given prefix $x$ conditioned on the blue part, is established by carefully counting the number of matrices whose green part satisfies certain prescribed properties.
        }
        \label{fig:m-repre}
    \end{figure}

    For brevity, we define 
    \begin{equation}
        \mu_{\mathbf a_i}(\cdot)=\mathbb P(\cdot|\mathbf a_i)
    \end{equation}
    to be the conditional measure given $\mathbf a_i$, under which the distribution of $\widehat{s}_j$ ($j=1,2,\ldots,n^{\rem}+y_{p,i}$) is defined in the proposition above.
    Unless otherwise specified, the rows and columns in this subsection refer to the rows and columns of the matrix defined in the previous paragraph.

    \medskip
	For any $m \in \Z^+$, we write 
    \begin{equation} \label{eq:def-resample-m-digs}
    \widehat{\mathbf a}_i^{m,\rem}\overset{\text{def.}}{=}(\widehat{s}_{1},\ldots,\widehat{s}_{m}),
    \quad\quad 
        \widehat{\mathbf{a}}^m_i\overset{\text{def.}}{=}(\widehat{s}_{m+1},\widehat{s}_{m+2},\ldots,\widehat{s}_{n^{\rem}})    
    \end{equation}
    to be the first $m$ and last $n^{\rem}-m$ strings starting with $p$ after $\mathbf a_i$; see Figure~\ref{fig:m-repre} for reference.
    Generating them from the distribution $\mu_{\mathbf a_i} \lt( \cdot|\widehat{\mathbf{a}}_i\rt)$, we condition on $\widehat{\mathbf{a}}^m_i$ and resample $\widehat{\mathbf a}_i^{m,\rem}$ (or, more precisely, all rows $\widehat{s}_j$ for $n^{\rem}+y_{p,i}-m+1 \leq j \leq N-i+\iota(p)+y_{p,i}-1$) to study the distribution of the latter.
    We will fix $m$ as a positive integer not depending on $N$ so that $m \ll N$ when $N$ is sufficiently large. This ensures that the total number of $j$-digits in the last $(N-n^{(1)}_\p+m)$ rows of the $t$-th column can be approximated by $n^{(t)}_j(1-\frac{n^{(1)}_\p}{N})$ (recall that $n^{\rem}=n^{(1),\rem}_{\p}=n^{(1)}_{\p}-(i-\iota(p)+1)$ 
    , so $N-i-n^{\rem}+\iota(p)+m-1=N-n^{(1)}_\p+m$); see also the proof of Lemma~\ref{lem:cond-resam}.
    Moreover, the condition $m \ll N$ also allows to estimate the ``weight" of a sample of $\widehat{s}_{n^{\rem}+y_{p,i}}$; see \eqref{eq:cond-onestr-weight} and \eqref{eq:cond-weight-upperbd} for reference.

    \medskip
    We now give the upper bound for the conditional probability, given $\widehat{\mathbf{a}}^m_i$, that all strings in $\widehat{\mathbf a}_i^{m,\rem}$ start with a specific prefix $x$ when $\widehat{\mathbf{a}}^m_i$ is ``typical". This constitutes the central step in the proof of Lemma~\ref{lem:cond-12moment}.
    
    We begin with defining the typical sequences of strings by explaining how to define the event $F_i$ in Lemma~\ref{lem:cond-12moment}.
    Let 
        \begin{equation}
        \label{eq:rem-digits}
            \widetilde{n}_{l}^{(t),\rem} = n_{ l}^{(t)}-\sum_{\iota(p) \leq j \leq  \tau(p)} 1_{\big\{ s_j[t]=l\big\}}.
        \end{equation}
    For some $\delta_* \in (0,\frac{1}{2})$, we consider the following event (recall \eqref{eq:P-Q-def}):
    \begin{equation}\label{eq:def:good event}
    \widetilde{F_i}^*=\Big\{
        \begin{aligned}
                \Big|
                \widetilde{n}_{ l}^{(t),\rem}-n_{ l}^{(t)}\cdot
                \Big(1-\frac{n_\p^{(1)}}{N}\Big)
                \Big| 
                \leq N^{\frac{1}{2}+\frac{\delta_*}{10}}\quad \mbox{for all $(t, l) \in \mathsf P$}
        \end{aligned}
        \Big\}.
    \end{equation}
    Note that 
    \[
        \widetilde{n}_l^{(t),\rem} \sim \Hyp(N-n_\p^{(1)},N,n_l^{(t)}).
    \]
    Therefore, by Lemma \ref{concen-hypergeo}, and the fact that the first shuffle is $\chi$-good, we obtain that $\mathbb P(\widetilde{F}^*_i) \geq 1-\exp(-2N^{\delta_*/6})$.
    Furthermore, we can now define the following $\s(\mathbf a_i)$-measurable event
    \begin{equation} \label{eq:def-Fi}
        F^*_i= \left \{ \mathbb P(\widetilde{F}^*_i|\mathbf a_i) \geq 1- \exp(-N^{\delta_*/6}) \right \}.
    \end{equation}
    By Lemma \ref{lem:steexpBaye}, we have 
    \begin{equation} \label{eq:exp-dec-Fi}
        \mathbb P(F^*_{i}) \geq 1- \exp(-N^{\delta_*/6}),
    \end{equation}
    which implies Lemma~\ref{lem:cond-12moment}(a) by replacing $\eta_*$ therein with $\delta_*/6$.
    We will see in the proof of the following lemma that, on the event $\widetilde{F}_i^*$, the total number of $j$-digits in the last $(N-n^{(t)}_l+m)$ rows of the $t$-th column is $(1+o(1))\widetilde n^{(t)}_j(1-\frac{n^{(1)}_\p}{N})$ when $(t,j) \in \mathsf P$; see also \eqref{eq:def-typ-remstr} and its proof above and below \eqref{eq:m-difference}.  
    
    Notably, we only require $\widetilde{n}_{l}^{(t),\rem}$ to concentrate well when $(t,l) \in \mathsf P$.
    This restriction is necessary because the desired concentration result for $Y_x$ holds only when $\mathbb{E} Y_x \geq N^q$, which is necessary to ensure good concentration for $\widetilde{n}_{l}^{(t), \rem}$; see also Lemma~\ref{concen-prefix}. 
    However for $(t,l) \in \mathsf Q$,
    $\mathbb E Y_x$ can be extremely small, making it impossible to achieve the desired concentration. 
    Consequently, we can only expect concentration to hold for $\widetilde{n}_{l}^{(t),\rem}$.
        
    The following lemma estimates the conditional probability, given $\widehat{\mathbf{a}}^m_i$, that all strings in $\widehat{\mathbf a}_i^{m,\rem}$ start with a specific prefix $x$ on the typical event $F_i$.
    Recall that $\lambda_x=\frac{1}{N}\mathbb E |\cI(B_x)|$, see \eqref{eq:exp-numstr-less-x} for reference.
    Also recall that $y_{p,i}$ stands for the number of strings in $\mathbf a_i$ equal to $s_i$.

     \begin{lem}
        \label{lem:cond-resam}
        For all $K \geq (\widetilde C_{\mu}+\varepsilon) \log N$, $\delta_*>0$ 
        and $(\chi,\rho,\varphi)$-almost-$\mu$-like shuffle $(\vec \bn^{(t)})_{t\leq K}$, there exists $\eta_*(\mu,\varepsilon,\delta_*)>0$ such that the following holds for $N$ sufficiently large. 
        On the event $F_{i}^*$, if $s_i \in \mathsf S ^p_{\for}$, then for each $0<m<n^{\rem}$ there exists an event $\overline{F}_{i}^{m} \in \sigma(\widehat{\mathbf a}^m_i,\mathbf a_i)$ with 
        \begin{equation}
        \label{eq:cond-typ}
            \mu_{\mathbf a_i}(\overline{F}^{m}_{i}) \geq 1 - \exp (-N^{\eta_*}).
        \end{equation}
        such that on the event $\overline{F}^{m}_{i}$, for all $x \in \mathsf S^p _{\for}$ and $x >_{\lex} s_i$:  
        \begin{equation}
        \label{eq:cond-geodecay}
            \mu_{\mathbf a_i} \lt( \widehat{s}_{1},\ldots,\widehat{s}_{m}\text{ starts with }x|\widehat{\mathbf{a}}^m_i\rt) \leq C(\mu,m,\chi)(\lambda_{x})^{m(1-o_\chi(1))} N^{O(m\chi)}.
            \footnote{Here in light of \eqref{eq:small-params}, we write $O(m\chi)$ to indicate a positive sequence $(a_N)_{N \geq 0}$ satisfying $0 \leq a_N \leq C(\mu,\delta_*)m\chi$, where $\delta_*$ refers to the parameter appearing in \eqref{eq:def:good event}. 
            We follow this convention in the remainder of this section.
            }
        \end{equation}
        In addition, for all $\widehat{\mathbf{a}}^m_i$ containing at least $y_{p,i}$ strings equal to $s_i$, on the event $\overline{F}^{m}_{i}$ we have 
        \begin{equation}\label{eq:cond-geodecay-si}
            \mu_{\mathbf a_i} \lt( \widehat{s}_{1}=\cdots=\widehat{s}_{m}=s_i|\widehat{\mathbf{a}}^m_i\rt) \leq C(\mu,m,\chi)(\lambda_{s_i})^{m(1-o_\chi(1))} N^{O(m\chi)}.
        \end{equation}
    \end{lem}
    
    Since $\widehat{\mathbf{a}}^m_i$ and $\widehat{\mathbf{a}}^{m,\rem}_i$ together contain at least $y_{p,i}$ copies of $s_i$ (see the discussion at the beginning of this subsection), we require that $\widehat{\mathbf{a}}^m_i$ alone contains at least $y_{p,i}$. This prevents any strings in $\widehat{\mathbf{a}}^{m,\rem}_i$ from being “forced’’ to equal $s_i$.
    We will use this lemma to establish Lemma~\ref{lem:cond-12moment}\ref{it:cond-12moment-b}, \ref{it:cond-12moment-c} and \ref{it:cond-12moment-d}) in the next subsection.
    
    The rigorous proof of Lemma~\ref{lem:cond-resam} relies on carefully counting and estimating the number of matrices in which the $m$ rows in $\widehat{\mathbf a}_i^{m,\rem}$ either start with $x$ or equal to $s_i$, but the intuition is simpler.
    The condition $s_i \in \mathsf S ^p_{\for}$ ensures that, roughly speaking, a constant proportion of strings in $[k]_0^M$ remains available for $\widehat{s}_{1},\ldots,\widehat{s}_{m}$ to choose from. Since $\lambda_x$ denotes the proportion of strings in the $[k]_0^M$ that starts with $x$, so the probability that a single string starts with the prefix $x$ is at most $C\lambda_x$ for some constant $C$.
    Because $m \ll N$, the $m$ remaining strings $\widehat{s}_{1},\ldots,\widehat{s}_{m}$ are ``nearly independent", which gives the desired result. 
    
    \medskip
    We now turn to establishing Lemma~\ref{lem:cond-resam}. For simplicity, we write
    \[
    \overline{N}=N-i-n^{\rem}+\iota(p)+m-1=N-n^{(1)}_l+m.
    \]
    Moreover, recall $y_{p,i}$ denotes the number of strings in $\mathbf a_i$ equal to $s_i$, and $ \widehat{\mathbf{a}}^m_i,  \widehat{\mathbf{a}}^{m,\rem}_i$ in \eqref{eq:def-resample-m-digs}. We introduce
    \begin{equation} 
    \label{eq:rem-m-dig}
        \overline{n}_{ l}^{(t),m,\rem}
        =
        n_{ l}^{(t)}-\sum_{\iota(p)\leq j \leq  i-y_{p,i}} 1_{\big\{ s_j[t]=  l\big\}}-\sum_{j \in [n^{\rem}-m]} 1_{\big\{\widehat{s}_{j+m}[t]= l\big\}}.
    \end{equation}
    as the number of digits $l$ not used by $\mathbf a_i$ and $\widehat{\mathbf a}^m_i$ in the $t$-th column.

    \begin{lem} \label{lem:onestr-weight}
        Let $\mathsf M$ be the number of $\overline{N}\times K$ matrices with elements in $[k]_0=\{0,1,\cdots,k-1\}$ such that
        \begin{itemize}
            \item the first digit of each the first $m$ rows is $p$;
            \item the number of $l$-digits among the first digits of the last $(\overline{N}-m)$ rows is exactly $n_{l}^{(1)}$ for all $l\in [k]_0$ and $l \neq p$;
            \item the number of $l$-digits in the $t$-th column is exactly $\overline{n}_{l}^{(t),m,\rem}$ for all $2 \leq t \leq K$ and $l\in [k]_0$.
        \end{itemize}
        We denote by $\mathcal M$ the set of these matrices.
        For all string sequences $\widetilde{\mathbf a}=(\widetilde{s}_{1},\ldots,\widetilde{s}_{m}) \in (B_{[p]})^m$,
        the number of the matrices described above such that the first $m$ lines are $\widetilde{\mathbf a}$ is $\mathsf M\nu(\widetilde{\mathbf a})$, where
        \begin{equation} \label{eq:onestr-weight}
            \nu(\widetilde{\mathbf a})=\prod_{t=2}^{K}\prod_{j=1}^{m}\frac{ \overline n_{\widetilde{s}_j[t]}^{(t),m,\rem}-\sum_{i<j} 1_{\widetilde{s}_i[t]=\widetilde{s}_j[t]}}{\overline{N}-j+1}= \prod_{t=2}^{K}\prod_{j=1}^{m}\frac{ \overline n_{\widetilde{s}_j[t]}^{(t),m,\rem}-\sum_{i<j} 1_{\widetilde{s}_i[t]=\widetilde{s}_j[t]}}{N-n_l^{(1)}+j}.
        \end{equation}
    \end{lem}

    \begin{proof}
        If we fix the first $m$ lines as $\widetilde{\mathbf a}$, the number of such matrices is that of $(\overline{N}-m)\times K$ matrices such that the number of $l$-digits in the $t$-th column is exactly $(\overline{n}_l^{(t)}-\sum_{j \in [m]}1_{\widetilde s_j[t]=l})$ for all $t\in [K]$ and $j\in [k]_0$, i.e.
        \[
            \prod_{t=1}^{K} \frac{(\overline{N}-m)!}{\prod_{ l \in [k]_0} ((\overline{n}_{ l}^{(t)}-\sum_{j \in [m]}1_{\widetilde s_j[t]={ l}})!)},
        \]
        where we let $\overline{n}_{p}^{(1)}=0$.
        Since we have $\sum_{j \in [m]}1_{\widetilde s_j[1]={p}}=0$ and
        \begin{equation}\label{eq:sfM}
        \mathsf M= \frac{(\overline{N}-m)!}{\prod_{l \in [k]_0,l \neq p} (\overline{n}_{ l}^{(1)}!)}\cdot\prod_{t=2}^{K} \frac{\overline{N}!}{\prod_{ l \in [k]_0} (\overline{n}_l^{(t),m,\rem}!)},
        \end{equation}
         $\mathsf M\nu(\widetilde{\mathbf a})$ is the number of the matrices whose first $m$ lines are $\widetilde{\mathbf a}$.
    \end{proof}

    Based on this lemma, we can calculate the conditional probability $\mu_{\mathbf a_i} \lt( \overline{\mathbf a}^{\rem,m}_i=\widetilde{\mathbf a}| \overline{\mathbf{a}}^m_i \rt)$ for all $\widetilde{\mathbf a}=(\widetilde{s}_{1},\ldots,\widetilde{s}_{m}) \in (\cS^p_{\geq_{\lex}s_i})^{m}$, where $\cS^p_{\geq_{\lex}s_i}$ is the set of the strings of length $K$ that are larger than or equal to $s$ and start with digit $p$.
    For simplicity, we introduce the following family of sequences of strings:
    \begin{equation} \label{eq:eq-atleast-z}
        \cS^p_{\geq_{\lex}s_i}(m,r)=\left\{ \widetilde{\mathbf a}\in (\cS^p_{\geq_{\lex}s_i})^{m}\big \vert \mbox{There are at least $r$ strings in $\widetilde{s}_{1},\ldots,\widetilde{s}_{m}$ equal to $s_i$.}\right\}
    \end{equation}
    Recall the definition of $\overline{n}_{l}^{(t),m,\rem}$ in \eqref{eq:rem-m-dig} and the fact that there are at least $y_{p,i}$ strings among $\widehat s_j$, $j \in [n^{\rem}]$.
    Moreover, recall that $y_{p,i}$ is the number of strings equal to $s_i$ in $\mathbf a_i$
    
    \begin{lem} \label{lem:cond-expand}
        Suppose that there are $y_1$ strings equal to $s_i$ among $\widehat {\mathbf a}_i$.
        Then for all $\widetilde{\mathbf a}\in \cS^p_{\geq_{\lex}s_i}(m,y_{p,i}-y_1)$, we have
        \begin{equation}
        \label{eq:cond-expand}
            \mu_{\mathbf a_i} \lt( \widehat{\mathbf a}^{\rem,m}_i=\widetilde{\mathbf a}| \widehat{\mathbf{a}}^m_i \rt)  
            = 
            \frac{ \nu_{\mathbf a_i, \widehat{\mathbf a}^m_i}(\widetilde{\mathbf a})}
            {\sum_{\widetilde{\mathbf a} \in \cS^p_{\geq_{\lex}s_i}(m,y_{p,i}-y_1)} \nu_{\mathbf a_i, \overline{\mathbf a}^m_i}(\widetilde{\mathbf a})},
        \end{equation}
        where
        \begin{equation} \label{eq:cond-onestr-weight}
            \nu_{\mathbf a_i, \widehat{\mathbf a}^m_i}(\widetilde{\mathbf a})=\prod_{t=1}^{K-1}\prod_{j=1}^{m}\frac{ \overline n_{\widetilde{s}_j[t]}^{(t),m,\rem}-\sum_{i<j} 1_{\widetilde{s}_i[t]=\widetilde{s}_j[t]}}{N-n_l^{(1)}+j}.
        \end{equation}
        For all $\widetilde{\mathbf a}\in (\cS^p_{\geq_{\lex}s_i})^m \setminus \cS^p_{\geq_{\lex}s_i}(m,y_{p,i}-y_1)$, we have $\mu_{\mathbf a_i} \lt( \widehat{\mathbf a}^{\rem,m}_i=\widetilde{\mathbf a}| \widehat{\mathbf{a}}^m_i \rt)=0$.
    \end{lem}

    To prove this, we need the following characterization of $\mu_{\mathbf a_i} \lt( \cdot|\widehat{\mathbf{a}}_i\rt)$, which is an immediate result of Lemma~\ref{lem:another-expr}; also see Figure~\ref{fig:m-repre}.
    Recall that $\widehat{\mathbf a}_i^{m,\rem}=(\widehat{s}_{1},\ldots,\widehat{s}_{m})$ and the distribution of $\widehat{s}_j$ ($j \in [n^{\rem}]$).

    \begin{lem}\label{lem:ano-repre-2}
        Consider all $\overline{N}\times K$ matrices with elements in $[k]_0$ such that: 
        \begin{enumerate}
            \item The first $m$ rows each start with $p$ and are larger than or equal to $s_i$ (in the lexicographical order);
            \item The number of $l$-digits in the first column of each of the last $N-n^{(1)}_p$ rows is $n_{ l}^{(1)}$ for all $l \in [k]_0 \backslash \{p\}$.
            \item There are at least $(y_{p,i}-y_1)$ rows equal to $s_i$;
            \item The number of $l$-digits in the $t$-th column is exactly $\overline{n}_{l}^{(t),m,\rem}$ for all $2 \leq t \leq K$ and $ l\in [k]_0$.
        \end{enumerate}
        Then for a uniformly random such matrix, the joint law of the first $m$ rows is that of $\widehat{\mathbf a}_i^{\rem,m}$ under $\mu_{\mathbf a_i} \lt( \cdot|\widehat{\mathbf{a}}_i\rt)$. 
    \end{lem}

    \begin{proof}
        The matrices considered here are precisely those obtained by removing the $(m + 1)$-th through $n^{\rem}$-th rows from the matrices in $\widehat{\mathcal M}$, that is, the blue portion in Figure~\ref{fig:m-repre}.
        The four requirements stated here correspond to the four conditions in Lemma~\ref{lem:another-expr}.
        In particular, the third requirement follows because there are at least $y_{p,i}$ rows equal to $s_i$ in total, but only $y_1$ rows equal to $s_i$ among the removed rows (i.e., $\widehat {\mathbf a}^m_i$), so the remaining rows must include at least $(y_{p,i}-y_1)$ rows equal to $s_i$.
        Hence, the desired result follows directly from the definition of $\mu_{\mathbf a_i}(\,\cdot\,|\,\widehat{\mathbf a}_i^m)$ together with Lemma~\ref{lem:another-expr}.
        We omit further details for brevity.
    \end{proof}

    \begin{proof}[Proof of Lemma~\ref{lem:cond-expand}]
        The fourth condition in Lemma~\ref{lem:ano-repre-2} implies that for all $\widetilde{\mathbf a}\in (\cS^p_{\geq_{\lex}s_i})^m \setminus \cS^p_{\geq_{\lex}s_i}(m,y_{p,i}-y_1)$, we have $\mu_{\mathbf a_i} \lt( \widehat{\mathbf a}^{\rem,m}_i=\widetilde{\mathbf a}| \widehat{\mathbf{a}}^m_i \rt)=0$.
        By Lemma \ref{lem:onestr-weight}, the number of matrices in Lemma~\ref{lem:ano-repre-2} whose first $m$ rows are $\widetilde{\mathbf a}$ is $\mathsf M\nu_{\mathbf a_i, \widehat{\mathbf a}^m_i}(\widetilde{\mathbf a})$, where $\mathsf M$ is a positive constant which does not depend on $\widetilde{\mathbf a}$.
        Summing $\mathsf M\nu_{\mathbf a_i, \widehat{\mathbf a}^m_i}(\widetilde{\mathbf a})$ over all $\widetilde{\mathbf a} \in \cS^p_{\geq_{\lex}s_i}(m,y_{p,i}-y_1)$ gives the total number of the matrices described above.
        The desired result now follows from Lemma~\ref{lem:ano-repre-2}.
    \end{proof}

    We now construct the event $\overline{F}^m_i$ in the statement of Lemma \ref{lem:cond-resam} and estimate the weight $\nu_{\mathbf a_i, \widehat{\mathbf a}^m_i}(\widetilde{\mathbf a})$.
    Recall the definition of $\overline{n}_{l}^{(t),m,\rem}$ in \eqref{eq:rem-m-dig}.
    For some positive integer $m$ and some $\delta_* \in (0,\frac{1}{2})$, we now define the event
    \begin{equation}
        \label{eq:def-typ-remstr}
            \overline{F}^{m}_i=\lt\{
            \begin{aligned}
                &\left|
                \overline{n}_{l}^{(t),m,\rem}-n_{l}^{(t)}\cdot\lt(1-\frac{n_p^{(1)}}{N}\rt)
                \right| 
                \leq N^{\frac{1}{2}+\frac{\delta_*}{5}}\quad \mbox{for all $(t,l) \in \mathsf P$} 
        \end{aligned}
        \rt\}
    \end{equation}
    The event $\overline{F}^m_i$ is chosen so that we can establish good bounds for the weight $\nu_{\mathbf a_i, \widehat{\mathbf a}^m_i}(\widetilde{\mathbf a})$, as given in \eqref{eq:cond-onestr-weight}, in the case where $\widehat{\mathbf a}_i^m$ contains at least $y_{p,i}$ strings equal to $s_i$.
    Recall the definition of $\lambda_{x}$ in \eqref{eq:exp-numstr-less-x}.
    The following lemma presents an upper bound for $ \sum_{\widetilde{\mathbf{a}} \in (B_x)^m}\nu_{\mathbf a_i, \widehat{\mathbf a}^m_i}(\widetilde{\mathbf a})$.

    \begin{lem}\label{lem:cond-weight-upperbd}
        Suppose that $\widehat{\mathbf a}_i^m$ contains at least $y_{p,i}$ strings equal to $s_i$.
        Under the assumption of Lemma~\ref{lem:cond-resam}, and with $\overline{F}^{m}_i$ defined as in \eqref{eq:def-typ-remstr}, for all $w \geq_{\lex} s_i$ and $w \in \mathsf S^\p _{\for}$, we have
        \begin{equation}\label{eq:cond-weight-upperbd}
            \sum_{\widetilde{\mathbf{a}} \in (B_w)^m}\nu_{\mathbf a_i, \widehat{\mathbf a}^m_i}(\widetilde{\mathbf a})
            \leq C_2(\mu,m,\chi)(\lambda_w)^{m(1-o_{\chi}(1))}.
        \end{equation}
        for all $\widetilde{\mathbf a}\in (\cS^\p_{\geq_{\lex}s_i})^m$.
    \end{lem}

    \begin{proof}
        We first establish that 
        \begin{equation}\label{eq:x-prefix}
            \sum_{\widetilde{\mathbf{a}} \in (B_w)^m}\nu_{\mathbf a_i, \widehat{\mathbf a}^m_i}(\widetilde{\mathbf a}) \leq  \left(\prod_{t=2}^{M}\frac{ \overline n_{w[t]}^{(t),m,\rem}}{N-n_\p^{(1)}}\right)^m.
        \end{equation}
        Recall that $B_w$ is the collection of all the strings starting with the prefix $w$, and let the length of $w$ be $M_w$.
        By Lemma~\ref{lem:onestr-weight}, $\mathsf M\sum_{\widetilde{\mathbf{a}} \in (B_w)^m}\nu_{\mathbf a_i, \widehat{\mathbf a}^m_i}(\widetilde{\mathbf a})$ is exactly the number of $\overline{N} \times K$ matrices in $\mathcal M$ such that the first $m$ lines start with $w$, so similarly to the proof of Lemma~\ref{lem:onestr-weight}, we have
        \begin{equation}
            \mathsf M \sum_{\widetilde{\mathbf{a}} \in (B_w)^m}\nu_{\mathbf a_i, \widehat{\mathbf a}^m_i}(\widetilde{\mathbf a}) = \prod_{t=1}^{M_w} \frac{(\overline{N}-m)!}{\prod_{l \in [k]_0} ((\overline{n}_{l}^{(t)}-\sum_{j \in [m]}1_{\widetilde x[t]={l}})!)} \cdot \prod_{t=M_w+1}^{K} \frac{\overline{N}!}{\prod_{l \in [k]_0} (\overline{n}_l^{(t),m,\rem}!)}.
        \end{equation}
        Therefore, by \eqref{eq:sfM}, we have 
        \begin{equation}
            \sum_{\widetilde{\mathbf{a}} \in (B_w)^m} \nu_{\mathbf a_i, \widehat{\mathbf a}^m_i}(\widetilde{\mathbf a}) = \prod_{t=2}^{M_w}\prod_{j=1}^{m}\frac{ \overline n_{\widetilde{s}_j[t]}^{(t),m,\rem}-\sum_{i<j} 1_{\widetilde{s}_i[t]=\widetilde{s}_j[t]}}{N-n_\p^{(1)}+j} \leq \left(\prod_{t=2}^{M}\frac{ \overline n_{w[t]}^{(t),m,\rem}}{N-n_\p^{(1)}}\right)^m,
        \end{equation}
        which implies \eqref{eq:x-prefix}.

        We now turn to estimating the product on the right side of \eqref{eq:x-prefix}.
       It suffices to estimate the product $\prod_{t=2}^{M_w}\frac{ \overline n_{w[t]}^{(t),m,\rem}}{N-n_\p^{(1)}}$.
        Recall the definition of $\mathsf P$ and $\mathsf Q$ in \eqref{eq:P-Q-def}.

        By \eqref{eq:def-typ-remstr} and $K=O(\log N)$, there exists $C_1(\mu,m,\chi)>0$ such that
        \begin{equation} \label{eq:cond-weight-upperbd-plike}
         \begin{aligned}
            \prod_{(t,w[t]) \in \mathsf P,2 \leq t \leq M_w} \frac{ \overline n_{w[t]}^{(t),m,\rem}}{N-n_l^{(1)}}
            &\leq\prod_{(t,w[t]) \in \mathsf P,2 \leq t \leq M_w} \bigg[\frac{n_{w[t]}^{(t)}}{N} \lt(1+ C_1(\mu,m,\chi) N^{-\frac{1}{2}+\frac{\delta}{5}}\rt)\bigg]\\
            &\leq \left(1+ C_1(\mu,m,\chi) N^{-\frac{1}{2}+\frac{\delta}{4}}\right)\prod_{(t,w[t]) \in \mathsf P,2 \leq t \leq M_w} \Big(\frac{n_{w[t]}^{(t)}}{N} \Big).
        \end{aligned}  
        \end{equation}
        Here we used the fact that $\frac{N-n^{(1)}_\p}{N}, \frac{n_{w[t]}^{(t)}}{N} \in( \chi,1)$ by definition for all $(t,w[t])\in \mathsf P$ and $N$ sufficiently large in the first line. 
        
        Moreover, by definition of $\overline{n}_{l}^{(t),m,\rem}$ in \eqref{eq:rem-m-dig}, we have $\overline{n}_{ l}^{(t),m,\rem} \leq n_{l}^{(t)}$.
        Therefore, we have
        \begin{equation}\label{eq:cond-weight-upperbd-qlike}
        \begin{aligned}
            \prod_{\substack{(t,w[t]) \in \mathsf Q, \\ 2 \leq t \leq M_w}} 
            \frac{ \overline n_{w[t]}^{(t),m,\rem}}{N-n_\p^{(1)}}
             &\leq \prod_{\substack{(t,w[t]) \in \mathsf Q, \\ 2 \leq t \leq M_w}} \Big(\frac{n_{w[t]}^{(t)}}{N} \cdot \frac{N}{N-n_\p^{(1)}}\Big)\\
             &\leq \prod_{\substack{(t,w[t]) \in \mathsf Q, \\ 2 \leq t \leq M_w}} \Big(\frac{n_{w[t]}^{(t)}}{N} \Big)\cdot (1-2\chi)^{-|\mathsf Q\cap\{(t,w[t]):2 \leq t \leq M\}|}.
        \end{aligned}
        \end{equation}
        Here we used the facts that $\frac{n_\p^{(1)}}{N} \leq 1-\chi$ (recall Subsection~\ref{subsec:3-shuffle-reduction}, and the first shuffle is $\chi$-good).
        Therefore, by \eqref{eq:lambda-x-repre} we have
        \begin{equation}
            \lambda_{w} \leq \prod_{\substack{(t,w[t]) \in \mathsf Q, \\ 2 \leq t \leq M_w}} \Big(\frac{n_{w[t]}^{(t)}}{N} \Big) \leq \chi^{|\mathsf Q\cap\{(t,w[t]):2 \leq t \leq M_w\}|}.
        \end{equation}
        Thus, we can obtain
        \begin{equation}
            |\mathsf Q\cap\{(t,w[t]):2 \leq t \leq M_w\}| \leq  \frac{\log \lambda_{w}}{\log \chi}.
        \end{equation}
        Combining this with \eqref{eq:cond-weight-upperbd-qlike} gives
        \begin{equation}
        \begin{aligned}
            \prod_{\substack{(t,w[t]) \in \mathsf Q, \\ 2 \leq t \leq M_w}} \frac{ \overline n_{w[t]}^{(t),m,\rem}}{N-n_\p^{(1)}} 
            &\leq 
            \prod_{\substack{(t,w[t]) \in \mathsf Q, \\ 2 \leq t \leq M_w}} \left(\frac{n_{w[t]}^{(t)}}{N} \right)\cdot (\lambda_{w})^{-\frac{\log (1-2\chi)}{\log \chi}}.
        \end{aligned}
        \end{equation}
        Combining the equation above with \eqref{eq:lambda-x-repre}, \eqref{eq:x-prefix} and \eqref{eq:cond-weight-upperbd-plike}, we can obtain
        \begin{equation}\label{eq:cond-weight-upperbd-2}
        \begin{aligned}
            \sum_{\widetilde{\mathbf{a}} \in (B_w)^m}\nu_{\mathbf a_i, \widehat{\mathbf a}^m_i}(\widetilde{\mathbf a})
            &\leq \left(\left(\prod_{t=2}^{M_w} \frac{n_{w[t]}^{(t)}}{N}\right)^m\cdot(\lambda_w)^{-m\frac{\log (1-2\chi)}{\log \chi}}\right) \cdot \left(1+ C_1(\mu,m,\chi) N^{-\frac{1}{2}+\frac{\delta}{4}}\right)^m\\
            &\leq C_2(\mu,m,\chi)\lambda_w^{m(1-o_\chi(1))}
        \end{aligned}
        \end{equation}
        where we use that $m$ is a positive integer not depending on $N$ in the last line. This completes the proof.
    \end{proof}
        
     We now turn to establishing the lower bound for $\nu_{\mathbf a_i, \widehat{\mathbf a}^m_i}(\widetilde{\mathbf a})$.
    When $(t,s[t]) \in \mathsf Q$, the number of digits can be abnormally small, preventing us from obtaining lower bounds for all $\widetilde{\mathbf a}\in (\cS^\p_{\geq_{\lex}s_i})^m$; see also the discussion below \eqref{eq:exp-dec-Fi}.
    To address this issue, we introduce the following set of strings $\mathsf{PM}(B_x)$ for each prefix $x$: 
    \begin{equation}\label{eq:def-pm-bx}
         \mathsf{PM}(B_x)=\{s \in B_x\,:\,(t,s[t]) \in \mathsf P \text{ for all } t \in [K]\}.
    \end{equation}
    Thus $\mathsf{PM}(B_x)$ consists of strings with prefix $x$ all of whose digits are ``$\chi$-good".
    This restriction allows us to apply the definition of $\overline{F}^{m}_i$ to bound each term in the product appearing in \eqref{eq:cond-onestr-weight}.
    
    Note that it is possible for $\mathsf{PM}(B_x)$ to be empty for certain prefixes $x$, but we will choose the prefixes considered here carefully so that $\mathsf{PM}(B_x) \neq \emptyset$.
    Moreover, restricting attention to strings in $\mathsf{PM}(B_x)$ rather than $\mathcal I(B_x)$ typically does not discard a significant proportion of the strings $(s_j)_{j = i + 1, \ldots, \tau{[\p]}}$ in each shuffle, provided $\chi$ is sufficiently small, since at least a proportion of $1 - (k - 1)\chi$ of strings lie in $\mathsf P$ at each shuffle; see also \eqref{eq:den-lobd} below.
    
    We can provide the following lower bound for $ \nu_{\mathbf a_i, \widehat{\mathbf a}^m_i}(\widetilde{\mathbf a})$ for all $\widetilde{\mathbf a} \in (\mathsf{PM}(B_x))^m$.

    \begin{lem} \label{lem:cond-weight-lowerbd}
        Suppose that $\widehat{\mathbf a}_i^m$ contains at least $y$ strings equal to $s_i$.
        Under the assumption of Lemma~\ref{lem:cond-resam}, and with $\overline{F}^{m}_i$ defined as in \eqref{eq:def-typ-remstr}, there exists a constant $C_4(\mu,m,\chi)>0$ not depending on the sequences of strings $\widehat{\mathbf a}^m_i$, $\mathbf a_i$ and $\widetilde{\mathbf a}$, such that
        \begin{equation}\label{eq:cond-weight-lowerbd}
            \nu_{\mathbf a_i, \widehat{\mathbf a}^m_i}(\widetilde{\mathbf a}) \geq C_4(\mu,m,\chi)  \prod_{j=1}^{m} \lambda_{\widetilde{s}_j}
        \end{equation}
        for all $\widetilde{\mathbf a}\in (\mathsf{PM}(B_x))^m$.
    \end{lem}
    
    \begin{proof}
        By \eqref{eq:def-typ-remstr}, we obtain for all prefixes $x$, string $\widetilde{s}\in\mathsf{PM}(B_x)$ and $j \in [m]$, $$\frac{\overline{n}_{\widetilde{s}_j[t]}^{(t),m,\rem}-\sum_{i<j} 1_{\widetilde{s}_i[t]=\widetilde{s}_j[t]}}{N-n_\p^{(1)}+j}\geq\frac{n_{\widetilde{s}_j[t]}^{(t)}}{N} \lt(1+ C_3(\mu,m,\chi) N^{-\frac{1}{2}+\frac{\delta}{5}}\rt)$$ 
        for some positive constant $C_3(\mu,m,\chi)$ not depending on the sequences of strings $\widehat{\mathbf a}^m_i$, $\mathbf a_i$ and $\widetilde{\mathbf a}$. This holds because, by definition, $\frac{N-n^{(1)}_\p}{N}, \frac{n_{\widetilde{s}_j[t]}^{(t)}}{N} \in( \chi,1)$ for all $j \in [m]$ and $N$ sufficiently large in the first line. 
        Imitating the calculation in \eqref{eq:cond-weight-upperbd-plike} and \eqref{eq:cond-weight-upperbd-2} now gives the desired bound.
    \end{proof}
    
    With Lemmas \ref{lem:cond-weight-upperbd} and \ref{lem:cond-weight-lowerbd} in hand, we are now ready to establish Lemma \ref{lem:cond-resam}.
    
    \begin{proof}[Proof of Lemma \ref{lem:cond-resam}]
        We first prove \eqref{eq:cond-typ}.
        Fix a positive integer $m$ in the proof from now on.
        Recall the definition of $\overline{n}_l^{(t),m,\rem}$ in \eqref{eq:rem-m-dig}, the definition of $F_i$ in \eqref{eq:def-Fi}, the definition of $\overline{F}^{m}_i$ in \eqref{eq:def-typ-remstr} and the fact that there exists a natural coupling $Q$ between $\overline{s}_j$ ($j=1,2,\ldots,n^{\rem}$) under $\mu_{\mathbf a_i}$ and $(s_{i+1},\ldots,s_{\tau(p)})$ conditioned on $\mathbf a_i$, see the part below Lemma~\ref{lem:steexpBaye} for reference.
        Therefore, we obtain that under $Q$, 
        \begin{equation} \label{eq:m-difference}
            \begin{aligned}
                |\overline{n}_l^{(t),m,\rem}- n_{l}^{(t),\rem}| \leq m
            \end{aligned}
        \end{equation}
        almost surely.
        Hence, by \eqref{eq:exp-dec-Fi} and Lemma \ref{lem:steexpBaye}, for all $\mathbf a_i \in F_i^*$, $\mu_{\mathbf a_i}(\overline{F}_i^m)\geq 1- \exp(-N^{\delta_*/6})$.
        Taking $\eta_*=\frac{\delta_*}{6}$ concludes the proof of \eqref{eq:cond-typ}.

        We now turn to the proof of \eqref{eq:cond-geodecay}.
        Let $\overline{\mathcal F}^m_i(\mathbf a_i)$ denote the collection of all sequences $\widetilde{\mathbf a}\in([k]_0^{K})^{n^{\rem}-m}$ such that $\mathbf a_i$ and $\widetilde{\mathbf a}$ are two sequences such that $\overline{F}_i^m$ occurs.
        We now show that for all $\widehat{\mathbf a}_i^m \in \overline{\mathcal F}_i^m(\mathbf a_i)$, \eqref{eq:cond-geodecay} holds.
        
        If the number of strings equal to $s_i$ in $\widehat{\mathbf a}_i^m$ is less than $y$, then $\widehat{\mathbf a}^{\rem,m}_i$ contains at least one string equal to $s_i$.
        Since $x>s_i$, the left-hand side of \eqref{eq:cond-geodecay} equals $0$ and \eqref{eq:cond-geodecay} holds.
        
        In the case where $\widehat{\mathbf a}_i^m$ contains at least $y$ strings equal to $s_i$, the event $\overline{F}^m_i$ is chosen so that we can derive a bound for $\nu_{\mathbf a_i, \widehat{\mathbf a}^m_i}(\widetilde{\mathbf a})$, as given in \eqref{eq:cond-onestr-weight}.
        Recall the definition of $\cS^p_{\geq_{\lex}s_i}(m,r)$ in \eqref{eq:eq-atleast-z}, and observe that $\cS^p_{\geq_{\lex}s_i}(m,y_{p,i}-y_1)=(\cS_{\geq_{\lex}s_i})^{m}$ since $y_1 \geq y_{p,i}$, where $y_1$ represents the number of strings equal to $s_i$ in $\widehat{\mathbf a}_i^m$.

        Recall the definition of $\iright^{(t)}$ in Definition~\ref{def:ileft-right}, $\mathsf S^{\p}_{\for}$ in Definition~\ref{def:s-for-back} and $a_k$ in \eqref{eq:def-ak}.
        Therefore, by the fact that the first shuffle is $\chi$-good, in the nonempty set $\mathsf{PM}(B_{[\p\iright
             ^{(2)}\iright^{(3)}\cdots\iright^{(a_3)}]})$, all strings start with $[\p\iright
             ^{(2)}\iright^{(3)}\cdots\iright^{(a_3)}]$.
        By \eqref{eq:cond-expand}, we have
        \begin{equation}\label{eq:expand}
        \begin{aligned}
            \mu_{\mathbf a_i} \lt( \widehat{s}_{n^{\rem}-m+1},\ldots,\widehat{s}_{n^{\rem}}\text{ starts with }x|\widehat{\mathbf{a}}^m_i\rt) 
             &=\sum_{\widetilde{s} \in (B_x)^m } \mu_{\mathbf a_i} \lt( \widehat{\mathbf a}^{\rem,m}_i=\widetilde{\mathbf a}| \widehat{\mathbf{a}}^m_i \rt)  \\
            \overset{\eqref{eq:cond-expand}}&{\leq} \frac{\sum_{\widetilde{s} \in (B_x)^m } \nu_{\mathbf a_i, \widehat{\mathbf a}^m_i}(\widetilde{\mathbf a})}{\sum_{\widetilde{s} \in (\mathsf{PM}(B_{[\p\iright
             ^{(2)}\iright^{(3)}\cdots\iright^{(a_3)}]}))^m } \nu_{\mathbf a_i, \widehat{\mathbf a}^m_i}(\widetilde{\mathbf a})}.
        \end{aligned}    
        \end{equation}
        
        The numerator has been estimated in \eqref{eq:cond-weight-upperbd}.
        For the denominator, 
        we can now derive from \eqref{eq:cond-weight-lowerbd}:
        \begin{equation}\label{eq:den-lobd}
        \begin{aligned}
            \sum_{\widetilde{\mathbf a} \in (\mathsf{PM}(B_{[\p\iright
             ^{(2)}\iright^{(3)}\cdots\iright^{(a_3)}]}))^m } \nu_{\mathbf a_i, \widehat{\mathbf a}^m_i}(\widetilde{\mathbf a}) 
             &\geq 
             \prod_{j=1}^{m}\Big( \sum_{\widetilde{s}_j \in \mathsf{PM}(B_{[\p\iright
             ^{(2)}\iright^{(3)}\cdots\iright^{(a_3)}]}) }\lambda_{\widetilde{s}_j}\Big)\\\
            &=
            \bigg( \frac{n^{(1)}_{p}}{N} \cdot
            \sum_{\widetilde{s}_j \in  \mathsf{PM}(B_{[\p\iright
             ^{(2)}\iright^{(3)}\cdots\iright^{(a_3)}]}) }
            \Big(
            \prod_{2 \leq t \leq M}\frac{n^{(t)}_{\widetilde{s}_j[t]}}{N} 
            \Big)
            \bigg)^m
            \\
            \overset{\eqref{eq:def-pm-bx}}&{=}   
            \bigg( \frac{n^{(1)}_{p}}{N} \cdot
            \prod_{2 \leq t \leq a_3} \frac{n^{(t)}_{\iright^{(t)}}}{N}  \cdot
            \prod_{a_3+1 \leq t \leq K} \Big(\sum_{ l: (t, l) \in \mathsf {P} } \frac{n^{(t)}_{l}}{N} \Big)
            \bigg)^m
            \\
           \overset{\eqref{eq:P-Q-def}}&{\geq} \bigg( \frac{n^{(1)}_{p}}{N} \cdot
            \prod_{2 \leq t \leq a_3} \frac{n^{(t)}_{\iright^{(t)}}}{N} \cdot
            \prod_{a_3+1 \leq t \leq K} \Big(1-(k-1)\chi \Big)
            \bigg)^m\\
        &=N^{m\log(1-(k-1)\chi)\cdot \frac{K-a_3}{\log N}}  (\lambda_{[\p\iright
             ^{(2)}\iright^{(3)}\cdots\iright^{(a_3)}]})^m.
        \end{aligned}
        \end{equation}
        Since $(1,\p)$ and $(a_i,\iright^{(a_i)})$ (for all $i \geq 1$) are not in $\mathsf Q$, we obtain
        \[
            \lambda_{[\p\iright
             ^{(2)}\iright^{(3)}\cdots\iright^{(a_3)}]} \geq \chi^3(1-(k-1)\chi)^{a_3}.
        \]
        Combining \eqref{eq:expand}, \eqref{eq:cond-weight-upperbd} and \eqref{eq:den-lobd} now gives the desired result \eqref{eq:cond-geodecay}. 

        We now establish \eqref{eq:cond-geodecay-si}.
        Recall the definition of $\cS^\p_{\geq_{\lex}s_i}(m,r)$ in \eqref{eq:eq-atleast-z}.
        Since there are at least $y_{p,i}$ strings in $\widehat{\mathbf{a}}^m_i$ equal to $s_i$, $\cS^\p_{\geq_{\lex}s_i}(m,y_{p,i}-y_1)=(\cS_{\geq_{\lex}s_i})^{m}$.
        Therefore, by \eqref{eq:expand}, \eqref{eq:den-lobd} and applying Lemma~\ref{lem:cond-weight-upperbd} with $w=s_i$, we can prove \eqref{eq:cond-geodecay}.
        This concludes the proof of Lemma~\ref{lem:cond-resam}.
    \end{proof}

    \subsection{Proof of Lemma~\ref{lem:cond-12moment}}
    \label{ref:sub2sec-pf-lem-resam}
    In this subsection, we will deduce Lemma~\ref{lem:cond-12moment}.
    Parts \ref{it:cond-12moment-b},\ref{it:cond-12moment-c} and \ref{it:cond-12moment-d} essentially follow from Lemma~\ref{lem:cond-resam}.
    In order to establish Lemma~\ref{lem:cond-12moment}\ref{it:cond-12moment-e},\ref{it:cond-12moment-f}, we introduce the events
   \begin{equation} 
            {H}_{x}=\left\{
            \begin{aligned}
                &\big||\cI(B_x)|-N \lambda_x \big| \leq (N \lambda_x)^{\frac{4}{5}},\\
                &\Big|n^{(t),x}_{ l}-N \lambda_x \cdot \frac{n^{(t)}_{ l}}{N}\Big| \leq (N \lambda_x)^{\frac{4}{5}} \mbox{  for all $(t,l) \in \mathsf{P}$ and $ M<t\leq K$.}
            \end{aligned}
            \right\}. 
        \end{equation}
    and 
     \[
        H_{i,x}=\{\P({H}_{x}|\mathbf a_i) \geq 1- \exp(-N^{\xi/2})\}.
    \]
    Recall from Lemma~\ref{concen-prefix} that $T_\xi$ is the set of all strings $x$ such that $\mathbb E|\cI(B_x)| \geq N^{\xi}$. For simplicity, we write 
    \begin{equation}\label{eq:def:H-i-xi}
         H_{i,\xi}=\bigcap_{x \in T_\xi} H_{i,x}.
    \end{equation}
    Define the $\s (\mathbf a_i)$-measurable event  
    \begin{equation} \label{eq:def-Fi-new}
        F_{i} = \bigcap_{j=S_{p-1}+1}^{i}(F^*_j \cap H_{j,\xi}),
    \end{equation}
    for all $i \in \{\iota(p),\iota(p)+1,\ldots,\tau(p)-2\}$.

    Recall that $E(G_{B_x})$ denotes the edges in $G$ with vertices in $B_x$. Also recall the definition of $\bp^{(i)}$ in \eqref{def:hipi} and $\mathcal{I}(B_x)$ in \eqref{eq:I-Bx}.
    We now provide the proof of Lemma~\ref{lem:cond-12moment}.
    
    \begin{proof}[Proof of Lemma~\ref{lem:cond-12moment}]
        We fix $\delta_*>0$ so that we can construct the events $F^*_i$ and $\overline{F}_i^1$ (resp.\ $\overline{F}_i^2$) as in Lemma~\ref{lem:cond-resam} so that \eqref{eq:cond-geodecay}, \eqref{eq:cond-typ} hold for $m=1$(resp.\ $m=2$).
        In addition, we fix some $\xi>0$ to construct the event $H_{i,\xi}$, and Lemma~\ref{lem:cond-12moment}\ref{it:cond-12moment-e}, \ref{it:cond-12moment-f} follows from the definition of $H_{i,x}$, $H_{i,\xi}$ and $F_i$.
        We then proceed to prove parts (a)–(d) of Lemma~\ref{lem:cond-12moment} sequentially.

        \vspace{5pt}
        \noindent\textit{Lemma~\ref{lem:cond-12moment}\ref{it:cond-12moment-a}.}
        By Lemma~\ref{lem:cond-resam}, we have $\mathbb P(F_{i}^*) \geq 1 - \exp(-N^{\eta})$.
        Moreover, by Lemma \ref{concen-hypergeo}, we have
        \begin{equation}
            \mathbb{P}({H}_{x}) \geq 1-\exp(-2N^{\xi/2}).
        \end{equation}
        Therefore, by Lemma \ref{lem:steexpBaye}, the event $H_{i,x}$
        occurs with probability higher than $1- \exp(-N^{\xi/2})$.
        By \eqref{eq:def:H-i-xi}, since $|T_\xi| = N^{O(1)}$, the event $H_{i,\xi}$ occurs with probability higher than $1- \exp(-\frac{1}{2}N^{\xi/2})$. 
        Therefore, by \eqref{eq:def-Fi-new} and the union bound, we have $F_i$ occurs with probability at least $1- \exp(-N^{\zeta})$ for all $\zeta$ sufficiently small. 

        \vspace{5pt}
        \noindent\textit{Lemma~\ref{lem:cond-12moment}\ref{it:cond-12moment-b}.}
        By Lemma~\ref{lem:another-expr}, the distribution of $|\cI(B_x)|$ conditioned on $\mathbf a_i$ is the same as that of $\left|\{j \in [n^{\rem}] : \widehat{s}_j \in B_{x}\}\right|$ under $\mu_{\mathbf a_i}$; see Lemma~\ref{lem:another-expr} for the definition of $\widehat{s}_j$.
        By symmetry, we have $\mu_{\mathbf a_i}(\widehat{s}_{j} \in  B_{x})=\mu_{\mathbf a_i}(\widehat{s}_{1} \in  B_{x})$ for all $j \in [n^{\rem}]$.
        Note that 
        \[
        \begin{aligned}
            \mu_{\mathbf a_i}(\widehat{s}_{n^{\rem}} \in  B_{x})
             &=\mathbb E_{\mu_{\mathbf a_i}} [\mu_{\mathbf a_i}(\widehat{s}_{j} \in  B_{x} | \widehat{\mathbf a}_i^1); \overline{F}_{i}^{1}]+\mathbb E_{\mu_{\mathbf a_i}} [\mu_{\mathbf a_i}(\widehat{s}_{j} \in  B_{x}|\widehat{\mathbf a}_i^1);(\overline{F}_{i}^{1})^c]\\
            \overset{\eqref{eq:cond-geodecay},\eqref{eq:cond-typ}}&{\leq} C(\mu,\chi)(\lambda_{x})^{1-o_{\chi}(1)}\cdot N^{O(\chi)}+\exp(-N^{\eta_*}),
        \end{aligned}
        \]
        where $\widehat{\mathbf a}_i^1=(\widehat{s}_1,\ldots,\widehat{s}_{n^{\rem}-1})$.
        Therefore, we have 
        \begin{equation}
        \begin{aligned}\label{eq:1-mom-cal}
            \mathbb E\left[\left|\cI(B_x) \right|\big\vert \mathbf a_i\right] &= \sum_{j \in [n^{\rem}]} \mu_{\mathbf a_i}(\widehat{s}_{j} \in  B_{x})\\
            &\leq N\lt(C(\mu,\chi)(\lambda_{x})^{1-o_{\chi}(1)}\cdot N^{O(\chi)}+\exp(-N^{\eta_*})\rt)\\
            & \leq C(\mu,\chi)N(\lambda_{x})^{1-o_{\chi}(1)}\cdot N^{O(\chi)}.
        \end{aligned}
        \end{equation}
        This concludes the proof of Lemma~\ref{lem:cond-12moment}\ref{it:cond-12moment-b}.

        \vspace{5pt}
        \noindent\textit{Lemma~\ref{lem:cond-12moment}\ref{it:cond-12moment-c}.}
        We introduce the events 
        \[
            A(i,j,s_i)=\{\mbox{There are at least $y$ strings equal to $s_i$ among $\widehat{s}_k$ for $k \in [n^{\rem}]\setminus\{j\}$}\},
        \]
        and 
        \[
            A(i,s_i)=\bigcap_{j \in [n^{\rem}]} A(i,j,s_i).
        \]
        Recall that $y$ is the number of strings equal to $s_i$ in $\mathbf a_i$, and by Lemma~\ref{lem:another-expr}, there exists a natural coupling $Q$ between $\widehat{s}_j$ ($j=1,2,\ldots,n^{\rem}$) under $\mu_{\mathbf a_i}$ and $(s_{i-y_{p,i}+1},\ldots,s_{\tau(p)})$ conditioned on $\mathbf a_i$, and we have 
        \[
            \left|j>i:s_j=s_i \right|=\left|j \in [n^{\rem}]: \widehat{s}_j=s_i \right|-y.
        \]
        Therefore, for all $j \in [n^{\rem}]$, on the event $A(i,j,s_i)^c$, we have $\left|l>i:s_l=s_i \right|=0$ and so
        \[
        \begin{aligned}
            \mathbb E\left[\left|j>i:s_j=s_i \right|\big\vert \mathbf a_i\right]&= \mathbb E_{\mu_{\mathbf a_i}}\left[\lt(\left|j \in [n^{\rem}]: \widehat{s}_j=s_i \right|-y\rt) \cdot 1_{A(i,s_i)}\right]\\
            &\leq \mathbb E_{\mu_{\mathbf a_i}}\left[\left|j \in [n^{\rem}]: \widehat{s}_j=s_i \right| \cdot 1_{A(i,s_i)}\right]\\
            &=\sum_{j \in [n^{\rem}]} \mu_{\mathbf a_i}(\widehat{s}_j=s_i;A(i,s_i))\\
            &\leq \sum_{j \in [n^{\rem}]} \mu_{\mathbf a_i}(\widehat{s}_j=s_i; A(i,j,s_i)).
        \end{aligned}
        \]
        By symmetry, $\mu_{\mathbf a_i}(\widehat{s}_j=s_i; A(i,j,s_i))=\mu_{\mathbf a_i}(\widehat{s}_{1}=s_i; A(i,n^{\rem},s_i))$.
        Therefore, we only need to upper-bound the latter.
        Note that $A(i,1,s_i)$ is exactly the event that $\widehat{\mathbf{a}}^m_i$ contains at least $y$ strings equal to $s_i$, so by \eqref{eq:cond-geodecay-si} and imitating the calculation in \eqref{eq:1-mom-cal}, we can obtain Lemma~\ref{lem:cond-12moment}\ref{it:cond-12moment-c}.
        
          \vspace{5pt}
        \noindent\textit{Lemma~\ref{lem:cond-12moment}\ref{it:cond-12moment-d}.} Note that by Lemma~\ref{lem:another-expr}, for all strings $\widetilde{s}$ larger than $s_i$, the distribution of $\left|E(G,\widetilde{s}) \right|$ conditioned on $\mathbf a_i$ is the same as that of $\left|\{j,j' \in [n^{\rem}] : \widehat{s}_j=\widehat{s}_{j'}=\widetilde{s},j \neq j'\}\right|$ under $\mu_{\mathbf a_i}$. 
        The argument similar to the proof of Lemma~\ref{lem:cond-12moment}\ref{it:cond-12moment-b} gives
        \begin{equation} \label{eq:two-point-estimate}
            \mu_{\mathbf a_i}(\widehat{s}_j=\widehat{s}_{j'}=\widetilde{s})=\mu_{\mathbf a_i}(\widehat{s}_{1}=\widehat{s}_{2}=\widetilde{s}) \leq C(\mu,\chi) (\lambda_{\widetilde{s}})^{2-o_{\chi}(1)}\cdot N^{O(\chi)}
        \end{equation}
        for all $j \neq j'$.
        
        Recall that the length of prefix $x$ is $M$ and note that
        \[
            B_x=\{[xw]:w \in [k]_0^{K-M}\}.
        \]
        Therefore, we have
        \begin{equation} \label{eq:decompose-edges}
        \begin{aligned}
            \E\left[\left|E(G_{B_x}) \right| \big\vert \mathbf a_i\right] 
            &\leq \sum_{w \in [k]_0^{K-M}}\mathbb E _{\mu_{\mathbf a_i}}[\left|\{j,j' \in [n^{\rem}] : \widehat{s}_j=\widehat{s}_{j'}=[xw]],j \neq j'\}\right|]\\
            &=\sum_{w \in [k]_0^{K-M}}\sum_{1 \leq j<j'\leq n^{\rem}} \mu_{\mathbf a_i}(\widehat{s}_j=\widehat{s}_{j'}=[xw]])\\
             \overset{\eqref{eq:two-point-estimate}}&{\leq}C(\mu,\chi) N^{2+O(\chi)}\sum_{w \in [k]_0^{K-M}} (\lambda_{[xw]})^{2-o_{\chi}(1)}.
        \end{aligned}
        \end{equation}
        In addition, using \eqref{eq:exp-numstr-less-x}, the last sum is at most
        \begin{equation} \label{eq:sum-all-str}
        \begin{aligned}
            \sum_{w \in [k]_0^{K-M}} (\lambda_{[xw]})^{2-o_{\chi}(1)}
            &=\sum_{w \in [k]_0^{K-M}}\prod_{t=1}^{M}\Big(\frac{n_{x[t]}^{(t)}}{N}\Big)^{2-o_{\chi}(1)} \prod_{t=M+1}^{K} 
            \Big(\frac{n_{w[t-M]}^{(t)}}{N}\Big)^{2-o_{\chi}(1)}\\
            &\leq(\lambda_x)^{2-o_{\chi}(1)}\sum_{w \in [k]_0^{K-M}} \prod_{M+1 \leq t \leq K} 
            \Big(\frac{n_{w[t-M]}^{(t)}}{N}\Big)^{2-o_{\chi}(1)} 
            \\
            &=(\lambda_x)^{2-o_{\chi}(1)} \prod_{M+1 \leq t \leq K} \bigg(\sum_{ l \in [k]_0}
            \Big(\frac{n_{ l}^{(t)}}{N}\Big)^{2-o_{\chi}(1)}\bigg)
            \\
            &=(\lambda_x)^{2-o_{\chi}(1)} N^{-\frac{1}{\log N}\sum_{t>M}\psi_{\bn^{(t)}/N}(2-o_{\chi}(1))}
            .
        \end{aligned}
        \end{equation}
        By the definition of $\psi_{\bp}(x)$, for all $x \in (1,2]$ and $\bp \in \mathcal D_k$ (i.e.\ $\sum_{ l \in [k]_0}p_{ l}=1$ and $0 \leq p_{ l} \leq 1$ for all $l \in [k]_0$), we have
        \begin{equation}
            \left|\frac{\mathrm d \psi_{\bp}(x)}{\mathrm d x}\right|
            = 
            \frac{\sum_{ l \in [k]_0} p_{ l}^{x} \log(p_{\hat l}^{-1})\cdot 1_{\{p_{l} \neq 0\}}}{\sum_{l \in [k]_0}p_{l}^{x} } 
            \overset{x \in (1,2]}{\leq} 
            \frac{\sum_{ l \in [k]_0} p_{ l} \log(p_{ l}^{-1})\cdot 1_{\{p_{l} \neq 0\}}}{\sum_{ l \in [k]_0}p_{ l}^2 } \leq k^2\cdot \sup_{p \in (0,1]} [p\log(1/p)].
        \end{equation}
        (Note $p\log(1/p)$ is bounded on $p\in (0,1]$.) Since $\bn^{(t)}/N \in \mathcal D_k$, we find that for all $\varsigma \in (0,\frac{1}{2})$:
        
        \begin{equation}
            |\psi_{\bn^{(t)}/N}(2-\varsigma)-\psi_{\bn^{(t)}/N}(2)| = \varsigma \cdot \left|\frac{\mathrm d \psi_{\bp}}{\mathrm d x}(\overline x)\right| = O(\varsigma),
        \end{equation}
        where $\overline x$ is a real number in $[2-\varsigma,2]$. This implies that $\psi_{\bn^{(t)}/N}(2-o_{\chi}(1))=\psi_{\bn^{(t)}/N}(2)\pm o_{\chi}(1)$.
        Plugging this inequality into \eqref{eq:sum-all-str}, we obtain 
        \begin{equation}\label{eq:sum-all-str-new}
            \sum_{w \in [k]_0^{K-M}} (\lambda_{[xw]})^{2-o_{\chi}(1)} \leq (\lambda_x)^{2-o_{\chi}(1)} N^{-\frac{1}{\log N}\sum_{t>M}\psi_{\bn^{(t)}/N}(2)+o_{\chi}(1)}
        \end{equation}
        Recall from \eqref{eq:def-cE} that $c_E(x)=-\frac{1}{\log N}\sum_{t>M}\psi_{\bn^{(t)}/N}(2)$.
        Therefore, combining \eqref{eq:cE}, \eqref{eq:decompose-edges} and \eqref{eq:sum-all-str-new} gives \eqref{eq:edges-in-Bx}.
        This concludes the proof of Lemma~\ref{lem:cond-12moment}\ref{it:cond-12moment-d}.
    \end{proof}

     \section{Lower Bound}

     \label{sec:lower-bound}
	The definition of $\overline{C}_{\mu}$ as a maximum in \eqref{eq:C-p} corresponds to two distinct obstructions to mixing: long increasing contiguous subsequences in the inverse permutation and \emph{cold spots}. Below we explain these obstructions separately in the following subsections, obtaining \eqref{eq:random-LowerBound} which implies \eqref{eq:LowerBound}.

    \subsection{Increasing sequences}
	\label{subsec:increasing}
	It is a rare event for a uniformly random permutation to have a long increasing contiguous subsequence of size $N^{\Omega(1)}$. Recall that $\widetilde C_{\mu}=\frac{1}{\E_{\mu}\log (1/p_{\Max})}$. The following proposition shows that if $K\leq (\widetilde C_{\mu}-\varepsilon) \log N$, then an inverse $K$-shuffled permutation typically contains a long increasing sequence, and so the deck is not mixed. 
	In particular, there will typically be $N^{\Omega(1)}$ strings consisting only of the most common digit in each shuffle.
	Due to the non-binomial cut sizes, the number of such strings is not binomial.
	However, a simple application of Lemma \ref{concen-prefix} gives the concentration.
    \begin{prop}
	\label{prop:new-increasing-LB} 
		If $K<(\widetilde C_{\mu}-\eps) \log N$, and $\bX$ is randomized-$\mu$-like, we have that \[\lim_{N\to\infty} d_{\TV}(\bP(\bX,K),\nu_{\Unif})=1. \]
	\end{prop}
    \begin{proof}
        We fix $\chi,\rho,\varphi\ll\eps$ and condition that the shuffle process is $(\chi,\rho,\varphi)$-almost-$\mu$-like. Recall the definition of $l_{\Max}$ above \eqref{eq:l-Max}.
        Let $x$ be the string with length $K$ such that
        \[x[t]=l_{\Max}^{(t)}, \mbox{ for all } 1\leq t\leq K.\]
        Also recall that $Y_x$ denotes the number of strings in $S_K$ equal to $x$. The first moment of $Y_x$ is at least:
	\begin{equation}\label{incre-first-moment}
        \begin{aligned}
        \mathbb EY_x
        =
        N\cdot\prod_{t\leq K}\frac{n_i^{(t)}}{N}
        &\geq 
        N\cdot \prod_{0\leq i\leq z} (p_{\Max}^{(i)}-\chi)^{Kh_i-O(\varphi)-O(\rho)}
        \\
        &\geq 
        N^{1-\frac{K}{\log N}\cdot \mathbb E_\mu \log (1/p_{\Max})-O(\delta)}
        \geq 
        N^{\Omega(\eps)}.
        \end{aligned}
        \end{equation}
        Here we used the fact that $\chi,\rho,\varphi\ll \delta\ll \eps$ and $K\leq (C_{\mu}-\varepsilon)\log N$.
        Thus by \eqref{concen-prefix-ineq} and by taking $q=\Omega(\eps)$ we obtain that
        \[
        \mathbb P\big(Y_x\geq \frac{1}{2}\mathbb EY_x\big)\geq 1-\exp(-N^{\Omega(\eps)}).
        \]
		The above calculation and Lemma \ref{lem:num-q-like} indicate that with probability $1-o(1)$, there exists an increasing sequence of length $N^{\Omega_\eps(1)}$ after $K$ shuffles. However, if $\pi\sim \nu_{\Unif}$, we have
		\begin{equation*}
			\nu_{\Unif}(\pi \mbox{ has an increasing sequence of length } N^{\Omega_\eps(1)})\leq N\cdot \frac{1}{(N^{\Omega(\eps)}!)}=o(1),
		\end{equation*}
		which completes the proof.
	\end{proof}

    \begin{rmk}\label{rmk-all-vertex}
        When $\mu(V_k)=1$, we can define $Y_x$ in the same way as in Proposition \ref{prop:new-increasing-LB}. The argument above indicates that for all $C>0$, $K=C\log N$, and for sufficiently small $\chi=\chi(C),\rho=\rho(C),\varphi=\varphi(C)$, \eqref{incre-first-moment} can be modified into
        \[\mathbb EY_x\geq N^{1+C\log (1-\chi)}\geq N^{\frac{1}{2}}.\]
        Therefore we can conclude \eqref{eq:all-vertex} using same arguments as in Proposition \ref{prop:new-increasing-LB}.
    \end{rmk}

    \subsection{Cold spots}
	\label{subsec:cold-spots}
	In this subsection, we take $K=\lfloor (C_\mu-\eps) \log N \rfloor$ and further assume $K \geq (\widetilde C_{\mu}+\eps) \log N$. The following criterion for non-mixing follows exactly as in \cite{lalley2000rate}.\footnote{In fact, the result in \cite[Proof of Proposition~2]{lalley2000rate} is more general in that the explicit hypothesis $|H(\vec\bn)|\geq N^{1/2}$ in Lemma \ref{lem:new-lowerboundcond} can be relaxed to $|H(\vec\bn)|\to \infty$ as $N\to \infty$.
    We provide this explicit rate of divergence for clarity.
    }

	\begin{lem}[\hspace{-0.05em}{{\cite[Proof of Proposition~2]{lalley2000rate}}}]
	\label{lem:new-lowerboundcond}
    For all shuffle process with pile size $\vec\bn$, suppose that there exists $H(\vec\bn)\subset [N]$ such that (with $E(H)$ the subgraph of the path graph induced by vertex set $H$):
		\begin{equation}
		\label{eq:H-properties}
		\begin{aligned}
			|H(\vec\bn)|&\geq N^{1/2} \\
			|\partial H(\vec\bn)|&\leq 2|H|^{1/2}\\
			\mathbb P\left[|E(G)\cap E(H(\vec\bn))|\geq |H(\vec\bn)|^{\frac{1}{2}+\delta}\right]&\to 1.
		\end{aligned}
		\end{equation}
	Then if $\nu \sim \bP(\vec\bn, K)$ and $K\leq (\widetilde C_{\mu}-\eps) \log N$, we have
	\begin{equation}
	\label{eq:hypothesis-test-criterion}	
        \nu\Big(\#\{i\in H: \sigma(i)<\sigma (i+1)\}\leq |H(\vec\bn)|/2+|H(\vec\bn)|^{1/2+{\delta/2}}\Big)
        \to 0,
    \end{equation}  
        where the limit is uniform for all $\vec\bn$.
	\end{lem}

    We will apply Lemma \ref{lem:new-lowerboundcond} for each $(\chi,\rho,\varphi)$-almost-$\mu$-like $\vbn$. We will prove in Lemma \ref{lem:constructcoldspot} that for each for each $(\chi,\rho,\varphi)$-almost-$\mu$-like $\vbn$, we can construct a subset $H(\vec\bn)$ satisfying the conditions in Lemma \ref{lem:new-lowerboundcond}. The proof of Lemma \ref{lem:constructcoldspot} will be posrponed in Section \ref{subsec:constructcoldspot}.

    \begin{lem}\label{lem:constructcoldspot}
        For each $(\chi,\rho,\varphi)$-almost-$\mu$-like $\vbn$, there exists $H(\vec\bn)\subset [N]$ such that \eqref{eq:H-properties} holds.
    \end{lem}
    
    We then let $A_{\vec\bn}$ be the event that
    \[
    |\{i\in H(\vec\bn): \sigma(i)<\sigma (i+1)\}|
    >|H(\vec\bn)|/2+|H(\vec\bn)|^{1/2+{\delta/2}}.
    \]
    We also denote by $A_{\mu,\chi,\rho,\varphi}$ the event that $A_{\vec\bn}$ occurs for at least one $(\chi,\rho,\varphi)$-almost-$\mu$-like $\vec\bn$.

    \begin{lem}\label{hypo-test1}
        Suppose that $\bX$ is randomized-$\mu$-like, $\nu\sim \bP(\bX,K)$ and $K\leq (\widetilde C_{\mu}-\eps) \log N$. We have
        \begin{equation}
            \nu\lt(A_{\mu,\chi,\rho,\varphi}\rt)\to 1.
        \end{equation}
    \end{lem}
    \begin{proof}
        By Lemma \ref{lem:new-lowerboundcond}, we have that for all $(\chi,\rho,\varphi)$-almost-$\mu$-like pile sizes $\vec\bn$,
        \begin{equation*}
            \nu(A_{\vec\bn} \mbox{ occurs }|\bX=\vec\bn)\to1 \mbox{ uniformly in }\vec\bn.
        \end{equation*}
        Thus we have
        \begin{align*}
            \nu\lt(A_{\mu,\chi,\rho,\varphi}\rt) 
            &\geq \sum_{\vec\bn \mbox{ is $(\chi,\rho,\varphi)$-almost-$\mu$-like}}
            \Big[\nu_{\Unif}\big(A_{\mu,\chi,\rho,\varphi}|\bX=\vec\bn\big) \nu_{\Unif}(\bX=\vec\bn)\Big]
            \\
            &\geq (1-o(1))\cdot\nu_{\Unif}\lt(\bX \mbox{ is $(\chi,\rho,\varphi)$-almost-$\mu$-like}\rt) =1-o(1).
        \end{align*}
        Here in the last line we used Lemma \ref{lem:num-q-like}.
    \end{proof}

    The event in \eqref{eq:hypothesis-test-criterion} defines a hypothesis test to detect a $\bP(\vbn,K)$-shuffled deck.
    Next we show that for a uniformly random deck, one can run many such tests for different $\vbn$, and with high probability none of these tests will detect non-uniformity.
    We now view $H(\cdot)$ as a mapping from the set of all $(\chi,\rho,\varphi)$-almost-$\mu$-like pile sizes to subsets of $[N]$.

    \begin{lem}\label{hypo-test2}
        With $\nu_{\Unif}$ the uniform measure on $\mathsf S_N$ and $A_{\mu,\chi,\rho,\varphi}$ the event defined in Lemma \ref{hypo-test1}, we have
        \begin{equation}
            \nu_{\Unif}(A_{\mu,\chi,\rho,\varphi})\to 0
        \end{equation}
        for any mappings $H(\cdot )$ such that $|H(\vec\bn)|\geq N^{1/2}$ and $|\partial H(\vec\bn)|=O(|H|^{1/2})$ for all $(\chi,\rho,\varphi)$-almost-$\mu$-like pile sizes $\vec\bn$.
    \end{lem}
    \begin{proof}
        Notice that for $j>i+1$, we have that the events $\{\sigma(i)<\sigma(i+1)\}$ and $\{\sigma(j)<\sigma(j+1)\}$ are independent under $\nu_{\Unif}$, and both hold with probability $1/2$. Therefore, for fixed $\bn$, if we define 
        \[E^{odd}=\#\{i \mbox{ is odd}: (i,i+1)\in |E(G)\cap E(H(\vec\bn))|\},\]
        \[E^{even}=\#\{i \mbox{ is even}: (i,i+1)\in |E(G)\cap E(H(\vec\bn))|\},\]
        then $E^{odd}$ and $E^{even}$ are binomial.
        Thus by Hoeffding's inequality, we may obtain the concentration for $|E(G)\cap E(H(\vec\bn))|$:
        \begin{equation}
            \nu_{\Unif}(|E(G)\cap E(H(\vec\bn))|\geq |H(\vec\bn)|/2+|H(\vec\bn)|^{1/2+\delta/2})=O(\exp(-|H(\vec\bn)|^{\delta}))=O(\exp(-N^{\delta/2})).
        \end{equation}
        Here we used the fact that $|\partial H(\vec\bn)|=O(|H(\vec\bn)|)^{1/2}$ and $|H(\vec\bn)|\geq N^{1/2}$.
        Since the total choice of $\vec\bn$ is less than $N^{Kk}=\exp(O(\log N)^2))$, we then conclude by a union bound over $\vec\bn$.
    \end{proof}

    Combining the above ingredients, we now deduce the desired mixing time lower bound for randomized $\mu$-like shuffles.

    \begin{proof}[Proof of \eqref{eq:LowerBound} and \eqref{eq:random-LowerBound}]
        Recall the definition of $A_{\vec\bn}$ above Lemma \ref{hypo-test1}. For \eqref{eq:LowerBound}, let $\nu \sim \bP(\vec\bn,K)$ and $\nu_{\Unif}$ be the uniform measure on $\mathsf S_N$. We have that $\nu(A_{\vec\bn})=1-o(1)$ by Lemma \ref{lem:new-lowerboundcond}, and that $\nu_{\Unif}(A_{\vec\bn})=o(1)$ by Lemma \ref{hypo-test2}. Thus
        \[d_{\TV}(\nu_{\Unif},\nu)\geq \nu(A_{\vec\bn})-\nu_{\Unif}(A_{\vec\bn})=1-o(1).\]
        For \eqref{eq:random-LowerBound}, let $\bX$ be a randomized-$\mu$-like shuffle process, $\nu \sim \bP(\bX,K)$ and $\nu_{\Unif}$ be the uniform measure on $\mathsf S_N$. Recall the definition of $A_{\mu,\chi,\rho,\varphi}$ in Lemma \ref{hypo-test1}.
        We conclude by Lemmas \ref{hypo-test1} and \ref{hypo-test2} that
        \[
        d_{\TV}(\nu_{\Unif},\nu)\geq \nu(A_{\mu,\chi,\rho,\varphi})-\nu_{\Unif}(A_{\mu,\chi,\rho,\varphi})=1-o(1).
        \qedhere
        \]
    \end{proof}

    \subsection{Proof of Lemma \ref{lem:constructcoldspot}: construction of the cold spots}\label{subsec:constructcoldspot}
    
    By Lemmas \ref{hypo-test1} and \ref{hypo-test2}, the proof of Proposition \ref{prop:new-increasing-LB} is now reduced to finding a deterministic set $H=H(\vec \bn)$ that satisfies the conditions of Lemma~\ref{lem:new-lowerboundcond}. In order to define such a deterministic set $H$, we first define a set of ``collision-likely'' prefixes, and we will choose $H$ to be the ``expected positions'' of strings with these prefixes. We will extend the definition of ``cold spots'' in \cite{mark2022cutoff}, since we need to tackle shuffle processes with $\chi$-bad shuffles.
    Let 
    \begin{equation}
        \label{def:alpha-tot}
        \alpha_{\tot}\log N=\left\lfloor \frac{1-\delta}{2\sum_{0\leq i\leq z}h_iI(\bp^{(i)},(\bp^{(i)})^{\theta_{\mu}})}\log N\right\rfloor\mbox{ and }\alpha_{\tot}^{(i)}\log N=|\{t\leq \alpha_{\tot}\log N:\bn^{(t)}/N\in D_i\}|.
    \end{equation}
    Recall the definition of $T$ in \eqref{def:set-t} and definition of $\mathfrak p^{(i)}$ in \eqref{def-frakp}. For all $0\leq i\leq z$, we choose positive integers $\alpha_l^{(i)}\log N$ for $l$ such that $(i,l)\in T$ satisfying
	\begin{equation}
		\label{eq:alpha-def}
		\sum_{l:(i,l)\in T}\alpha_l^{(i)}=\alpha_{\tot}^{(i)} \text{ and }\left|\alpha_l^{(i)}\log(N)-\mathfrak p^{(i)}_l\alpha_{\tot}^{(i)}\log(N) \right|\leq 1.
	\end{equation}
	We claim that $\alpha_{\tot} \leq \frac{2}{3}C_{\mu}$. Indeed, by \eqref{eq:entropy-psi} we have
    \[
    \alpha_{\tot}+O(\delta)
    =\frac{1}{2\sum_ih_iI(\bp^{(i)},(\bp^{(i)})^{\theta_{\mu}})}
    <\frac{\theta_\mu}{2\sum_{i}h_i\psi_{\bp^{(i)}}(\theta_\mu)}
    \leq \frac{2\overline C_\bp}{3},
    \]
    where we used the fact that $\chi \ll 1$ and $\psi_{\mu}(\theta_\mu)=\sum_{i}h_i\psi_{\bp^{(i)}}(\theta_\mu)+o_{\chi}(1)$.
    Therefore, we may let 
    \[
    \beta_{\tot}\log(N)=K-\lfloor2\rho\log N\rfloor-\alpha_{\tot}\log(N)\geq\Omega(\log N),
    \]
    \[
    \beta_{\tot}^{(i)}\log N=|\{\alpha_{\tot}\log N+\lfloor2\rho\log N\rfloor<t\leq K: \bn^{(t)}/N\in D_i\}|.
    \]
    Similarly to $\mathfrak p^{(i)}$, we define
    \begin{equation}\label{def-frakq}
        \mathfrak q^{(i)}=\lt(\frac{(p_l^{(i)})^2}{\sum_{l:(i,l) \in T}(p_l^{(i)})^2}\rt)_{l:(i,l)\in T}.
    \end{equation}
    We choose positive integers $\beta_l^{(i)}\log(N) $ for all $(i,l)\in T$ with
    \begin{equation}\label{def:betali}
        \sum_{l=0}^{k-1}\sum_{0\leq i\leq z} \beta_l^{(i)}=\beta_{\tot}\text{ and }\left|\beta_l^{(i)}\log(N)- \mathfrak q^{(i)}_l \beta_{\tot}^{(i)}\log(N) \right|\leq 1.
    \end{equation}
	Note that $\alpha_{\tot}\log N+\beta_{\tot} \log N=K-\lfloor2\rho\log N\rfloor$. Since the shuffle process is $(\chi,\rho,\varphi)$-almost-$\mu$-like, we have 
    \[
    \alpha_{\tot}^{(i)}=h_i\alpha_{\tot}\pm O(\rho)\pm O(\varphi),\quad
    \beta_{\tot}^{(i)}=h_i\beta_{\tot}\pm O(\rho)\pm O(\varphi).
    \]
	We now consider $G$-edges coming from strings with $\alpha_l^{(i)} \log N$ digits $l$ in $\mathsf P_i$ in the first $\alpha_{\tot} \log N$ digits, and $\beta_l^{(i)} \log N$ digits $l$ digits in $\mathsf P_i$ in the last $\beta_{\tot} \log N$ digits for $i=0,1,\ldots,k-1$. 
    A straightforward calculation shows that
    \begin{equation}\label{eq:alpha-bound}
    \sum_{(i,l)\in T} \alpha_l^{(i)} \log(p_l^{(i)})\geq 
    \frac{(1-\delta)}{2\sum_{0\leq i\leq z}h_i I(\bp^{(i)},(\bp^{(i)})^{\theta_{\mu}})}\cdot \sum_{(i,l)\in T} \frac{h_i(p^{(i)}_l)^{\theta_{\mu}}\log(p_l^{(i)})}{\phi_{\bp^{(i)}}(\theta_{\mu})}
    \geq
    -\frac{1}{2}+\frac{\delta}{4}.
\end{equation}
    Here in the second inequality we used the fact that $\chi\ll \delta$ and
    \begin{equation}
        \sum_{l:(i,l)\in T}(p^{(i)}_l)^{\theta_{\mu}}\log(p_l^{(i)})=I(\bp^{(i)},(\bp^{(i)})^{\theta_{\mu}})\pm O(\chi).
    \end{equation}

	\begin{defi}		
		The length $\alpha_{\tot}\log(N)$ string $x\in [k]_0^M$ is a \textbf{collision-likely prefix} (we write $x\in \Pre$) if for all $(i,l)\in T$,
        \begin{equation}
            x \mbox{ contains } \alpha_l^{(i)}\log(N) \mbox{ of } l\mbox{-digits in positions } \{1\leq t\leq \alpha_{\tot}\log(N): \bn^{(t)}/N\in D_i\}.
        \end{equation}		
	\end{defi}

    Since $\vec\bn$ is $(\chi,\rho,\varphi)$-almost-$\mu$-like and $\alpha_{\tot}\log N<K$, we can always find $g_1<g_2$ in $[\alpha_{\tot}\log N+1,\alpha_{\tot}\log N+2\rho\log N]$ such that $\bn^{(g_1)},\bn^{(g_2)}$ are $\chi$-good. (Recall the definition of $\chi$-good above \eqref{def:almost}.)
	
	\begin{defi}
		\label{def:cl}
		The string $s\in [k]_0^{K}$ is \textbf{collision-likely} (we write $s\in \CL$) if $s$ satisfies the following properties.
		\begin{itemize}
			\item The first $\alpha_{\tot}\log(N)$ digits of $s$ form a collision-likely prefix.
			\item $s[g_1]=\ileft^{(g_1)}$, $s[g_2]=\iright^{(g_2)}$. (Recall \eqref{def:ileft-right}.)
            \item $s[t]=l^{(t)}_{\Max}$ for all $t\in [\alpha_{\tot}+1,\cdots,\alpha_{\tot}+\lfloor2\rho\log N\rfloor]\backslash\{g_1,g_2\}$.
			\item For all $(i,l)\in T$, we have \[
            s \mbox{ contains } \beta_l^{(i)}\log(N) \mbox{ of } l\mbox{-digits in } \{\alpha_{\tot}+\lfloor2\rho\log N\rfloor+1\leq t\leq K: \bn^{(t)}/N\in D_i\}.
            \]
		\end{itemize}
	\end{defi}
    
    We now give the definition of $H(\vec\bn)$. Recalling \eqref{eq:Jx}, set
    \begin{equation}\label{def:coldspot}
            H(\vec\bn)\equiv \mathbb Z\cap \Big(\bigcup_{x\in \Pre} NJ_x\Big).
    \end{equation}
    The set $H(\vbn)$ now contains all expected locations of ``collision-likely" prefixes for pile sizes $\vbn$.
	\begin{lem}
		As $N\to\infty$, we have $|H(\vec\bn)|\geq N^{1/2}$ and $|\partial H(\vec\bn)|\leq 2|H(\vec\bn)|^{\frac{1}{2}}$. More precisely,
		\begin{equation}\label{eq:cold-spot-size}
		    |H(\vec\bn)|=N^{1+\sum_{(i,l)\in T} \alpha_l^{(i)}(\log(p_l^{(i)})\pm O(\chi))+\alpha_{\tot}\sum_{i=0}^{z}h_iH(\{\alpha_l^{(i)}\}_{l:(i,l)\in T})-o(\delta)}\geq N^{1/2}.
		\end{equation}
	\end{lem}

    \begin{proof}
        For each $x \in \Pre$, by \eqref{eq:alpha-bound}, we have
        \begin{equation}\label{eq:one-inter-exp-length}
            \lambda_x = N^{\sum_{(i,l)\in T} \alpha_l^{(i)}(\log(p_l^{(i)})\pm O(\chi))} = N^{\frac{-1+\delta}{2}\pm o(\delta)}, 
        \end{equation}
        where we have used $\sum_{(i,l)\in T} \alpha_l^{(i)} \leq \frac{K}{\log N}$ is bounded and $\chi\ll \delta$.
        Moreover, it is easy to see 
        \begin{equation} \label{eq:one-interval-length}
            \lfloor N \lambda_x \rfloor \leq \Z \cap NJ_x \leq \lceil N\lambda_x\rceil.
        \end{equation}
        This implies that $|\mathbb Z\cap NJ_x|\geq N^{1/2+\delta/2}$ and thus $|H| \to \infty$ since $\Pre \neq \emptyset$. Moreover, since these intervals are mutually disjoint, we have $|\Pre|\leq N^{1/2-\delta/2}$, this means the number of connected components of $H$ is smaller than the size of each component, hence $|\partial H|\leq 2|H|^{1/2}$.
        In addition, by Lemma~\ref{lem:entropy} and the definition of $\alpha^{(i)}_{l}$, we have
        \begin{equation}
        \begin{aligned}
            |\Pre|&=\prod_{0 \leq i \leq z} \binom{\alpha^{(i)}_{\tot}\log(N)}{\{\alpha_l^{(i)}\log N\}_{l:(i,l)\in T}}\\
            &=N^{\sum_{i=0}^{z}\alpha^{(i)}_{\tot}H(\{\alpha_l^{(i)}\}_{l:(i,l)\in T})+o(1)} \\
            &= N^{\alpha_{\tot}\sum_{i=0}^{z}h_iH(\{\alpha_l^{(i)}\}_{l:(i,l)\in T})\pm O(\chi)}.
        \end{aligned}
        \end{equation}
        Here we used $\varphi,\rho\ll\chi\ll \delta$. Again, note that these intervals are mutually disjoint.
        Combining this with \eqref{eq:one-inter-exp-length}, \eqref{eq:one-interval-length} now gives the first equation of \eqref{eq:cold-spot-size}. 
    \end{proof}

	\begin{lem}
		\label{lem:KS}
		With probability $1-o(1)$, every $s_i\in \CL$ obeys $i\in H$.
	\end{lem}

    \begin{proof}
        Consider the event
        \begin{equation}
            \Lambda=\lt\{\lt||\cI(B_x)|-\mathbb E|\mathcal I(B_x)|\rt|\leq N^{\frac{1}{2}+\frac{\delta}{20}}, \mbox{ for all }x\in [k]_0^M\rt\}.
        \end{equation}
        By Lemmas \ref{concen-prefix} and \ref{lem:concen-block-smallmean} and a union bound, we see $\mathbb P(\Lambda)=1-o(1)$.
        Note that 
        \begin{equation}
            \iota(y)=\sum_{j=1}^{K}\sum_{0\leq l<y[j]}\lt|\mathcal I(B_{x[1:j-1]l})\rt| \mbox{ and } \tau(y) =\sum_{j=1}^{K}\sum_{y[j]<l\leq k-1}\lt|\mathcal I(B_{x[1:j-1]l})\rt|.
        \end{equation}
        On the event $\Lambda$ we have that
        \begin{equation}
            \left|\iota(y)-Nt_y\right|\leq N^{\frac{1}{2}+\frac{\delta}{10}}\mbox{ and }\left|\tau(y)-N(t_y+\lambda_y)\right|\leq N^{\frac{1}{2}+\frac{\delta}{10}} \mbox{ for all }y \in [k]_0^M.
        \end{equation}
        Now we prove that on the event $\Lambda$, $i\in H$ for all $i\in [N]$ such that $s_i\in \CL$. 
        For such $i$, we see that $s_i$ starts with a prefix $x\in \Pre$. Recall the definition of $g_1,g_2$ above Definition~\ref{def:cl}. Let $u_1,u_2\in [k]_0^{\alpha_{\tot}\log N+\lfloor2\log N\rfloor}$ such that
        \begin{align*}
            &u_j[t]=s_i[t] \mbox{ for all } t \in [\alpha_{\tot}\log N+\lfloor2\log N\rfloor]\backslash\{g_1,g_2\},\\ &u_1[g_1]=\ileft^{(g_1)},u_2[g_1]=\iright^{(g_1)}, u_1[g_2]=\ileft^{(g_2)}, u_2[g_2]=\iright^{(g_2)}.
        \end{align*}
        By Definition~\ref{def:cl}, we have $t_{s_i}\geq t_x+\lambda_{u_1}$ and $t_{s_i}+\lambda_{s_i}\leq t_x+\lambda_x-\lambda_{u_2}$. Therefore, we have that
        \begin{equation}\label{eq:interval-left}
            i\geq \iota(s_i)\geq Nt_{s_i}-N^{\frac{1}{2}+\frac{\delta}{10}}\geq N(t_x+\lambda_{u_1})-N^{\frac{1}{2}+\frac{\delta}{10}}\geq Nt_x.
        \end{equation}
        Here we used that by \eqref{def:ileft-right} and \eqref{eq:alpha-bound},
        \begin{equation}
            N\lambda_{u_1}= N\lambda_x \cdot \prod^{\alpha_{\tot}\log N+\lfloor2\rho\log N\rfloor}_{t=\alpha_{\tot}\log N} \frac{n_{u_1[t]}^{(t)}}{N}\geq N\lambda_x\cdot k^{-2\rho\log N}\cdot \chi^2\geq N^{\frac{1}{2}+\frac{\delta}{10}}.
        \end{equation}
        Similarly to \eqref{eq:interval-left}, we have $i\leq N(t_x+\lambda_x)$, which indicates that $i\in H$.
    \end{proof}
	
	\noindent
    Define
	\begin{equation}
		\gamma\equiv 2+2\sum_{(i,l)\in T} (\alpha_l^{(i)}+\beta_l^{(i)}) \log(p_l^{(i)})+\alpha_{\tot}\sum_{i=0}^{z}h_iH\lt(\{\alpha_l^{(i)}\}_{l:(i,l)\in T}\rt)+\beta_{\tot}\sum_{i=0}^{z} h_iH\lt(\{\beta_l^{(i)}\}_{l:(i,l)\in T}\rt).
	\end{equation}
	
	\begin{lem}
        We have
	    \begin{equation} 
		\label{eq:gamma}
        \gamma \geq \frac{1}{2}\lt(1+\sum_{(i,l)\in T} \alpha_l^{(i)} \log(p_l^{(i)}) + \alpha_{\tot}\sum_{i=0}^{z}h_iH\lt(\{\alpha_l^{(i)}\}_{l:(i,l)\in T}\rt) \rt) + \Omega(\eps).
	\end{equation}
	\end{lem}

    \begin{proof}
        It suffices to lower bound
        \begin{align}
            3+3\sum_{(i,l)\in T} \alpha_l^{(i)} \log(p_l^{(i)})&+4\sum_{(i,l)\in T} \beta_l^{(i)} \log(p_l^{(i)})\\
            &+\alpha_{\tot}\sum_{i=0}^{z}h_iH\lt(\{\alpha_l^{(i)}\}_{l:(i,l)\in T}\rt)+2\beta_{\tot}\sum_{i=0}^{z} h_iH\lt(\{\beta_l^{(i)}\}_{l:(i,l)\in T}\rt).\label{eq:gamma-bound}
        \end{align}
        Recall the definition of $\beta_l^{(i)}$'s in \eqref{def:betali}. We write $\bar \bp^{(i)}=(p_l^{(i)}/\sum_{l':(i,l')\in T}p^{(i)}_{l'})_{l:(i,l)\in T}$. We have
        \begin{align}
            2\beta_{\tot}\sum_{i=0}^{z} h_iH\lt(\{\beta_l^{(i)}\}_{l:(i,l)\in T}\rt)
            &=2\beta_{\tot}\sum_{i=0}^{z} h_i\lt(2I(\bar\bp^{(i)},\mathfrak q^{(i)})-\psi_{\bar\bp^{(i)}}(2)\rt)\\
            &=-4\beta_{\tot}\sum_{i=0}^{z} h_i\sum_{l=0}^{k-1}\frac{(p_l^{(i)})^2\log (p_l^{(i)})}{\phi_{\bp^{(i)}}(2)}
            -2\beta_{\tot}\sum_{i=0}^{z} h_i\psi_{\bp^{(i)}}(2)\pm O(\chi)\\
            &=-4\sum_{(i,l)\in T} \beta_l^{(i)} \log(p_l^{(i)})-2\beta_{\tot}\sum_{i=0}^{z} h_i\psi_{\bp^{(i)}}(2)\pm O(\chi)\\
            &=-4\sum_{(i,l)\in T} \beta_l^{(i)} \log(p_l^{(i)})-2\beta_{\tot}\psi_{\mu}(2)\pm o_{\chi}(1).\label{eq:beta-bound}
        \end{align}
        In addition, recall the definition of $\alpha_{\tot}$ in \eqref{def:alpha-tot}, we have
        \begin{align}
            \alpha_{\tot}\sum_{i=0}^{z} h_iH\lt(\{\alpha_l^{(i)}\}_{l:(i,l)\in T}\rt)
            &=\alpha_{\tot}\sum_{i=0}^{z} h_i\lt(\theta_{\mu}I(\bar\bp^{(i)},\mathfrak p^{(i)})-\psi_{\bp^{(i)}}(\theta_{\mu})\rt)\pm O(\chi)\\
            &=\frac{\theta_{\mu}}{2}-\alpha_{\tot}\sum_{i=0}^{z} h_i\psi_{\bp^{(i)}}(\theta_{\mu})\pm O(\chi)\\
            &=\frac{\theta_{\mu}}{2}-\alpha_{\tot}\psi_{\mu}(\theta_\mu)+o_{\chi}(1)\pm O(\chi)\\
            &=\frac{\theta_{\mu}}{2}-2\alpha_{\tot}\psi_{\mu}(2)\pm o_{\chi}(1).
            \label{eq:alpha-lower-bound}
        \end{align}
        Combining \eqref{eq:alpha-bound}, \eqref{eq:beta-bound} and \eqref{eq:alpha-lower-bound}, we see that \eqref{eq:gamma-bound} is lower-bounded by
        \begin{align*}
            3+\frac{-3+3\delta}{2}+\frac{\theta_{\mu}}{2}-2\alpha_{\tot}\psi_{\mu}(2)-2\beta_{\tot}\psi_{\mu}(2)+O(\delta)=\frac{3+\theta_{\mu}}{2}- \frac{2K}{\log N}\psi_{\mu}(2)+O(\delta)\geq \Omega(\eps).
        \end{align*}
        Here we used the fact that $\chi\ll\delta\ll \varepsilon$ and $K=\lfloor(C_\mu-\eps)\log N\rfloor$.
    \end{proof}

	\begin{lem}
	\label{lem:coldspot-edges}
	For some $\delta=\delta(\mu,\eps)$ the following holds: with probability $1-o(1)$,
		\begin{equation}\label{coldspot-edges-eq}
			|E(G(S_K))\cap E(H)|\geq |H|^{1/2+\delta}.
		\end{equation}
	\end{lem}
	\begin{proof}
		For $s\in \CL$, let $A_s$ be the event that $s$ occurs at least two times in $S_K$. By inclusion-exclusion, 
		\begin{align}
			\mathbb P(A_s)
			&=\mathbb P\big(\bigcup_{i<j}\{\overline s_i=\overline s_j=s\}\big)\nonumber
            \\
			&\geq \sum_{i<j}\mathbb P(\overline s_i=\overline s_j=s)-\sum_{i<j,k<l,(i,j)\neq (k,l)}\mathbb P(\overline s_i=\overline s_j=\overline s_k=\overline s_l=s)\nonumber\\
			&\geq \sum_{i<j}\mathbb P(\overline s_i=\overline s_j=s)-3\sum_{i<j<k}\mathbb P(\overline s_i=\overline s_j=\overline s_k=s)-3\sum_{i<j<k<l}\mathbb P(\overline s_i=\overline s_j=\overline s_k=\overline s_l=s)\nonumber\\
			&=\binom{N}{2}\mathbb P(\overline s_i=\overline s_j=s)-3\binom{N}{3}\mathbb P(\overline s_i=\overline s_j=\overline s_k=s)-3\binom{N}{4}\mathbb P(\overline s_i=\overline s_j=\overline s_k=\overline s_l=s)\label{inclusion-exclusion}.
		\end{align}    
		Direct calculation shows that for $1\leq i<j\leq N$,
		\begin{equation}\label{coldspot-edges-eq1}
			\mathbb P(\overline s_i=\overline s_j=s)=\prod_{t\leq K}\mathbb P(\overline s_i[t]=\overline s_j[t]=s[t])=\prod_{t\leq K}\frac{n_{s[t]}^{(t)}(n_{s[t]}^{(t)}-1)}{N(N-1)}.
		\end{equation}
		Notice that for any $t \in \mathsf P_i$, 
        \[
        \frac{n_{s[t]}^{(t)}(n_{s[t]}^{(t)}-1)}{N(N-1)}=\Big(\frac{n_{s[t]}^{(t)}}{N}+O(1/N)\Big)^2=(p_{s[t]}^{(i)}+O(\chi))^2,
        \] 
        so we can immediately deduce from \eqref{coldspot-edges-eq1} that 
		\begin{equation}\label{coldspot-edges-eq2}
		    \begin{aligned}
		        \mathbb P(\overline s_i=\overline s_j=s)&=\prod_{(i,l)\in T}(p_l^{(i)}+O(\chi))^{2(\alpha_l^{(i)}+\beta_l^{(i)})\log N}
                =N^{2\sum_{(i,l)\in T}(\alpha_l^{(i)}+\beta_l^{(i)})\log p_l^{(i)}\pm O(\chi)}.
		    \end{aligned}
		\end{equation}
		Similar calculation shows that for $i<j<k<l$,
		\begin{equation}\label{coldspot-edges-eq3}
            \begin{aligned}
                \mathbb P(\overline s_i=\overline s_j=\overline s_k=s)\leq N^{3\sum_{(i,l)\in T}(\alpha_l^{(i)}+\beta_l^{(i)})\log p_l^{(i)}\pm O(\chi)},
            \end{aligned}	
		\end{equation}
		\begin{equation}\label{coldspot-edges-eq4}
			\mathbb P(\overline s_i=\overline s_j=\overline s_k=\overline s_l=s)\leq N^{4\sum_{(i,l)\in T}(\alpha_l^{(i)}+\beta_l^{(i)})\log p_l^{(i)}\pm O(\chi)}.
		\end{equation}
		Combining \eqref{coldspot-edges-eq2}, \eqref{coldspot-edges-eq3} and \eqref{coldspot-edges-eq4}, we have
		\begin{equation}\label{coldspot-edges-eq5}
			\binom{N}{3}\mathbb P(\overline s_i=\overline s_j=\overline s_k=s)\leq N^{3+3\sum_{(i,l)\in T}(\alpha_l^{(i)}+\beta_l^{(i)})\log p_l^{(i)}+O(\chi)}=O(N^{-\delta/10})\binom{N}{2}\mathbb P(\overline s_i=\overline s_j=s),
		\end{equation}
		\begin{equation}\label{coldspot-edges-eq6}
			\binom{N}{4}\mathbb P(\overline s_i=\overline s_j=\overline s_k=\overline s_l=s)\leq N^{4+4\sum_{(i,l)\in T}(\alpha_l^{(i)}+\beta_l^{(i)})\log p_l^{(i)}+O(\chi)}=O(N^{-\delta/10})\binom{N}{2}\mathbb P(\overline s_i=\overline s_j=s),
		\end{equation}
		where we used the fact that 
		\begin{equation*}
			1+\sum_{(i,l)\in T}(\alpha_l^{(i)}+\beta_l^{(i)})\log p_l^{(i)}\leq 1+\frac{K}{\log N}\cdot \sum_{i}h_i\log p_{\term{max}}^{(i)}\leq -\delta
		\end{equation*}
		since $\chi \ll \delta$.
		Thus, combining \eqref{inclusion-exclusion}, \eqref{coldspot-edges-eq5} and \eqref{coldspot-edges-eq6}, we have
		\begin{equation}\label{prob-CL-occur}
		\mathbb P(A_s)
            =
            \binom{N}{2}
            \prod_{t\leq K}
            \Big(\frac{n_{s[t]}^{(t)}}{N}+O(1/N)\Big)^2(1-O(N^{-\delta/10}))
            =
            N^{2+2\sum_{(i,l)\in T}(\alpha_l^{(i)}+\beta_l^{(i)})\log p_l^{(i)}\pm O(\chi)}.
		\end{equation}
		With $Z=\sum_{s\in \CL}1_{A_s}$, we have $|E(G(S_K))\cap E(H)|\geq Z$. The above calculation and Lemma \ref{lem:entropy} yield
		\begin{equation}\label{expected-CL-occur}
		\mathbb EZ
            \geq
            N^{\gamma-O(\chi)}.
		\end{equation}
		To calculate the second moment, first notice that for $s\neq s^{\prime}\in \CL$, we have
		\begin{align*}
			\mathbb P(A_s\cap A_{s^\prime})&\leq \sum_{i<j,\&\,k<l}
            \mathbb P(\overline s_i=\overline s_j=s, \overline s_k=\overline s_l=s^\prime)\\
			&=\binom{N}{4}\binom{4}{2}
            \prod_{t\leq K}\mathbb P(\overline s_i[t]=\overline s_j[t]=s[t], \overline s_k[t]=\overline s_l[t]=s^\prime[t]),\\
			&=\binom{N}{4}\binom{4}{2}
                \prod_{t\leq K}\Big(\frac{n_{s[t]}^{(t)}}{N}+O(1/N)\Big)^2\Big(\frac{n_{s^{\prime}[t]}^{(t)}}{N}+O(1/N)\Big)^2,
		\end{align*}
		where we used the fact that
		\begin{equation*}
			\mathbb P(\overline s_i[t]=\overline s_j[t]=s[t], \overline s_k[t]=\overline s_l[t]=s^\prime[t])=
			\begin{cases}
				\frac{n_{s[t]}^{(t)}(n_{s[t]}^{(t)}-1)(n_{s[t]}^{(t)}-2)(n_{s[t]}^{(t)}-3)}{N(N-1)(N-2)(N-3)}, &\mbox{ if }s[t]=s^\prime[t],
                \\
                ~
                \\
			\frac{n_{s[t]}^{(t)}
                (n_{s[t]}^{(t)}-1)n_{s^\prime[t]}^{(t)}(n_{s^\prime[t]}^{(t)}-1)}{N(N-1)(N-2)(N-3)}, &\mbox{ if }s[t]\neq s^\prime[t].
			\end{cases}
		\end{equation*}
		Then, by the first equality in \eqref{prob-CL-occur} we have that
		\begin{align*}
			\mathbb P(A_s)\mathbb P(A_{s^\prime})&= (1-O(N^{-\delta/10}))\binom{N}{2}^2
            \prod_{t\leq K}
            \bigg[
            \Big(\frac{n_{s[t]}^{(t)}}{N}+O(1/N)\Big)^2
            \Big(\frac{n_{s^{\prime}[t]}^{(t)}}{N}+O(1/N)\Big)^2
            \bigg].
		\end{align*}
		Thus we obtain that
		\begin{equation*}
		  \mathbb P(A_s\cap A_{s^\prime})-\mathbb P(A_s)\mathbb P(A_{s^\prime})
          \leq 
          O(N^{-\delta/10})\cdot N^4 
          \prod_{t\leq K}
          \bigg[
          \Big(\frac{n_{s[t]}^{(t)}}{N}\Big)^2
          \Big(\frac{n_{s^{\prime}[t]}^{(t)}}{N}\Big)^2
          \bigg]
          \leq 
          O(N^{-\delta/10})
          \mathbb P(A_s)\mathbb P(A_{s^\prime}).
		\end{equation*}
		Thus we have
		\begin{equation}\label{var-CL-occur}
			\Var(Z)=\sum_{s,s^\prime\in \CL} [\mathbb P(A_s\cap A_{s^\prime})-\mathbb P(A_s)\mathbb P(A_{s^\prime})]\leq \sum_{s,s^\prime\in \CL}O(N^{-\delta/10})\mathbb P(A_s)\mathbb P(A_{s^\prime})= O(N^{-\delta/10})(\mathbb EZ)^2.
		\end{equation}
		Combining \eqref{expected-CL-occur}, \eqref{var-CL-occur} and $\chi,\eta \ll \delta \ll \varepsilon$, we obtain by the Markov inequality that
		\begin{equation}
			\mathbb P(|E(G(S_K))\cap E(H)|\geq N^{\gamma-\delta})\geq 
            \mathbb P\big(Z\geq \frac{1}{2}\mathbb EZ\big)
            \geq 
            1-\frac{4\Var(Z)}{(\mathbb EZ)^2}=1-o(1).
		\end{equation}
		Then by \eqref{eq:gamma}, we know that $N^{\gamma-\delta}\geq|H|^{1/2+\delta}$ (by taking $\delta$ sufficiently small depending only on $\eps$), and thus we complete the proof.
	\end{proof}

\appendix

\section{Non-Convexity of $\mu\mapsto\oC_{\mu}$}
\label{app:non-convex-example}

Recall the definition of $\psi_{\mu}$ in and below \eqref{eq:phi-psi}. It is easy to see that $\psi_{\mu}(1)=0$, and that $\psi_{\mu}(x)$ is concave and $\psi_{\mu}(x)/x$ is increasing (by the H{\"o}lder and power mean inequalities respectively). 
In fact these essentially characterize all possible functions $\psi_{\mu}$ up to rescaling.
Moreover the smaller class $\cup_{k\geq 1}\{\psi_{\bp}\}_{\bp\in\cD_k}$ suffices to approximate all such functions (this follows from Proposition~\ref{prop:characterize-psi} below and convexity of $\Psi$).

Let $\Psi$ be the set of functions $f:[1,4]\to \bbR_+$ which are concave and satisfy $f(1)=0$ and $f(2)>0$ and have $f(x)/x$ increasing. 
We take $\theta_f$ to satisfy $f(\theta_f)=2f(2)$; it follows easily from the definitions that $\theta_f\in [3,4]$ for all $f\in\Psi$.
We let 
\[
C_f
=
\frac{3+\theta_f}{4f(2)},
\quad\quad
\widetilde C_f
=
\frac{1}{f'(4)},
\quad\quad
\oC_f=\max(C_f,\widetilde C_f).
\]
(Here $f'(4)$ is always well-defined as a derivative from the left since $f$ is concave on domain $[1,4]$.)

\begin{prop}
\label{prop:characterize-psi}
Then for any $f\in\Psi$ and $\eps>0$ there exists $K_*(f,\eps)$ and a sequence of probability distributions $\bp^{(K)}\in\cD_{k(K)}$ for $K\geq K_*$ such that $\psi_{\bp^{(K)}}(x)\log K$ approximates $f$ in the sense that:
\begin{enumerate}[label = (\Roman*)]
    \item 
    \label{it:f-approx}
    $\Big|f(x)-\frac{\psi_{\bp}(x)}{\log K}\Big|\leq \eps$ for all $ x\in [2,4]$.
    \item 
    \label{it:theta-approx}
    $\Big|\theta_{f}-\frac{\theta_{\bp^{(K)}}}{\log K}\Big|\leq \eps$.
    \item 
    \label{it:C-approx}
    $\Big|C_{f}-\frac{C_{\bp^{(K)}}}{\log K}\Big|\leq \eps$.
    \item 
    \label{it:tilde-C-approx}
    $\Big|\widetilde C_{f}-\frac{\widetilde C_{\bp^{(K)}}}{\log K}\Big|\leq \eps$.
\end{enumerate}
(The last two points also then imply $\Big|\oC_{f}-\frac{\oC_{\bp^{(K)}}}{\log K}\Big|\leq \eps$.)
\end{prop}

\begin{proof}
Without loss of generality we can slightly perturb $f$ to be strictly concave (without changing any of the relevant quantities by more than $\eps/2$).
Such $f$ can be approximated as a minimum of affine functions
\[
f_{\delta}(x)=\min_{j\in [J]}[a_j x-b_j]
\]
with $|f(x)-f_{\delta}(x)|\leq\delta$ on $x\in [2,4]$.
Explicitly, we can take $a_j x-b_j=(x-x_j)f'(x_j)+f(x_j)$ to be a tangent line approximation to $f$ at $x_j$ for a large discrete set of points $2=x_1\leq x_2\leq \dots\leq x_J=4$ with $|x_{j+1}-x_j|\leq \delta'$ with $\delta'$ small depending on $f$ and $\delta$. (Any choice of supergradient can be used for $f'(x_j)$ if $x_j$ is not a point of differentiability).
The condition that $f(x)/x$ is increasing implies that $a_j,b_j\geq 0$ for each $j$.
Since we assumed $f$ is strictly concave we also have $a_j>b_j$ for each $j$, i.e. each tangent line is strictly above the point $(1,0)=(1,f(1))$.

We choose $\delta$ small depending on $(f,\eps)$, and then $K_*$ large depending on $(f,\eps,\delta,\{x_1,...,x_J\})$.
(Note that the notation of this appendix is unrelated to \eqref{eq:small-params}.) 
We take $\bp$ to include $\lceil K^{b_j}\rceil$ coordinates with value $K^{-a_j}$, for each $1\leq j\leq J$.
Since $a_j>b_j$ these coordinates have total mass at most $2\sum_{j=1}^J K^{b_j-a_j}<1$ for large $K_*$.
To make $\bp$ a probability distribution, the remaining mass of $\bp$ is placed on an extremely large number (say $K^{10\cdot \max_j(|a_j|+|b_j|)}$) of coordinates with equal mass, which contribute negligibly to the value of $\phi_{\bp}(x)$ for $x\in [1,4]$. 
It is easy to see that with this choice, one has 
\[
\phi_{\bp}(x)
\in 
[K^{-f_{\delta}(x)},(J+1)K^{-f_{\delta}(x)}],\quad\forall x\in [2,4].
\]
Since $\psi_{\bp}(x)=-\log \phi_{\bp}(x)$ we have 
\[
\psi_{\bp}(x)
\in 
[f_{\delta}(x)\log(K) -\log(J+1), f_{\delta}(x)\log(K)],\quad\forall x\in [2,4].
\]
This yields part~\ref{it:f-approx} as $K_*$ is large depending on $J$.
Part~\ref{it:theta-approx} is immediate since $f$ is strictly increasing, so $\theta_f$ is stable to small perturbations.
Part~\ref{it:C-approx} follows by combining the previous two (and adjusting $\eps$).
Finally part~\ref{it:tilde-C-approx} follows from concavity of $f$, which ensures that $f'(4)$ is the minimal value of $a_j=f'(x_j)$ that appears.
\end{proof}

By taking explicit choices $f,\breve f\in\Psi$, we now provide the counterexample claimed in Proposition~\ref{prop:non-convexity}.

\begin{proof}[Proof of Proposition~\ref{prop:non-convexity}]
Given Proposition~\ref{prop:characterize-psi}, it suffices to find $f,\breve f\in\Psi$ so that $\hat f(x)\equiv\frac{f(x)+\breve f(x)}{2}$ satisfies 
\[
\oC_{\hat f}>\max(\oC_f,\oC_{\breve f}).
\]
Indeed, then for $K$ large the approximations $(\bp^{(K)},\breve\bp^{(K)})$ as in Proposition~\ref{prop:characterize-psi} will furnish the desired counterexample.\footnote{As stated the number of piles in $(\bp^{(K)},\breve\bp^{(K)})$ as constructed above will not be equal to each other, which is technically required in Proposition~\ref{prop:non-convexity}. However this is easily fixed: recall that both $(\bp^{(K)},\breve\bp^{(K)})$ include a very large number of equal-mass tiny coordinates which contribute negligibly to $\psi$. By increasing the number of such tiny coordinates in either $\bp^{(K)}$ and $\breve\bp^{(K)}$ as needed, we can ensure they have an equal number of piles while retaining the conclusions of Proposition~\ref{prop:characterize-psi}.}
In fact we will have $\oC_f=\oC_{\breve f}$ below, and parametrize the construction by a suitably small absolute constant $\eta$.
We want $\theta_{\hat f}$ to be large given the values $\theta_f<\theta_{\breve f}$ (e.g. larger than their average), and the idea below is simply to make $f$ grow fast until $\theta_f$ and then grow slowly afterward.

We take $f$ to be linear on each of the intervals $[1,2]$ and $[2,3.5-\eta]$ and $[3.5-\eta,4]$, with values:
\[
f(1)=0,\quad 
f(2)=6.5-\eta,
\quad 
f\big(3.5-\eta\big)
=
13-2\eta
\]
and satisfy $f'(x)=4$ on $x\in (3.5-\eta,4]$.
Then $f$ is concave because the derivatives on each subinterval are $6.5-\eta>(6.5-\eta)\cdot (2/3) > 4$, and also $f(x)/x$ is increasing (here it suffices to check the endpoints).

Meanwhile we take $\breve f$ to be linear on $[1,2]$ and $[2,4]$ with values:
\[
\breve f(1)=1,\quad 
\breve f(2)=6.5+\eta,
\quad 
\breve f(3.5+\eta)=13+2\eta.
\]
One can similarly check that $\breve f\in \Psi$.
It is immediate then that 
\begin{align*}
3+\theta_f&=6.5-\eta=f(2)
\quad\quad 
\implies\quad\quad 
C_f=1/4;
\\
3+\theta_{\breve f}&=6.5+\eta=\breve f(2)
\quad\quad 
\implies 
\quad\quad 
C_{\breve f}=1/4.
\end{align*}
Meanwhile $f'(4)=4$ so $\widetilde C_f=1/4$, and $\breve f'(4)=\frac{6.5+\eta}{1.5+\eta}>4$ so $\widetilde C_{\breve f}<1/4$.
Thus $\oC_f=\oC_{\breve f}=1/4$.

Finally for $\hat f$, we have $\hat f(2)=6.5$. So to show $C_{\hat f}>1/4$ it remains to verify that $\theta_{\hat f}>3.5$, or equivalently that $\hat f(3.5)<2\hat f(2)=13$.
Using piece-wise linearity it is easy to calculate that 
\begin{align*}
f(3.5)&=(13-2\eta)+4\eta=13+2\eta,\quad \quad
\breve f(3.5)=13-\frac{7\eta}{3}\pm O(\eta^2)
\\
\implies 
\hat f(3.5)&=13-\frac{\eta}{6}\pm O(\eta^2).
\end{align*}
Thus $\hat f(3.5)<13$ indeed holds for a suitably small constant $\eta$, completing the construction.
\end{proof}

 \footnotesize
 \bibliographystyle{alphaabbr}
\bibliography{shufflebib}
\end{document}